\documentclass[a4paper,11pt,reqno]{amsart}
\usepackage{amsmath,amssymb,url,stmaryrd}
\usepackage{amsthm}
\usepackage[T1]{fontenc}
\usepackage{lmodern}
\usepackage[top=30truemm,bottom=30truemm,left=20truemm,right=20truemm]{geometry}
\usepackage{tikz}
\usetikzlibrary{positioning}
\usetikzlibrary{arrows.meta}
\usetikzlibrary{calc}
\usepackage{color}
\usepackage{enumitem}
\usepackage{hyperref}

\newenvironment{enumerate2}{\begin{enumerate}[label=\textup{(\alph*)}]}
{\end{enumerate}}

\date{\today}

\newtheorem{Thm}{Theorem}[section]
\newtheorem{Cor}[Thm]{Corollary}
\newtheorem{Lem}[Thm]{Lemma}
\newtheorem{Prop}[Thm]{Proposition}
\newtheorem{Def-Prop}[Thm]{Definition-Proposition}

\theoremstyle{definition}
\newtheorem{Def}[Thm]{Definition}
\newtheorem{Rem}[Thm]{Remark}
\newtheorem{Ex}[Thm]{Example}
\newtheorem{Nota}[Thm]{Notation}

\newcommand{\calA}{\mathcal{A}}
\newcommand{\calC}{\mathcal{C}}
\newcommand{\calD}{\mathcal{D}}
\newcommand{\calF}{\mathcal{F}}
\newcommand{\calH}{\mathcal{H}}
\newcommand{\calT}{\mathcal{T}}

\newcommand{\calW}{\mathcal{W}}
\newcommand{\calX}{\mathcal{X}}

\newcommand{\ovcalF}{\overline{\calF}}
\newcommand{\ovcalT}{\overline{\calT}}

\newcommand{\sfD}{\mathsf{D}}
\newcommand{\sfK}{\mathsf{K}}

\newcommand{\rmb}{\mathrm{b}}
\newcommand{\rmf}{\mathrm{f}}
\newcommand{\rmt}{\mathrm{t}}
\newcommand{\rmw}{\mathrm{w}}

\newcommand{\ovrmf}{\overline{\rmf}}
\newcommand{\ovrmt}{\overline{\rmt}}

\newcommand{\Q}{\mathbb{Q}}
\newcommand{\R}{\mathbb{R}}
\newcommand{\Z}{\mathbb{Z}}

\newcommand{\frakC}{\mathfrak{C}}
\newcommand{\frakF}{\mathfrak{F}}

\DeclareMathOperator{\Hom}{\mathsf{Hom}}
\DeclareMathOperator{\End}{\mathsf{End}}

\DeclareMathOperator{\Ext}{\mathsf{Ext}}

\DeclareMathOperator{\rad}{\mathsf{rad}}
\DeclareMathOperator{\soc}{\mathsf{soc}}

\DeclareMathOperator{\Ker}{\mathsf{Ker}}
\DeclareMathOperator{\Image}{\mathsf{Im}} \renewcommand{\Im}{\Image}

\DeclareMathOperator{\module}{\mathsf{mod}} \renewcommand{\mod}{\module}
\DeclareMathOperator{\proj}{\mathsf{proj}}
\DeclareMathOperator{\inj}{\mathsf{inj}}

\DeclareMathOperator{\add}{\mathsf{add}}

\DeclareMathOperator{\Filt}{\mathsf{Filt}}
\DeclareMathOperator{\Fac}{\mathsf{Fac}}
\DeclareMathOperator{\Sub}{\mathsf{Sub}}
\DeclareMathOperator{\thick}{\mathsf{thick}}
\DeclareMathOperator{\simple}{\mathsf{sim}}

\DeclareMathOperator{\tors}{\mathsf{tors}}
\DeclareMathOperator{\ftors}{\mathsf{f-tors}}
\DeclareMathOperator{\torf}{\mathsf{torf}}
\DeclareMathOperator{\ftorf}{\mathsf{f-torf}}

\DeclareMathOperator{\smc}{\mathsf{smc}}
\DeclareMathOperator{\twosmc}{2\mathsf{-smc}}
\DeclareMathOperator{\silt}{\mathsf{silt}}
\DeclareMathOperator{\twosilt}{2\mathsf{-silt}}

\DeclareMathOperator{\twopsilt}{2\mathsf{-psilt}}

\DeclareMathOperator{\sfT}{\mathsf{T}}
\DeclareMathOperator{\sfF}{\mathsf{F}}

\DeclareMathOperator{\SH}{\mathsf{SH}}
\DeclareMathOperator{\reduc}{\mathsf{red}}

\DeclareMathOperator{\supp}{\mathsf{supp}}
\DeclareMathOperator{\fac}{\mathsf{fac}}
\DeclareMathOperator{\sub}{\mathsf{sub}}

\newcommand{\Cone}{\mathsf{Cone}}

\newcommand{\TF}{\mathsf{TF}}

\DeclareMathOperator{\Facet}{\mathsf{Facet}}
\DeclareMathOperator{\Face}{\mathsf{Face}}

\newcommand{\smallone}{\begin{smallmatrix}1\end{smallmatrix}}
\newcommand{\smalltwo}{\begin{smallmatrix}2\end{smallmatrix}}

\renewcommand{\epsilon}{\varepsilon}
\renewcommand{\Gamma}{\varGamma}
\renewcommand{\Theta}{\varTheta}
\renewcommand{\Phi}{\varPhi}
\renewcommand{\Psi}{\varPsi}
\renewcommand{\Sigma}{\varSigma}
\renewcommand{\ell}{l}

\DeclareMathOperator*{\sfmax}{\mathsf{max}} \renewcommand{\max}{\sfmax}
\DeclareMathOperator*{\sfmin}{\mathsf{min}} \renewcommand{\min}{\sfmin}
\DeclareMathOperator{\sfdim}{\mathsf{dim}} \renewcommand{\dim}{\sfdim}

\numberwithin{equation}{section}

\begin{document}
\title
{The interval neighborhoods in the real Grothendieck groups}

\author{Sota Asai} 
\address{Sota Asai: Graduate School of Mathematical Sciences,
University of Tokyo,  
3-8-1 Komaba, Meguro-ku, Tokyo-to, 153-8914, Japan}
\email{sotaasai@g.ecc.u-tokyo.ac.jp}

\begin{abstract}
For a finite dimensional algebra $A$,
the TF equivalence on the real Grothendieck group $K_0(\proj A)_\R$
can be regarded as a completion of the $g$-fan.
For example, the silting cones $C^\circ(U)$ of 
2-term presilting complexes $U$ give 
the most fundamental family of TF equivalence classes.
The next step is studying the TF equivalence classes
around each silting cone $C^\circ(U)$.
Thus, in this paper, 
we investigate the closed interval neighborhood $D(U)$ of $C^\circ(U)$.
As our main result, we give a $2^{|U|}:1$ correspondence between 
the TF equivalence classes in $D(U)$ and those in $K_0(\proj B)_\R$, 
where $B$ is the algebra appearing 
in the $\tau$-tilting reduction at $U$.
For this purpose, we give an explicit description of 
defining inequalities and the faces of $D(U)$ as a polyhedral cone, 
by using 2-term simple-minded collections and $M$-TF equivalences.
\end{abstract}

\maketitle 
\setcounter{tocdepth}{1}
\tableofcontents

\section{Introduction}

\subsection{Background: the $g$-fan and the TF equivalence}

\emph{Tilting complexes}
in the perfect derived category $\sfK^\rmb(\proj A)$ for a ring $A$
is a fundamental notion characterizing \emph{derived equivalences}
\cite{Rickard}.
In order to get categorical relationships between rings,
it is important to study many tilting complexes.
For this purpose, \emph{silting complexes} \cite{KV},
which generalize tilting complexes, play an important role.
Aihara-Iyama \cite{AI} established \emph{mutations} of silting complexes,
which allow us to construct silting complexes in a systematic way.

Within the range of \emph{2-term silting complexes}
for a finite dimensional algebra $A$ over a field $K$,
we have many more nice properties.
A 2-term silting complex $S=\bigoplus_{i=1}^n S_i$ with $S_i$ indecomposable
has a unique mutation at each $S_j$ \cite{AIR,DF}.
Thus mutations of 2-term silting complexes
are strongly related to \emph{cluster theory} \cite{FZ,KeY}.
Moreover, 2-term silting complexes have bijections
with other important notions 
in the module category $\mod A$ and 
its bounded derived category $\sfD(A):=\sfD^\rmb(\mod A)$;
see \cite{AIR,A-semi,BY,IT,KY,MS}.
There have been various types of studies of (2-term) silting complexes including recent ones \cite{ALS,Garcia,Gnedin,GZ,HI,IK}.

Mutations of 2-term silting complexes can be expressed
as a fan in the \emph{real Grothendieck group}
$K_0(\proj A)_\R$.
For each basic 2-term presilting complex $U=\bigoplus_{i=1}^m U_i$ 
with $U_i$ indecomposable,
the \emph{silting cones} $C^\circ(U),C(U) \subset K_0(\proj A)_\R$ are
defined by
\begin{align*}
C^\circ(U):=\sum_{i=1}^m \R_{>0}[U_i], \quad
C(U):=\sum_{i=1}^m \R_{\ge 0}[U_i].
\end{align*}
As in \cite{AHIKM}, the set of all closed silting cones $C(U)$
is a non-singular fan in $K_0(\proj A)_\R$, 
called the \emph{$g$-fan} of $A$.
In general, the $g$-fan is not necessarily complete.
In fact, the $g$-fan is complete if and only if
there are only finitely many isoclasses of basic 2-term silting complexes
\cite{A}.

In this context,
the \emph{TF equivalence} on $K_0(\proj A)_\R$ \cite{A} is 
a useful notion which can be regarded as a completion of the $g$-fan.
Each $\theta \in K_0(\proj A)_\R$ induces an $\R$-linear form
$\theta \colon K_0(\mod A)_\R \to \R$
via the \emph{Euler bilinear form},
and give two \emph{semistable torsion pairs}
$(\ovcalT_\theta,\calF_\theta)$ and $(\calT_\theta,\ovcalF_\theta)$ 
in $\mod A$ \cite{BKT}.
We say that $\theta,\eta \in K_0(\proj A)_\R$ are \emph{TF equivalent}
if $(\ovcalT_\theta,\calF_\theta)=(\ovcalT_\eta,\calF_\eta)$ and 
$(\calT_\theta,\ovcalF_\theta)=(\calT_\eta,\ovcalF_\eta)$.

Then each silting cone $C^\circ(U)$ is 
a typical TF equivalence class \cite{A,BST,Y}.
Namely, $\theta \in K_0(\proj A)_\R$ belongs to $C^\circ(U)$ if and only if
\begin{align*}
(\ovcalT_\theta,\calF_\theta)
&=({^\perp H^{-1}(\nu U)},\Sub H^{-1}(\nu U))
=:(\ovcalT_U,\calF_U), \\
(\calT_\theta,\ovcalF_\theta)
&=(\Fac H^0(U),{H^0(U)^\perp})
=:(\calT_U,\ovcalF_U).
\end{align*}
The association $U \mapsto C^\circ(U)$ gives an injection
from the set $\twopsilt A$ of basic 2-term presilting complexes
to the set of TF equivalence class in $K_0(\proj A)_\R$.
This map is restricted to a bijection
from the set $\twosilt A$ of basic 2-term silting complexes
to the set of full-dimensional TF equivalence classes in $K_0(\proj A)_\R$
\cite{A}.
Thus we can recover the $g$-fan from the TF equivalence.

These notions are strongly related also to
the \emph{wall-chamber structure} on $K_0(\proj A)_\R$
introduced by \cite{Bridgeland,BST}
by using the \emph{$\theta$-semistable subcategory}
$\calW_\theta:=\ovcalT_\theta \cap \ovcalF_\theta$ of King \cite{K}.
As examples of relevant studies, we refer to 
\cite{BCDMTY,BPPW,HIKT,Mizuno,PYK,STTV}.

\subsection{The main result on the interval neighborhood $D(U)$}
As the next step following silting cones,
in this paper,
we investigate the local behavior of the TF equivalence around $C^\circ(U)$.
For this purpose, we define the following subsets of $K_0(\proj A)_\R$ as 
the main subjects of our study.

\begin{Def}[Definition-Proposition \ref{Def-Prop_D(U)}]
For each $U \in \twopsilt A$,
the \emph{interval neighborhoods} $D^\circ(U),D(U)$
of $C^\circ(U)$ are defined by
\begin{align*}
D(U)&:=\{ \theta \in K_0(\proj A)_\R \mid H^0(U) \in \ovcalT_\theta, \ 
H^{-1}(\nu U) \in \ovcalF_\theta\}, \\
D^\circ(U)&:=\{ \theta \in K_0(\proj A)_\R \mid 
H^0(U) \in \calT_\theta, \ H^{-1}(\nu U) \in \calF_\theta\}\\
&=\{ \theta \in K_0(\proj A)_\R \mid 
\calT_U \subset \calT_\theta \subset \ovcalT_\theta \subset \ovcalT_U\}.
\end{align*}
\end{Def}

We remark that $D^\circ(U)$ appeared in \cite{A}.
To study the interval neighborhoods,
we use the \emph{$\tau$-tilting reduction} at $U$ introduced 
by Jasso \cite{Jasso}.
He constructed a certain algebra $B=B_U$ 
from the \emph{maximal completion} (the \emph{Bongartz completion}) of $U$,
and proved that the algebra $B$ controls
the interval $[\calT_U,\ovcalT_U]$ of torsion classes,
the subcategory $\calW_U:=\ovcalT_U \cap \ovcalF_U$ of $\mod A$,
and the set $\twopsilt_U A$
of basic 2-term presilting complexes containing $U$ as direct summands;
see Propositions \ref{Prop_reduc} and \ref{Prop_reduc_silt}.
This method motivated many works 
including \cite{AET,AP,BDH,CWZ,ES,Tattar}.

The aim of this paper is revealing the relationship
between the sets $\TF_A^{D(U)}$ and $\TF_B$ 
of TF equivalence classes in $D(U)$ and $K_0(\proj B)_\R$, respectively.
To reveal the relationship between $D(U)$ and the algebra $B$,
we use the canonical surjective $\R$-linear map 
$\pi \colon K_0(\proj A)_\R \to K_0(\proj B)_\R$
which is compatible with the $\tau$-tilting reduction at $U$.
Its kernel is the subspace $\R C(U)$ spanned by $C(U)$.
The following is the main result of this paper,
where $U_I:=\bigoplus_{i \in I} U_i$ for any $I \subset \{1,\ldots,m\}$.

\begin{Thm}[Theorems \ref{Thm_TF_2^m_B} and \ref{Thm_rho_explicit}]
\label{Thm_TF_2^m_B_intro}
Let $U=\bigoplus_{i=1}^m U_i \in \twopsilt A$ with $U_i$ indecomposable.
\begin{enumerate}
\item
We have a bijection
\begin{align*}
\TF_A^{D(U)} \to 2^{\{1,\ldots,m\}} \times \TF_B, \quad
E \mapsto (\{ i \in \{1,\ldots,m\} \mid E \subset D^\circ(U_i) \},\pi(E)).
\end{align*}
\item
There exists a piecewise linear injection 
$\rho \colon K_0(\proj B)_\R \to D(U)$
such that $\rho$ is $\R$-linear on each $E \in \TF_B$ and 
that the inverse map of the bijection in (1) is 
\begin{align*}
2^{\{1,\ldots,m\}} \times \TF_B \to \TF_A^{D(U)}, \quad
(I,E) \mapsto C(U_I)+\rho(E).
\end{align*}
\end{enumerate}
\end{Thm}

Theorem \ref{Thm_rho_explicit} also gives
an explicit description of $\rho$ by using the maximal completion of $U$.
This extends the corresponding statement 
for $D^\circ(U)$ in \cite[Theorem 4.5]{A}.

To prove Theorem \ref{Thm_TF_2^m_B_intro}, 
we focus on the structure of
$D(U)$ as a rational polyhedral cone.
The main part of this paper is devoted to describing
the faces of $D(U)$ explicitly.

\subsection{Facets and defining inequalities of the cone $D(U)$}

As the first step, we consider the facets of $D(U)$.
For this purpose, we use \emph{semibricks}.
Write $U=\bigoplus_{i=1}^m U_i$ with each $U_i$ indecomposable.
Then we have two semibricks $Y^+=\bigoplus_{i=1}^m Y_i^+$ 
and $X^-=\bigoplus_{i=1}^m X_i^-$
whose torsion and torsion-free closures
are $\calT_U$ and $\calF_U$, respectively \cite{A-semi}.
Here, the direct summands $Y_i^+,X_i^-$ are bricks or zero. 
See Proposition \ref{Prop_rad_soc_U} for the construction.
Then we have the following equalities.

\begin{Prop}[Proposition \ref{Prop_D(U)_Y_X}]\label{DcU semibrick intro}
Let $U=\bigoplus_{i=1}^m U_i \in \twopsilt A$.
Then we have
\begin{align*}
D(U)&=\{ \theta \in K_0(\proj A)_\R \mid 
Y^+ \in \ovcalT_\theta, \ X^- \in \ovcalF_\theta \}
=\bigcap_{i=1}^m \{ \theta \in K_0(\proj A)_\R \mid 
Y_i^+ \in \ovcalT_\theta, \ X_i^- \in \ovcalF_\theta \},\\
D^\circ(U)&=\{ \theta \in K_0(\proj A)_\R \mid 
Y^+ \in \calT_\theta, \ X^- \in \calF_\theta \}
=\bigcap_{i=1}^m \{ \theta \in K_0(\proj A)_\R \mid 
Y_i^+ \in \calT_\theta, \ X_i^- \in \calF_\theta \}.
\end{align*}
\end{Prop}

We mainly use these descriptions to study the set $\Facet D(U)$ 
of facets of $D(U)$. 
For each $i \in \{1,\ldots,m\}$ and $\epsilon \in \{\pm\}$,
we define the subsets $\partial_i^\epsilon$ of the boundary
$\partial D(U)$ by
\begin{align*}
\partial_i^+:=\{ \theta \in D(U) \mid Y_i^+ \notin \calT_\theta\}, \quad
\partial_i^-:=\{ \theta \in D(U) \mid X_i^- \notin \calF_\theta\}, \quad
\partial_i:=\partial_i^+\cup \partial_i^-.
\end{align*}
Moreover we set
\begin{align*}
\Facet_i^\epsilon D(U)
:=\{F \in \Facet D(U) \mid F \subset \partial_i^\epsilon \},\quad
\Facet_i D(U)
:=\Facet_i^+ D(U) \cup \Facet_i^- D(U).
\end{align*}
We show the following decomposition theorem.

\begin{Thm}[Theorem \ref{Thm_partial_facet}]\label{union facet intro}
Let $U=\bigoplus_{i=1}^m U_i \in \twopsilt A$. 
For each facet $F$ of $D(U)$,
there uniquely exists $(i,\epsilon) \in \{1,\ldots,m\} \times \{\pm\}$
such that $F \in \Facet_i^\epsilon D(U)$.
Thus we have a decomposition
\begin{align*}
\Facet D(U)
=\bigsqcup_{i=1}^m(\Facet_i^+ D(U) \sqcup \Facet_i^- D(U)).
\end{align*}
\end{Thm}

Therefore for each $F \in \Facet D(U)$,
we set $(i_F,\epsilon_F)$ as the unique pair $(i,\epsilon)$
such that $F \in \Facet_i^\epsilon D(U)$.
Then we define a nonzero module $L_F$ by 
\begin{align*}
L_F:=\begin{cases}
\rmw_\theta Y_{i_F}^+ & (\epsilon_F={+}) \\
\rmw_\theta X_{i_F}^- & (\epsilon_F={-})
\end{cases},
\end{align*} 
where $\theta$ is an arbitrary element of $F^\circ$;
see Definition \ref{Def_L_F} for detail.
The modules $L_F$ are bricks,
and $L_F$ give defining inequalities of $D(U)$ as follows.

\begin{Thm}[Theorem \ref{Thm_L_F_inn_vec} and Corollary \ref{Cor_D(U)_ineq}]
\label{facet inner intro}
Let $U=\bigoplus_{i=1}^m U_i \in \twopsilt A$.
Then each $F \in \Facet D(U)$
satisfies the following properties.
\begin{enumerate}
\item
The module $L_F$ is a simple object of $\calW_\theta$ 
for any $\theta \in F^\circ$.
\item
The element $\epsilon_F[L_F] \in K_0(\mod A)$ is 
an inner normal vector of $F$ such that
\begin{align*}
F=\{ \theta \in D(U) \mid \theta(L_F)=0\}.
\end{align*}
\end{enumerate}
Therefore we have
\begin{align*}
D(U)=\bigcap_{F \in \Facet D(U)}
\{ \theta \in K_0(\proj A)_\R \mid \epsilon_F \cdot\theta(L_F) \ge 0 \}.
\end{align*}
\end{Thm}

This labeling of the facets is consistent with the brick labeling
\cite{A-semi} of the \emph{exchange quiver} of 2-term silting complexes; 
see Theorems \ref{Thm_L_F_exch} and \ref{Thm_C(V/Ui)_Z_=}.

\subsection{Faces of $D(U)$ and $M$-TF equivalences on $K_0(\proj B)_\R$}

Next we consider arbitrary faces of $D(U)$.
Our main tools are 
the canonical maps $\lambda_U,\lambda'_U \colon D(U) \to D(U)$
satisfying the properties below.

\begin{Prop}[Propositions \ref{Prop_lambda_coeff} and \ref{Prop_affine}]
Let $U \in \twopsilt A$. 
Then there exist unique piecewise linear maps 
$\lambda_U,\lambda'_U \colon D(U) \to D(U)$ such that any $\theta \in D(U)$
satisfies $\lambda_U(\theta)+\lambda'_U(\theta)=\theta$ and 
\begin{align*}
\{\eta \in \R C(U) \mid \theta-\eta \in D(U)\}
=\lambda_U(\theta)+(-C(U)), \quad
(\theta+\R C(U)) \cap D(U)=\lambda'_U(\theta)+C(U).
\end{align*}
\end{Prop}

Moreover these maps enjoy the axiom of projections
$\lambda_U \circ \lambda_U=\lambda_U$, 
$\lambda_U \circ \lambda'_U=0$, 
$\lambda'_U \circ \lambda_U=0$ and 
$\lambda'_U \circ \lambda'_U=\lambda'_U$.
The image of $\lambda_U$ is $C(U)$, while
the image of $\lambda'_U$ is the \emph{link} $L(U)$ of $C(U)$ given by
\begin{align*}
L(U):=D(U) \setminus \left( \bigcup_{i=1}^m D^\circ(U_i) \right)
=\bigcap_{i=1}^m \partial_i.
\end{align*}
For example, $L(U)=\partial D(U)$ holds if $U$ is indecomposable,
and $L(U)=C(V_1) \cup C(V_2)$ holds if $m=|A|-1$,
where $U \oplus V_1$ and $U \oplus V_2$ are the two elements 
in $\twosilt_U A$.
We show that $D(U)$ is decomposed to $C(U)$ and $L(U)$ in the following sense.

\begin{Thm}[Theorem \ref{Thm_C(U)_L(U)}]\label{max F bijection intro}
Let $U \in \twopsilt A$.
Then we have a bijection
\begin{align*}
C(U) \times L(U) &\to D(U), \quad 
(\eta,\eta') \mapsto \eta+\eta'
\end{align*}
whose inverse is given by 
$\theta \mapsto (\lambda_U(\theta),\lambda'_U(\theta))$.
\end{Thm}

By using this, we show in Theorem \ref{Thm_pi_L(U)} that 
the surjective linear map $\pi \colon K_0(\proj A)_\R \to K_0(\proj B)_\R$ 
is restricted to a bijection
\begin{align*}
\pi|_{L(U)} \colon L(U) \to K_0(\proj B)_\R.
\end{align*}
Its inverse map $\rho \colon K_0(\proj B)_\R \to L(U)$ 
is nothing but the piecewise linear map $\rho$
in Theorem \ref{Thm_TF_2^m_B_intro}.

Based on these results, 
we construct finite complete generalized fans in $K_0(\proj B)_\R$
from certain subsets of the set $\Face D(U)$ of faces.
For each $I \subset \{1,\ldots,m\}$, we define
\begin{align*}
\Face^\circ_I D(U):=\{ F \in \Face D(U) \mid
\text{$F \subset \partial_i$ if and only if $i \in I$}\}.
\end{align*}
Sending these faces to $K_0(\proj B)_\R$ by $\pi$,
we obtain a finite complete generalized fan.

\begin{Thm}[Theorem \ref{Thm_Sigma_I}]
Let $U=\bigoplus_{i=1}^m U_i \in \twopsilt A$.  
For each subset $I \subset \{1,\ldots,m\}$, 
the set $\Sigma_I$ is a finite complete generalized fan in $K_0(\proj B)_\R$,
where
\begin{align*}
\Sigma_I:=\{\pi(F) \mid F \in \Face^\circ_I D(U)\}.
\end{align*}
\end{Thm}

To show Theorem \ref{Thm_TF_2^m_B_intro},
it is crucial to show that $\Sigma_I$ is coarser
than the TF equivalence on $K_0(\proj B)_\R$.
For each module $M \in \mod B$, 
we have the \emph{$M$-TF equivalence} \cite{AsI2} on $K_0(\proj B)_\R$
as a coarsening of the TF equivalence,
and this gives a finite complete generalized fan $\Sigma(M)$ 
in $K_0(\proj B)_\R$ 
consisting of the closures of $M$-TF equivalence classes;
see Definition \ref{Def_M-TF} and Proposition \ref{Prop_M-TF_fan}.
Thus it suffices to show the existence of $M_I \in \mod B$
such that $\Sigma_I=\Sigma(M_I)$.

\begin{Thm}[Theorem \ref{Thm_Sigma(M_I)}]\label{Sigma coincide intro}
Let $U=\bigoplus_{i=1}^m U_i \in \twopsilt A$.
Then there exist $M_1,\ldots,M_m \in \mod B$ such that,
for any subset $I \subset \{1,\ldots,m\}$,
we have
\begin{align*}
\Sigma_I=\Sigma(M_I), \ 
\text{where} \ 
M_I:=\bigoplus_{i \in I}M_i.
\end{align*}
Moreover there exist mutually inverse bijections
\begin{align*}
\Face^\circ_I D(U) &\to \Sigma(M_I), &
F &\mapsto \pi(F);\\
\Sigma(M_I) &\to \Face^\circ_I D(U), &
\sigma &\mapsto \pi^{-1}(\sigma) \cap \partial_I=C(U/U_I)+\rho(\sigma),
\end{align*}
which preserve inclusions and intersections.
\end{Thm}

The modules $M_i$ are constructed
from the 2-term simple-minded collections $X,Y \in \twosmc A$
corresponding the maximal and minimal completions 
$S,T \in \twosilt A$ of $U$.
The indecomposable direct summands $X_i,Y_i$ corresponding to $U_i$
have the triangle $X_i[-1] \to W_i \to Y_i \to X_i$ with $W_i \in \calW_U$ 
as in Notation \ref{Nota_X_Y} and Proposition \ref{Prop_sim_W_U}.
Then we define $M_i:=\Phi(W_i)$ 
by the equivalence $\Phi \colon \calW_U \to \mod B$
in the $\tau$-tilting reduction.
We remark that $Y_i^+,X_i^-$ in Proposition \ref{DcU semibrick intro}
coincide with $H^0(Y_i),H^{-1}(X_i)$ here, respectively.

\subsection{The organization}

In Section \ref{Sec_preliminary},
we recall basic terminology and results 
on torsion pairs, 2-term silting complexes,
2-term simple-minded collections and $\tau$-tilting reduction.
Section \ref{Sec_pre_Grothendieck} is devoted to 
preparation of notions on the real Grothendieck group $K_0(\proj A)_\R$,
including semistable torsion pairs, the TF equivalence,
$M$-TF equivalences, the $g$-fan and interval neighborhoods.
In Section \ref{Sec_facet},
we mainly deal with properties on facets of $D(U)$
such as the purities of $\partial_i^\epsilon$
and the brick labeling $L_F$ of facets.
Then all faces are studied in Section \ref{Sec_face},
where we obtain a finite complete generalized fan
$\Sigma_I=\Sigma(M_I)$ in $K_0(\proj B)_\R$
for each $I \subset \{1,\ldots,m\}$.
In Section \ref{Sec_main_result},
we conclude our proof of the main result Theorem \ref{Thm_TF_2^m_B_intro}.

\subsection*{Acknowledgment}
The author thanks Osamu Iyama for giving useful advices.
This work was supported by JSPS KAKENHI Grant Numbers 
JP19K14500, JP20J00088 and JP23K12957. 

\subsection*{Conventions}\label{Subsec_Nota}

In this paper, $K$ is an arbitrary field,
and $A$ is a finite dimensional algebra over $K$.

We write $\proj A$ for the category of finitely generated projective 
right $A$-modules, and
$\mod A$ for the category of finitely generated right $A$-modules.
Moreover $\sfK^\rmb(\proj A)$ denotes
the homotopy category of bounded complexes in $\proj A$,
and $\sfD(A):=\sfD^\rmb(\mod A)$ denotes
the derived category of bounded complexes in $\mod A$.
All subcategories are assumed to be full subcategories
and closed under isomorphism classes.

Let $\calD$ be $\mod A$, $\sfD(A)$ or $\sfK^\rmb(\proj A)$.
Then $\calD$ is Krull-Schmidt.
Thus each $M \in \calD$ is uniquely written as
$M \simeq \bigoplus_{i=1}^m M_i^{\oplus s_i}$
with each $M_i$ indecomposable, $s_i \in \Z_{\ge 1}$ 
and $M_i \not \simeq M_j$ if $i \ne j$.
In this case, we set $|M|:=m$,
so $M$ is the number of isoclasses of
indecomposable direct summands of $M$.
Moreover $M$ is said to be \emph{basic} if $s_i=1$ for all $i$.

Unless otherwise stated, we set $n:=|A|$,
which is called the \emph{rank} of the algebra $A$.
We write $P(1),\ldots,P(n)$ 
for the pairwise nonisomorphic indecomposable projective $A$-modules.
The symbol $L(i)$ denotes the simple top of $P(i)$.
Thus $L(1),\ldots,L(n)$ are the pairwise nonisomorphic simple $A$-modules.

Let $\calD$ as above, 
and $\calC \subset \calD$ be a subcategory.
For each $X \in \calD$, 
a morphism $f \colon X \to C$ with $C \in \calC$
is called a \emph{left $\calC$-approximation}
if any $g \colon X \to C'$ with $C' \in \calC$ 
admits some $h \colon C \to C'$ such that $g=hf$.
A left $\calC$-approximation $f \colon X \to C$ is called
a \emph{minimal} left $\calC$-approximation if
any endomorphism $h \colon C \to C$ satisfying $f=hf$ is an automorphism.
If any $X \in \calD$ has a left $\calC$-approximation,
then $\calC$ is said to be \emph{covariantly finite} in $\calD$.
Dually \emph{(minimal) right $\calC$-approximations}
and \emph{contravariantly finite} subcategories of $\calD$ are defined,
and $\calC$ is said to be \emph{functorially finite} in $\calD$
if $\calC$ is both covariantly finite and contravariantly finite in $\calD$.

\section{Preliminary on silting theory}
\label{Sec_preliminary}

\subsection{Torsion pairs}

In this subsection, we recall the definition of torsion pairs.
For a subcategory $\calC \subset \mod A$,
we define subcategories
$\calC^\perp:=\{ X \in \mod A \mid \Hom_A(\calC,X)=0\}$ and 
${^\perp \calC}:=\{ X \in \mod A \mid \Hom_A(X,\calC)=0\}$.

Let $\calT,\calF \subset \mod A$ be subcategories.
Then $(\calT,\calF)$ is called a \emph{torsion pair} in $\mod A$
if $\calF=\calT^\perp$ and $\calT={^\perp \calF}$.
In this case, each $M \in \mod A$ has a short exact sequence
\begin{align*}
0 \to M_1 \to M \to M_2 \to 0 \quad (M_1 \in \calT, \ M_2 \in \calF),
\end{align*}
which is unique up to isomorphisms.
This exact sequence is called a \emph{canonical sequence} of $M$
with respect to $(\calT,\calF)$.

A full subcategory $\calT$ is called a \emph{torsion class} in $\mod A$
if there exists some $\calF$ such that
$(\calT,\calF)$ is a torsion pair in $\mod A$.
It is well-known that $\calT$ is a torsion class in $\mod A$ if and only if 
$\calT$ is closed under taking factor modules and extensions in $\mod A$.
Dually, $\calF$ is called a \emph{torsion-free class} in $\mod A$ 
if there exists some $\calT$ such that
$(\calT,\calF)$ is a torsion pair in $\mod A$;
or equivalently, $\calF$ is 
closed under taking submodules and extensions in $\mod A$.
We write $\tors A$ (resp.~$\torf A$) 
for the set of torsion (resp.~torsion-free) classes in $\mod A$.

\subsection{2-term presilting complexes}

Next we recall 2-term (pre)silting complexes \cite[5.1]{KV}.

We say that $U$ is \emph{presilting} if
$\Hom_{\sfK^\rmb(\proj A)}(U,U[\ell])=0$ for all $\ell \in \Z_{\ge 1}$,
and that $U$ is \emph{silting} if
$U$ is presilting and the smallest thick subcategory $\thick U$
is $\sfK^\rmb(\proj A)$.
Here, a \emph{thick} subcategory means 
a triangulated subcategory which is closed under taking direct summands.

Moreover $U$ is said to be \emph{2-term} if
$U$ is isomorphic to some complex $U^{-1} \to U^0$
whose terms except the $-1$st and the $0$th ones are zero.
We write $\twopsilt A$ (resp.~$\twosilt A$)
for the set of isoclasses of
basic 2-term presilting (resp.~silting) complexes in $\sfK^\rmb(\proj A)$.
Unless otherwise stated, 
we assume that each $U_i$ is indecomposable
if we write $U=\bigoplus_{i=1}^m U_i \in \twopsilt A$.

The following property is fundamental.
Here, $\add X$ means the additive closure of each $X$ in $\sfK^\rmb(\proj A)$.

\begin{Prop}\label{Prop_U_S}\cite[Proposition 2.16]{Ai}
For each $U \in \twopsilt A$,
there exists $S \in \twosilt A$ such that $U \in \add S$.
\end{Prop}

Moreover, $U \in \twosilt A$ if and only if 
$U \in \twopsilt A$ and $|U|=|A|$
\cite[Proposition 3.3]{AIR}.

We next explain that
2-term presilting complexes give functorially finite torsion pairs.
Here, a torsion pair $(\calT,\calF)$ is said to be \emph{functorially finite}
if both $\calT$ and $\calF$ are functorially finite in $\mod A$.
We use the Nakayama functor 
$\nu \colon \sfK^\rmb(\proj A) \to \sfK^\rmb(\inj A)$,
where $\inj A$ is the category of finitely generated injective $A$-modules.

\begin{Def-Prop}\cite[Theorem 5.10]{AS}
Let $U \in \twopsilt A$.
Then we define two functorially finite torsion pairs 
\begin{align*}
(\ovcalT_U,\calF_U):=({^\perp H^{-1}(\nu U)},\Sub H^{-1}(\nu U)), \quad
(\calT_U,\ovcalF_U):=(\Fac H^0(U),{H^0(U)^\perp}).
\end{align*}
\end{Def-Prop}

By \cite[Theorem 2.12]{AIR},
$(\ovcalT_U,\calF_U)=(\calT_U,\ovcalF_U)$ holds if and only if 
$U \in \twosilt A$.
Thus, if $U \in \twopsilt A$ is not silting,
then $(\ovcalT_U,\calF_U)$ and $(\calT_U,\ovcalF_U)$ never coincide.
Their difference is described 
by the subcategory $\calW_U:=\ovcalT_U \cap \ovcalF_U$.
As in Proposition \ref{Prop_reduc} (1),
$\calW_U$ is a wide subcategory of $\mod A$, 
and will play a very important role in this paper.

The result below on functorially finite torsion pairs is fundamental.
We write $\ftors A$ (resp.~$\ftorf A$)
for the set of functorially finite torsion (resp.~torsion-free) classes
in $\mod A$.

\begin{Prop}\label{Prop_silt_tors}
\cite[Theorems 2.7, 3.2]{AIR}
There exist bijections
$\twosilt A \to \ftors A$ and $\twosilt A \to \ftorf A$
given by $S \mapsto \ovcalT_S=\calT_S$ and $S \mapsto \calF_S=\ovcalF_S$.
\end{Prop}

By using these torsion pairs, 
we can characterize when direct sums of 2-term presilting complexes
are again 2-term presilting.

\begin{Lem}\label{Lem_T_U_T_V}
\cite[Lemma 3.13]{A} \cite[Lemma 2.3, Proposition 3.11]{AsI1}
Let $U,V \in \twopsilt A$.
\begin{enumerate}
\item
The condition $U \in \add V$ holds 
if and only if $\calT_U \subset \calT_V$ and $\calF_U \subset \calF_V$.
\item
The direct sum $U \oplus V$ is (not necessarily basic) presilting
if and only if $\calT_U \subset \ovcalT_V$ 
and $\calF_U \subset \ovcalF_V$.
\end{enumerate}
\end{Lem}

\subsection{2-term simple-minded collections and semibricks}

This subsection is devoted to explaining 
2-term simple-minded collections in $\sfD(A)$ and semibricks in $\mod A$,
both of which are generalizations of semisimple modules.

We first give the definition of 2-term simple-minded collections
introduced by \cite{Al-Nofayee}.
Let $X=\bigoplus_{i=1}^m X_i \in \sfD(A)$ with $X_i$ indecomposable.
Then $X$ is called a \emph{simple-minded collection} 
in $\sfD(A)$ if the following conditions hold.
\begin{enumerate2}
\item
For any $i \in \{1,\ldots,m\}$, 
the endomorphism algebra $\End_{\sfD(A)}(X_i)$ is a division algebra.
\item
For any $i \ne j \in \{1,\ldots,m\}$, 
we have $\Hom_{\sfD(A)}(X_i,X_j)=0$.
\item
For any $i,j \in \{1,\ldots,m\}$ and $\ell \in \Z_{\ge 1}$, 
we have $\Hom_{\sfD(A)}(X_i,X_j[-\ell])=0$.
\item
The smallest thick subcategory of $\sfD(A)$ containing $X$
is $\sfD(A)$.
\end{enumerate2}
A simple-minded collection $X$ in $\sfD(A)$ is
said to be \emph{2-term} 
if $H^i(X)=0$ for all $i \in \Z \setminus \{-1,0\}$.
We write $\twosmc A$ for the set of all 2-term simple-minded collections
in $\sfD(A)$.

By \cite[Corollary 5.5]{KY}, 
any 2-term simple-minded collection $X$ satisfies $|X|=|A|$.
The following bijection with 2-term silting complexes is crucial,
where $\simple \calA$ denotes 
the set of isoclasses of an abelian category $\calA$.

\begin{Prop}\label{Prop_silt_smc}
\cite[Theorem 6.1]{KY} \cite[Corollary 4.3]{BY}
There exists a bijection $\twosilt A \to \twosmc A$ given by
$S \mapsto \simple \calH_S$, 
where $\calH_S \subset \sfD(A)$ is an abelian length category
\begin{align*}
\calH_S:=\{ X \in \sfD(A) \mid 
\text{for all $\ell \ne 0$, $\Hom_{\sfD(A)}(S,X[\ell])=0$} \}.
\end{align*}
\end{Prop}

As in \cite[Lemma 5.3]{KY},
the abelian category $\calH_S$ above is the heart 
of a certain intermediate t-structure \cite{BBD} in $\sfD^\rmb(\mod A)$.

If $S \in \twosilt A$ corresponds to $X \in \twosmc A$,
then we can make the indices of the indecomposable direct summands
of $S$ and $X$ compatible as follows.

\begin{Def-Prop}\label{Def-Prop_SH}
\cite[Lemma 5.3]{KY}
Let $S=\bigoplus_{i=1}^n S_i \in \twosilt A$.
Take $X \in \twosmc A$ corresponding to $S$ under the bijection in 
Proposition \ref{Prop_silt_smc}.
Then we have a unique decomposition $X=\bigoplus_{i=1}^n X_i$ such that
\begin{align*}
\Hom_{\sfD(A)}(S_{i'},X_i[\ell]) \simeq
\begin{cases}
R_i \quad \text{as left $R_j$-modules} & ((i,\ell)=(i',0)) \\
0 & ((i,\ell) \ne (i',0))
\end{cases},
\end{align*}
where $R_i$ is the division algebra $\End_{\sfD(A)}(X_i)$.
We set
\begin{align*}
\SH(S):=X=\bigoplus_{i=1}^n X_i,
\end{align*} 
including the indices of the indecomposable direct summands.
\end{Def-Prop}

For example, if $A$ is basic,
then $A \in \twosilt A$ satisfies $\calH_S=\mod A$.
For the decomposition $A=\bigoplus_{i=1}^n P(i)$,
we get $\SH(A)=\bigoplus_{i=1}^n L(i)$.

As a consequence of \cite[Theorem 4.8]{IY}, 
$\calH_S$ is equivalent to the module category $\mod \End_{\sfD(A)}(S)$.
Thus we have the following.

\begin{Lem}\label{Lem_R_V}
Let $S=\bigoplus_{i=1}^n S_i \in \twosilt A$,
and $X=\bigoplus_{i=1}^n X_i:=\SH(S) \in \twosmc A$.
Then there exists an isomorphism 
\begin{align*}
\End_{\sfD(A)}(S_i)/{\rad\End_{\sfD(A)}(S_i)} \to \End_{\sfD(A)}(X_i).
\end{align*}
\end{Lem}

We will use some results on (semi)bricks defined as follows.
A module $M \in \mod A$ is called a \emph{brick} 
if $\End_A(M)$ is a division algebra,
and $M$ is called a \emph{semibrick}
if $M$ admits a decomposition $M=\bigoplus_{i=1}^m M_i$ 
such that each $M_i$ is a brick and $\Hom_A(M_i,M_j)=0$ for any $i \ne j$. 
In particular, a semibrick is basic.

Semibricks are obtained from 2-term simple-minded collections as follows.

\begin{Def-Prop}\label{Def-Prop_stalk}
\cite[Remark 4.11]{BY}
Let $X=\bigoplus_{i=1}^n X_i \in \twosmc A$.
Then for any $i \in \{1,\ldots,n\}$, 
we have $X_i \in \mod A$ or $X_i \in (\mod A)[1]$.
Thus $X=H^0(X) \oplus H^{-1}(X)[-1]$ holds.
We define the following symbols.
\begin{enumerate}
\item
We set $X^+:=H^0(X)$ and $X^-:=H^{-1}(X)$,
which are semibricks.
\item
For each $i \in \{1,\ldots,n\}$, 
we set $X_i^+:=H^0(X_i)$ and $X_i^-:=H^{-1}(X_i)$. 
Then one of them is a brick, and the other is zero.
\end{enumerate}
\end{Def-Prop}

As in \cite[Theorem 3.3]{A-semi}, $X \in \twosmc A$ 
is determined by either one of the semibricks $X^+$ and $X^-$.
These semibricks satisfy the following properties,
where $\sfT(M)$ and $\sfF(M)$ are 
the torsion and torsion-free closures of $M$, respectively.

\begin{Prop}\label{Prop_rad_soc}
Let $S=\bigoplus_{i=1}^n S_i \in \twosilt A$,
and $X=\bigoplus_{i=1}^n X_i:=\SH(S) \in \twosmc A$.
\begin{enumerate}
\item
\cite[Theorem 3.3]{A-semi}
The torsion pair $(\calT_S,\calF_S)$ coincides with $(\sfT(X^+),\sfF(X^-))$.
\item
\cite[Lemma 3.14]{A-semi}
For each $i \in \{1,\ldots,n\}$, we have
\begin{align*}
X_i^+=H^0(S_i)/\sum_{f \in \rad_A(H^0(S),H^0(S_i))} \Im f, \quad 
X_i^-=\bigcap_{f \in \rad_A(H^{-1}(\nu S_i),H^{-1}(\nu S))} \Ker f,
\end{align*}
and hence
\begin{align*}
X^+=H^0(S)/\rad_{\End_A(H^0(S))}H^0(S)
\quad \text{and} \quad
X^-=\soc_{\End_A(H^{-1}(\nu S))}H^{-1}(\nu S).
\end{align*}
\end{enumerate}
\end{Prop}

\subsection{The exchange quiver of 2-term silting complexes}

Let $S,T \in \twosilt A$ with $U$ their maximum common direct summand.
If $|U|=|A|-1$, then we say that $T$ is a \emph{mutation} of $S$ at 
the indecomposable direct summand $S/U$.
The following is the fundamental property on mutations in $\twosilt A$.

\begin{Def-Prop}\label{Def-Prop_adjacent}
\cite[Definition-Proposition 2.28]{AIR}
Let $S=\bigoplus_{i=1}^n S_i \in \twosilt A$
and $j \in \{1,\ldots,n\}$.
Then the following statements hold.
\begin{enumerate}
\item
There uniquely exists $T \in \twosilt A$ which is a mutation of $S$ at $S_j$.
\item
Let $X=\bigoplus_{i=1}^n X_i:=\SH(S)$.
Then $T$ in (1) satisfies exactly one of the following conditions.
\begin{enumerate2}
\item
We have $\calT_T=\calT_{S/S_j} \subsetneq \calT_S$, 
$\calF_T \supsetneq \calF_{S/S_j}=\calF_S$
and $X_j^+ \ne 0$.
In this case, we call $T$ the \emph{left mutation} of $S$ at $S_j$,
and set $\mu_j^-(S):=T$.
\item
We have $\calT_T \supsetneq \calT_{S/S_j}=\calT_S$, 
$\calF_T=\calF_{S/S_j} \subsetneq \calF_S$
and $X_j^- \ne 0$.
In this case, we call $T$ the \emph{right mutation} of $S$ at $S_j$,
and set $\mu_j^+(S):=T$.
\end{enumerate2}
\end{enumerate}
\end{Def-Prop}

In either case, $\calT_S$ and $\calT_T$ are adjacent
with respect to inclusions of torsion classes in $\mod A$
by \cite[Example 3.5]{DIJ}.
We also remark that left/right mutations in $\twosilt A$
are special cases of mutations of silting complexes in \cite{AI}.

It is natural to express mutations 
of 2-term silting complexes as a quiver with labels.

\begin{Def}\label{Def_exchange}
We define the following.
\begin{enumerate}
\item
The \emph{exchange quiver} $Q(\twosilt A)$ is set to be 
the unique quiver satisfying the conditions below.
\begin{enumerate2}
\item
The vertices set is $\twosilt A$.
\item
For $S,T \in \twosilt A$, there exists an arrow $S \to T$
if and only if $T$ is a left irreducible mutation of $S$.
\end{enumerate2}
\item
\cite[Definition 2.14, Theorem 3.12]{A-semi}
The \emph{brick labeling} of the exchange quiver of $\twosilt A$
is labeling each arrow $S \to T$ with the brick $X_j^+$
in Definition-Proposition \ref{Def-Prop_adjacent} (2)(a).
\end{enumerate}
\end{Def}

Therefore if $S \in \twosilt A$ and $X=\SH(S) \in \twosmc A$,
then the indecomposable direct summands of $X^+$ (resp.~$X^-$)
are the labels of all arrows starting at (resp.~ending at) $S$
in the exchange quiver.
In particular, for each $S \in \twosilt A$,
there are exactly $n=|A|$ arrows which start or end at $S$.

\begin{Ex}
If $A$ is the path algebra $K(1 \to 2)$,
then $Q(\twosilt A)$ is the following,
where $V:=(P(2) \to P(1)) \in \twopsilt A$.
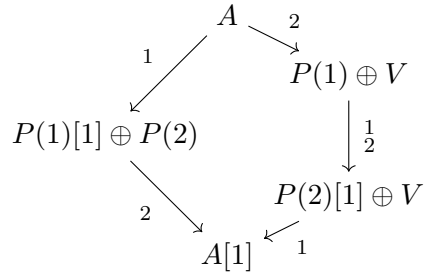
\begin{align*}
\begin{tikzpicture}[->,baseline=0pt,scale=0.8]
\node (1) at ( 0, 2) {$A$};
\node (2) at ( 2, 1) {$P(1) \oplus V$};
\node (3) at ( 2,-1) {$P(2)[1] \oplus V$};
\node (4) at (-2, 0) {$P(1)[1] \oplus P(2)$};
\node (5) at ( 0,-2) {$A[1]$};
\draw (1) to [edge label={$\begin{smallmatrix}2\end{smallmatrix}$}] (2);
\draw (2) to [edge label={$\begin{smallmatrix}1\\2\end{smallmatrix}$}] (3);
\draw (3) to [edge label={$\begin{smallmatrix}1\end{smallmatrix}$}] (5);
\draw (1) to [edge label'={$\begin{smallmatrix}1\end{smallmatrix}$}] (4);
\draw (4) to [edge label'={$\begin{smallmatrix}2\end{smallmatrix}$}] (5);
\end{tikzpicture}
\end{align*}
The brick labeling implies that 
the 2-term simple-minded collection
$\SH(P(1) \oplus V)$ is $\begin{smallmatrix}1\\2\end{smallmatrix} \oplus
\begin{smallmatrix}2\end{smallmatrix}[1]$.
\end{Ex}

As explained in \cite{KT},
the same brick labeling as ours is obtained 
by using results of \cite{BST}
on stability conditions in the real Grothendieck group $K_0(\proj A)_\R$, 
which we will consider in the next section.
We also remark that the brick labeling is extended to 
the Hasse quiver of all torsion classes in $\mod A$ \cite{BCZ,DIRRT}.

\subsection{Maximal and minimal completions} \label{Subsec_max_min}

By Proposition \ref{Prop_silt_tors}, for each $U \in \twopsilt A$,
there uniquely exist $S,T \in \twosilt A$ such that
\begin{align}\label{S T chara}
(\ovcalT_U,\calF_U)=(\ovcalT_S,\calF_S), \quad
(\calT_U,\ovcalF_U)=(\calT_T,\ovcalF_T).
\end{align}
By Lemma \ref{Lem_T_U_T_V} (1),
$S$ and $T$ have $U$ as direct summands, and
$V \in \twosilt A$ contains $U$ as a direct summand 
if and only if $\calT_U \subset \calT_V$ and $\calF_U \subset \calF_V$.
Among such $V$,
the torsion class $\calT_V$ becomes the largest if $V=S$,
and becomes the smallest if $V=T$.
Thus we define as follows.

\begin{Def}
For each $U \in \twopsilt A$, we call $S$ and $T$ in \eqref{S T chara}
the \emph{maximal completion} and 
the \emph{minimal completion} of $U$, respectively.
\end{Def}

We remark that $S$ and $T$ are also called 
the \emph{Bongartz completion} and 
the \emph{Bongartz cocompletion} of $U$.
The following property is useful and important.

\begin{Lem}
\cite[Lemma 3.13]{A}
Let $U \in \twopsilt A$,
and $S$ and $T$ be its maximal and minimal completions respectively.
Then $U$ is the maximum common direct summand of $S$ and $T$.
\end{Lem}

For each $U \in \twopsilt A$, the minimal completion $T$ is 
the left simultaneous mutation of the maximal completion $S$ at $S/U$
in the sense of \cite{AI}.
Thus \cite[Definition 2.30, Theorem 2.31]{AI} give the following notation
which we will frequently use in this paper.

\begin{Nota}\label{Nota_S_T}
Let $U=\bigoplus_{i=1}^m U_i \in \twopsilt A$,
and $S$ and $T$ be its maximal and minimal completions respectively.
Then we decompose $S$ and $T$ into indecomposable direct summands as 
\begin{align*}
S=\bigoplus_{i=1}^nS_i \quad \text{and} \quad T=\bigoplus_{i=1}^nT_i
\end{align*}
so that they satisfy the following conditions.
\begin{enumerate}
\item
For each $i \in \{1,\ldots,m\}$, we have $S_i=U_i=T_i$.
\item
For each $j \in \{m+1,\ldots,n\}$, there exists a triangle
\begin{align*}
S_j \xrightarrow{f} U'_j \to T_j \to S_j[1]
\end{align*}
such that $f$ is a minimal left $(\add U)$-approximation of $S_i$.
\end{enumerate}
\end{Nota}

Generalizing \cite[Theorem 7.12]{KY}, 
simultaneous mutations of a simple-minded collection 
at any direct summand can be defined
so that they are compatible with those of silting complexes.
We will be a detailed proof in \cite{A-smc};
see also \cite{BDH,BCPW}.
Applying this theory to 
the 2-term simple-minded collections $X,Y \in \twosmc A$
corresponding to $S,T \in \twosilt A$,
we get the following notation.
For any subcategory $\calC \subset \sfD(A)$,
the symbol $\Filt \calC$ denotes the filtration closure of $\calC$ 
in $\sfD(A)$.

\begin{Nota}\label{Nota_X_Y}
In Notation \ref{Nota_S_T}, we set
\begin{align*}
X=\bigoplus_{i=1}^n X_i:=\SH(S) \in \twosmc A, \quad
Y=\bigoplus_{i=1}^n Y_i:=\SH(T) \in \twosmc A.
\end{align*}
Setting $\calW:=\Filt \{X_{m+1},\ldots,X_n\} \subset \sfD(A)$,
the relationship between $X$ and $Y$ are given as follows. 
\begin{enumerate}
\item
If $i \in \{1,\ldots,m\}$, then there exists a  triangle 
\begin{align*}
X_i[-1] \xrightarrow{f_i} W_i \to Y_i \to X_i
\end{align*} 
in $\sfD(A)$ with $f_i$ 
a minimal left $\calW$-approximation.
\item
If $j \in \{m+1,\ldots,n\}$, 
then $X_j \in \mod A$ is a brick, and $Y_j=X_j[1] \in (\mod A)[1]$.
\end{enumerate}
\end{Nota}

Thus $\calW$ is contained in $\mod A$.
Actually it coincides with
$\calW_U:=\ovcalT_U \cap \ovcalF_U \subset \mod A$ appearing 
in the next subsection; see Proposition \ref{Prop_sim_W_U}.

By Definition-Proposition \ref{Def-Prop_SH} and Lemma \ref{Lem_R_V},
we have
\begin{align*}
\Hom_{\sfD(A)}(S_{i'},X_i[\ell]) &\simeq
\begin{cases}
R_i \quad \text{as left $R_i$-modules} & ((i,\ell)=(i',0)) \\
0 & ((i,\ell) \ne (i',0))
\end{cases} \simeq
\Hom_{\sfD(A)}(T_{i'},Y_i[\ell]),
\end{align*}
where $R_i:=\End_{\sfD(A)}(U_i)/\rad \End_{\sfD(A)}(U_i)$.

In the rest of this paper, 
once we write $U:=\bigoplus_{i=1}^m U_i \in \twopsilt A$,
we freely use Notations \ref{Nota_S_T} and \ref{Nota_X_Y}.

Recall that we defined 
\begin{align*}
X^+:=H^0(X), \quad X^-:=H^{-1}(X), \quad 
X_i^+:=H^0(X_i), \quad X_i^-:=H^{-1}(X_i)
\end{align*}
for each $X=\bigoplus_{i=1}^n X_i \in \twosmc A$ and $i \in \{1,\ldots,n\}$.
As in Definition-Proposition \ref{Def-Prop_stalk},
exactly one of $X_i^+$ and $X_i^-$ is a brick, and the other is zero.
Moreover $X^+$ and $X^-$ are semibricks.

The equalities \eqref{S T chara} and Proposition \ref{Prop_rad_soc} (1)
immediately give the following.

\begin{Prop}\label{Prop_T_U_T(Y^+)}
Let $U \in \twopsilt A$. 
Then we have 
\begin{align*}
\calT_U=\calT_T=\sfT(Y^+) \quad \text{and} \quad 
\calF_U=\calF_S=\sfF(X^-).
\end{align*}
\end{Prop}

By Notation \ref{Nota_X_Y},
the objects $Y_i$ and $X_i$ with the same indices satisfy
the next properties.

\begin{Prop}\label{Prop_X_i^-_Y_i^+}
Let $U=\bigoplus_{i=1}^m U_i \in \twopsilt A$.
\begin{enumerate}
\item
For each $i \in \{1,\ldots,m\}$, we have $Y_i^+ \ne 0$ or $X_i^- \ne 0$.
\item
For each $j \in \{m+1,\ldots,n\}$, we have $X_j^+=Y_j^- \ne 0$.
\end{enumerate}
\end{Prop}

\begin{proof}
(1)
Otherwise, $Y_i \in (\mod A)[1]$ and $X_i \in \mod A$ hold.
Then the last morphism of the triangle
$X_i[-1] \xrightarrow{f_i} W_i \to Y_i \to X_i$ is zero,
and we get $W_i \simeq X_i[-1] \oplus Y_i$.
This contradicts $W_i \in \Filt \{X_{m+1},\ldots,X_n\} \subset \mod A$.

(2) is clear.
\end{proof}

The modules $Y_i^+,X_i^-$ for $i \in \{1,\ldots,m\}$ 
appearing in (1) can be calculated 
by the formula in Proposition \ref{Prop_rad_soc} (2) from $S$ and $T$.
Combining with Proposition \ref{Prop_T_U_T(Y^+)},
these modules can be obtained directly from $U$.

\begin{Prop}\label{Prop_rad_soc_U}
Let $U=\bigoplus_{i=1}^m U_i \in \twopsilt A$.
For each $i \in \{1,\ldots,m\}$, we have
\begin{align*}
Y_i^+=H^0(U_i)/\sum_{f \in \rad_A(H^0(U),H^0(U_i))} \Im f, \quad 
X_i^-=\bigcap_{f \in \rad_A(H^{-1}(\nu U_i),H^{-1}(\nu U))} \Ker f,
\end{align*}
and hence the semibricks $Y^+$ and $X^-$ are
\begin{align*}
Y^+=H^0(U)/\rad_{\End_A(H^0(U))}H^0(U), \quad
X^-=\soc_{\End_A(H^{-1}(\nu U))}H^{-1}(\nu U).
\end{align*}
\end{Prop}

Note that this proposition does not provide information on 
$Y_i^-,X_i^+$ for $i \in \{1,\ldots,m\}$ or
$X_j^+=Y_j^-$ for $j \in \{m+1,\ldots,n\}$.

\subsection{The $\tau$-tilting reduction and the associated algebra $B$}
\label{subsec reduc}

In this subsection, we explain the useful method
\emph{$\tau$-tilting reduction} introduced by Jasso \cite{Jasso}.
We use the maximal/minimal completions $S$ and $T$ of $U$ also here.

One of the things which $\tau$-tilting reduction deals with 
is the interval 
\begin{align*}
[\calT_U,\ovcalT_U]=
\{\calT \in \tors A \mid \calT_U \subset \calT \subset \ovcalT_U\}
\end{align*}
in the poset $\tors A$ of torsion classes with respect to inclusions.
This interval coincides with $[\calT_T,\calT_S]=[\ovcalT_T,\ovcalT_S]$.
To describe this interval, 
Jasso considered the \emph{$\tau$-perpendicular category}
\begin{align*}
\calW_U:=\ovcalT_U \cap \ovcalF_U
\end{align*}
and the algebra
\begin{align}\label{def B}
B=B_U:=\End_A(H^0(S))/\langle e \rangle,
\end{align}
where $e$ is the idempotent endomorphism
$H^0(S) \to H^0(U) \to H^0(S)$.

These satisfy the following properties.
Here, we call a subcategory $\calC \subset \mod A$ is 
a \emph{wide subcategory} if $\calC$ is closed under 
taking kernels, cokernels and extensions in $\mod A$;
or equivalently, if $\calC$ is an abelian subcategory 
closed under extensions in $\mod A$.

\begin{Prop}\label{Prop_reduc}(cf.~\cite[Theorem 4.12]{DIRRT})
Let $U \in \twopsilt A$.
\begin{enumerate}
\item
\cite[Proposition 3.6]{Jasso}
The category $\calW_U$ is a wide subcategory of $\mod A$.
\item
\cite[Theorem 3.8]{Jasso}
We have an equivalence
\begin{align*}
\Phi:=\Hom_A(H^0(S),?) \colon \calW_U \to \mod B
\end{align*}
of abelian categories.
\item
\cite[Theorems 3.12, 3.14]{Jasso}
We have isomorphisms
\begin{align*}
[\calT_U,\ovcalT_U] \to \tors B, \quad
[\calT_U,\ovcalT_U] \cap \ftors A \to \ftors B
\end{align*}
of posets, both given by $\calT \mapsto \Phi(\calT \cap \calW_U)$.
\end{enumerate}
\end{Prop}

The property (1) means that $[\calT_U,\ovcalT_U]$
is a \emph{wide interval} in $\tors A$ in the sense of \cite{AP}.

Moreover, $\tau$-tilting reduction also deals with the sets
\begin{align*}
\twosilt_U A:=\{ V \in \twosilt A \mid U \in \add V \}, \quad
\twopsilt_U A:=\{ V \in \twopsilt A \mid U \in \add V\}
\end{align*}
of 2-term (pre)silting complexes in $\sfK^\rmb(\proj A)$
containing $U$ as a direct summand.
These sets are related to the interval $[\calT_U,\ovcalT_U]$,
because $V \in \twopsilt A$ satisfies $U \in \add V$
if and only if $[\calT_V,\ovcalT_V] \subset [\calT_U,\ovcalT_U]$
by Lemma \ref{Lem_T_U_T_V}.
Then we have the following properties.

\begin{Prop}\label{Prop_reduc_silt}
\cite[Theorems 3.16, 4.12]{Jasso}
\cite[Theorem 4.11]{AHIKM}
\cite[Theorem A.7]{BY}
Let $U \in \twopsilt A$.
There exists a triangle functor 
$\reduc \colon \sfK^\rmb(\proj A) \to \sfK^\rmb(\proj B)$
which satisfies the following properties.
\begin{enumerate2}
\item
We have $\reduc(S)=B$, $\reduc(U)=0$ and $\reduc(T)=B[1]$. 
\item
The functor $\reduc$ induces bijections
\begin{align*}
\reduc \colon \twosilt_U A \to \twosilt B, \quad 
\reduc \colon \twopsilt_U A \to \twopsilt B,
\end{align*}
preserving indecomposable direct summands which are not in $\add U$.
\item
The bijection $\reduc \colon \twosilt_U A \to \twosilt B$
induces an isomorphism $Q(\twosilt_U A) \to Q(\twosilt B)$
between the exchange quivers. 
\item
For any $V \in \twopsilt_U A$, we have
\begin{align*}
\Phi(\calT_V \cap \calW_U)=\calT_{\reduc(V)}, \quad
\Phi(\calF_V \cap \calW_U)=\calF_{\reduc(V)}.
\end{align*}
\end{enumerate2}
\end{Prop}

The isomorphism $Q(\twosilt_U A) \to Q(\twosilt B)$ in (c)
is compatible with the brick labeling of the exchange quivers 
in Definition \ref{Def_exchange}.

\begin{Prop}\label{Prop_reduc_label}
\cite[Theorems 2.21, 3.12]{A-semi}
Let $U \in \twopsilt A$.
If a brick $L$ is the label of an arrow $V_1 \to V_2$ 
in $Q(\twosilt_U A)$,
then we have $L \in \calW_U$, 
and $\Phi(L)$ is the label of the arrow $\reduc(V_1) \to \reduc(V_2)$ 
in $Q(\twosilt B)$.
\end{Prop}

Let $U=\bigoplus_{i=1}^m U_i \in \twopsilt A$.
Under Notations \ref{Nota_S_T} and \ref{Nota_X_Y},
$S_{m+1},\ldots,S_n$ are the indecomposable direct summands of $S/U$,
and $\reduc(S_{m+1}),\ldots,\reduc(S_n)$ are
the non-isomorphic indecomposable projective $B$-modules
by Proposition \ref{Prop_reduc_silt} (a)(b).
Corresponding to this, the indecomposable direct summands 
$X_{m+1},\ldots,X_n$ of $X=\SH(S)$ satisfy the properties below 
by Propositions \ref{Prop_reduc_silt} (c) and \ref{Prop_reduc_label}. 
A detailed proof will be given in \cite{A-smc}.

\begin{Prop}\label{Prop_sim_W_U} 
Let $U=\bigoplus_{i=1}^m U_i \in \twopsilt A$, 
and use Notation \ref{Nota_X_Y}.
\begin{enumerate}
\item
Let $j \in \{m+1,\ldots,n\}$. 
Then $X_j \in \calW_U$ holds.
Moreover, $\Phi(X_j) \in \mod B$ is the simple top of 
the indecomposable projective $B$-module $\reduc(S_j)$. 
\item
We have 
\begin{align*}
\simple \calW_U=\{X_{m+1},\ldots,X_n\}, \quad
\calW_U=\Filt \{X_{m+1},\ldots,X_n\}.
\end{align*}
\end{enumerate}
\end{Prop}

Now we summarize some results in this section 
from the point of view of the exchange quiver.
By Proposition \ref{Prop_X_i^-_Y_i^+}, 
$Q(\twosilt A)$ looks like as follows
around $S$ and $T$, where $i_1,i_2,i_3 \in \{1,\ldots,m\}$.
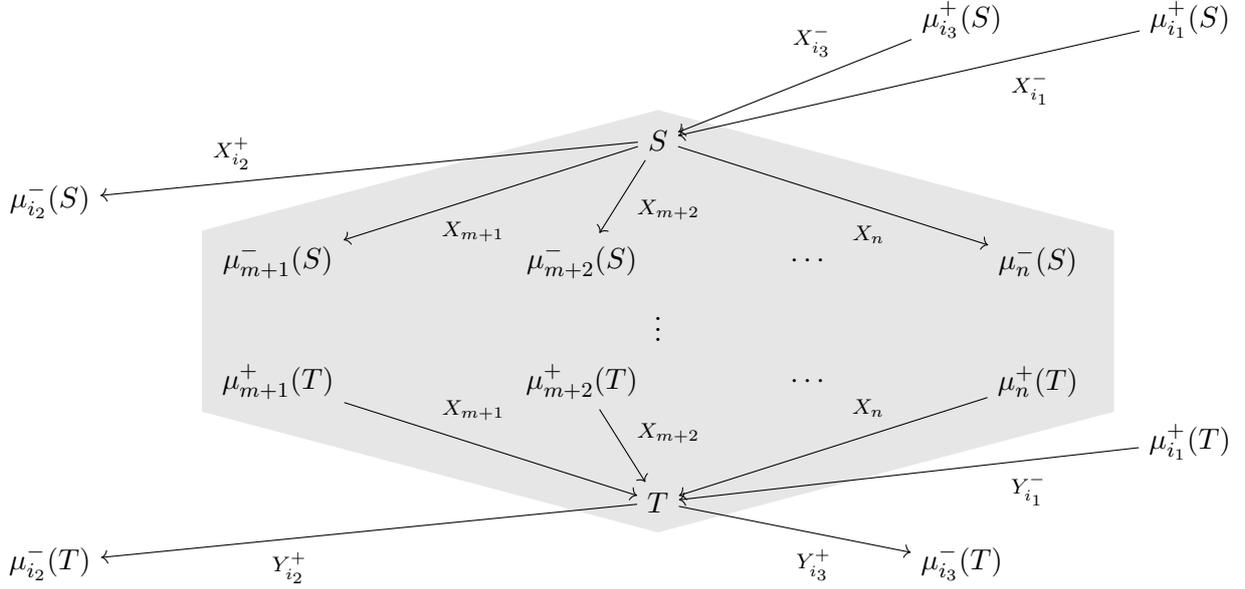
\begin{align*}
\begin{tikzpicture}[->,baseline=0pt,xscale=1, yscale=0.8]
\fill[black!10] (0,3.5)--(-6,1.5)--(-6,-1.5)--
(0,-3.5)--(6,-1.5)--(6,1.5)--cycle;
\node (S0) at ( 0, 3) {$S$};
\node (S1) at ( 7, 5) {$\mu_{i_1}^+(S)$};
\node (S2) at (-8, 2) {$\mu_{i_2}^-(S)$};
\node (S3) at ( 4, 5) {$\mu_{i_3}^+(S)$};
\node (S4) at (-5, 1) {$\mu_{m+1}^-(S)$};
\node (S5) at (-1, 1) {$\mu_{m+2}^-(S)$};
\node (S6) at ( 2, 1) {$\cdots$};
\node (S7) at ( 5, 1) {$\mu_n^-(S)$};
\node (CC) at ( 0, 0) {$\vdots$};
\node (T0) at ( 0,-3) {$T$};
\node (T1) at ( 7,-2) {$\mu_{i_1}^+(T)$};
\node (T2) at (-8,-4) {$\mu_{i_2}^-(T)$};
\node (T3) at ( 4,-4) {$\mu_{i_3}^-(T)$};
\node (T4) at (-5,-1) {$\mu_{m+1}^+(T)$};
\node (T5) at (-1,-1) {$\mu_{m+2}^+(T)$};
\node (T6) at ( 2,-1) {$\cdots$};
\node (T7) at ( 5,-1) {$\mu_n^+(T)$};
\draw (S0) to[edge label ={$\scriptstyle X_{m+1}$},pos=0.7] (S4);
\draw (S0) to[edge label ={$\scriptstyle X_{m+2}$},pos=0.4] (S5);
\draw (S0) to[edge label'={$\scriptstyle X_n$},    pos=0.7] (S7);
\draw (S1) to[edge label ={$\scriptstyle X_{i_1}^-$},pos=0.3] (S0);
\draw (S0) to[edge label'={$\scriptstyle X_{i_2}^+$},pos=0.7] (S2);
\draw (S3) to[edge label'={$\scriptstyle X_{i_3}^-$},pos=0.3] (S0);
\draw (T4) to[edge label ={$\scriptstyle X_{m+1}$},pos=0.3] (T0);
\draw (T5) to[edge label ={$\scriptstyle X_{m+2}$},pos=0.6] (T0);
\draw (T7) to[edge label'={$\scriptstyle X_n$},    pos=0.3] (T0);
\draw (T1) to[edge label ={$\scriptstyle Y_{i_1}^-$},pos=0.3] (T0);
\draw (T0) to[edge label ={$\scriptstyle Y_{i_2}^+$},pos=0.7] (T2);
\draw (T0) to[edge label'={$\scriptstyle Y_{i_3}^+$},pos=0.7] (T3);
\end{tikzpicture}
\end{align*}

The gray region denotes $Q(\twosilt_U A)$.
The arrow $\alpha_j \colon S \to \mu_j^-(S)$ labeled by $X_j$ 
for $j \in \{m+1,\ldots,n\}$
corresponds to the left mutation of $S$ at $S_j \notin \add U$.
Thus $\alpha_j$ is contained in the gray region.
Under the isomorphism $Q(\twosilt_U A) \to Q(\twosilt B)$,
the arrow $\alpha_j$ is sent to the arrow $B \to \mu_j^-(B)$ labeled
by the simple $B$-module $\Phi(X_j)$
by Propositions \ref{Prop_reduc_silt} and \ref{Prop_sim_W_U}.

We also look at the arrows not contained in $Q(\twosilt_U A)$
in the picture.
We first consider such arrows around $S \in \twosilt A$.
Since $i_1,i_2,i_3 \in \{1,\ldots,m\}$, 
we have $S_{i_1},S_{i_2},S_{i_3} \in \add U$,
so none of the mutations $\mu_{i_1}^+(S),\mu_{i_2}^-(S),\mu_{i_3}^+(S)$ of $S$
are in $\twosilt_U A$.
Thus the arrows $\mu_{i_1}^+(S) \to S$, 
$S \to \mu_{i_2}^-(S)$, $\mu_{i_3}^+(S) \to S$ starts or ends 
outside the gray region.
By Definition-Proposition \ref{Def-Prop_adjacent},
the directions of these arrows correspond to 
$X_{i_1}^- \ne 0$, $X_{i_2}^+ \ne 0$, $X_{i_3}^- \ne 0$,
which are their labels.
Similar properties hold for $T \in \twosilt A$.

By Proposition \ref{Prop_X_i^-_Y_i^+} (1), 
for each $i \in \{1,\ldots,m\}$, there are three possibilities
(i) $X_i^+ \ne 0$ and $Y_i^+ \ne 0$,
(ii) $X_i^- \ne 0$ and $Y_i^- \ne 0$,
(iii) $X_i^- \ne 0$ and $Y_i^+ \ne 0$.
The elements $i_1,i_2,i_3$ in the picture express
the situations (i), (ii), (iii), respectively.

\section{Preliminary on the real Grothendieck group}
\label{Sec_pre_Grothendieck}

In this section, we prepare many notions related to
the real Grothendieck group $K_0(\proj A)_\R$,
including the TF equivalence and interval neighborhoods.

\subsection{Polyhedral cones and generalized fans}

We prepare some conventions and properties 
on polyhedral cones and generelized fans 
in Euclidean spaces following  \cite[Subsection 2.1]{AsI2}.

For any subsets $X,Y \subset \R^n$,
we set $X+Y:=\{x+y \mid x \in X, \ y \in Y\}$ and $-X:=\{-x \mid x \in X\}$.
Thus $X+(-Y)$ means $\{x-y \mid x \in X, \ y \in Y\}$.

\begin{Def}
Let $C \subset \R^n$ be a nonempty subset.
Then $C$ is called a \emph{polyhedral cone}
(resp.~a \emph{rational polyhedral cone})
if there exist finitely many elements $v_1,\ldots,v_m \in \R^n$ 
($v_1,\ldots,v_m \in \Q^n$) such that
$C=\sum_{i=1}^m \R_{\ge 0}v_i$.
\end{Def}

If $C$ is a polyhedral cone, then $C \cap (-C)$ 
is always the maximum vector subspace of $\R^n$ contained in $C$.
We say that $C$ is \emph{strongly convex} if $C \cap (-C)=\{0\}$.

Let $C$ be a polyhedral cone in the Euclidean space $\R^n$. 
Then $\R C$ denotes the $\R$-vector subspace of $K_0(\proj A)_\R$ 
spanned by $C$.
The dimension $\dim_\R C$ of the polyhedral cone $C$ means $\dim_\R \R C$.
The \emph{relative interior} $C^\circ$ of $C$ is defined 
as the interior of $C$ as a subset of $\R C$.

Next we define faces of polyhedral cones.

\begin{Def}
Let $C$ be a polyhedral cone in $\R^n$.
A subset $F \subset C$ is called a \emph{face} 
if the following two conditions hold.
\begin{enumerate2}
\item
The subset $F$ is nonempty, convex, and satisfies $\R_{\ge 0}F \subset F$.
\item
If $v,w \in C$ satisfy $v+w \in F$, then $v,w \in F$.
\end{enumerate2}
Moreover, a face $F$ of $C$ is called a \emph{facet} of $C$ 
if $\dim_\R F=\dim_\R C-1$.
We define $\Face C$ (resp.~$\Facet C$) 
as the set of faces (resp.~facets) of $C$.
\end{Def}

It is well-known that a subset $F \subset C$ is a face
if and only if there exists a hyperplane $H \subset \R^n$ 
such that $H \cap C=F$ and that $C$ is contained in one of 
the two half-spaces given by $H$.
Moreover, facets are considered only for full-dimensional polyhedral cones
in this paper.

Next we recall the definition of generalized fans in Euclidean spaces.

\begin{Def}\label{Def_fan}\cite{CLS}
Let $\Sigma$ be a set of polyhedral cones in $\R^n$.
Then $\Sigma$ is called a \emph{generalized fan} in $\R^n$ 
if the following hold.
\begin{enumerate2}
\item
For any $\sigma \in \Sigma$, we have $\Face \sigma \subset \Sigma$.
\item
For any $\sigma_1,\sigma_2 \in \Sigma$, we have
$\sigma_1 \cap \sigma_2 \in \Face \sigma_1 \cap \Face \sigma_2$. 
\end{enumerate2}
Let $\Sigma$ be a generalized fan in $\R^n$.
\begin{enumerate}
\item
The union $\bigcup_{\sigma \in \Sigma} \sigma$ 
is called the \emph{support} of $\Sigma$,
and $\Sigma$ is said to be \emph{complete} if its support is $\R^n$.
\item
We say that $\Sigma$ is \emph{finite} if $\Sigma$ is a finite set.
\item
We say that $\Sigma$ is \emph{rational} 
if all $\sigma \in \Sigma$ are rational polyhedral cones.
\item
The generalized fan $\Sigma$ is just called a \emph{fan}
if all $\sigma \in \Sigma$ are strongly convex.
\end{enumerate}
\end{Def}

If $\frakC$ is a set of polyhedral cones in $\R^n$,
then we write $\max \frakC$ for the set of maximal elements of $\frakC$
with respect to inclusions.
The following property is easy to check.

\begin{Lem}\label{Lem_fan_inj}
Let $\Sigma,\Sigma'$ be finite complete generalized fans in $\R^n$.
Assume that there exists an injection $u \colon \max \Sigma \to \max \Sigma'$
such that any $\sigma \in \max \Sigma$ satisfies $\sigma \subset u(\sigma)$.
Then we have $\Sigma=\Sigma'$, and $u$ is identity.
\end{Lem}

To construct a new generalized fan via a linear map, 
the following property is useful.

\begin{Lem}\label{Lem_induce_fan}
Let $f \colon \R^n \to \R^{n'}$ be an $\R$-linear map,
and $\Sigma$ be a generalized fan in $\R^n$ whose support is $X$.
Set $\Sigma_\times \subset \Sigma$ by
\begin{align*} 
\Sigma_\times:=\{\sigma \in \Sigma \mid f^{-1}(f(\sigma)) \cap X=\sigma\}.
\end{align*}
Then we have the following statements.
\begin{enumerate}
\item
The set $\Sigma':=\{f(\sigma) \mid \sigma \in \Sigma_\times\}$ 
is a generalized fan in $\R^{n'}$.
\item
We have mutually inverse bijections
\begin{align*}
\Psi_1 \colon \Sigma_\times &\to \Sigma', & 
\Psi_2\colon \Sigma' &\to \Sigma_\times, \\
\sigma &\mapsto f(\sigma),& \sigma' &\mapsto f^{-1}(\sigma') \cap X.
\end{align*}
\item
If $\max \Sigma \subset \Sigma_\times$, 
then the support of $\Sigma'$ is $f(X)$.
\end{enumerate}
\end{Lem}

\begin{proof}
(1)
Clearly each element in $\Sigma'$ is a polyhedral cone in $\R^{n'}$.
We check the conditions (a) and (b) from Definition \ref{Def_fan}.

(a)
Let $\sigma' \in \Sigma'$ and $\tau' \in \Face \sigma'$.
It suffices to show $\tau' \in \Sigma'$.

Take $\sigma \in \Sigma_\times$ such that $\sigma'=f(\sigma)$.
Then we have $\sigma=f^{-1}(\sigma') \cap X$. 
Set $\tau:=f^{-1}(\tau') \cap X$.

We claim $\tau \in \Face \sigma$.
We get $\tau \subset \sigma$ by $\tau' \subset \sigma'$,
and $\tau$ is convex since so are $\tau'$ and $\sigma$.
Now let $v,w \in \sigma$ satisfy $v+w \in \tau$.
We have $f(v),f(w) \in \sigma'$ and $f(v)+f(w)=f(v+w) \in \tau'$.
Then since $\tau' \in \Face \sigma'$,
we have $f(v),f(w) \in \tau'$.
Since $v,w \in \sigma \subset X$, 
we have $v,w \in f^{-1}(\tau') \cap X=\tau$.
Thus we get $\tau \in \Face \sigma$.

Since $\Sigma$ is a generalized fan, we have $\tau \in \Sigma$.
Moreover $f(\tau)=\tau'$ follows from $\tau' \subset \sigma' \subset f(X)$
and the definition of $\tau$.
Therefore $\tau \in \Sigma_\times$ holds, so we get $\tau'=f(\tau)$.
Now (a) is proved.

(b)
Let $\sigma'_1,\sigma'_2 \in \Sigma'$.
We show $\sigma'_1 \cap \sigma'_2 \in \Face \sigma'_1 \cap \Face \sigma'_2$.

Clearly, $\sigma'_1 \cap \sigma'_2$ is a convex subset of $\sigma'_1$.
To obtain $\sigma'_1 \cap \sigma'_2 \in \Face \sigma'_1$, 
we assume $v',w' \in \sigma'_1$ and $v'+w' \in \sigma'_1 \cap \sigma'_2$, 
and prove $v',w' \in \sigma'_1 \cap \sigma'_2$.

Take $\sigma_1,\sigma_2 \in \Sigma_\times$ such that $\sigma'_1=f(\sigma_1)$
and $\sigma'_2=f(\sigma_2)$.
We have
\begin{align}\label{preimage cap}
f^{-1}(\sigma'_1 \cap \sigma'_2) \cap X
=f^{-1}(f(\sigma_1) \cap f(\sigma_2)) \cap X
&=(f^{-1}(f(\sigma_1)) \cap X) \cap (f^{-1}(f(\sigma_2)) \cap X) \notag \\
&\stackrel{\sigma_1,\sigma_2 \in \Sigma_\times}{=}\sigma_1 \cap \sigma_2.
\end{align}
Since $\sigma'_1=f(\sigma_1)$, 
we take $v,w \in \sigma_1 \subset X$ such that $v'=f(v)$ and $w'=f(w)$.
Then 
\begin{align*}
v+w \in f^{-1}(v'+w') \cap X \subset 
f^{-1}(\sigma'_1 \cap \sigma'_2) \cap X
\stackrel{\eqref{preimage cap}}{=}\sigma_1 \cap \sigma_2.
\end{align*}
Since $\sigma_1,\sigma_2 \in \sigma$ and $\Sigma$ is a generalized fan,
we have $\sigma_1 \cap \sigma_2 \in \Face \sigma_1$.
Then $v,w \in \sigma_1$ and $v+w \in \sigma_1 \cap \sigma_2$
give $v,w \in \sigma_1 \cap \sigma_2$.
Thus $v',w' \in f(\sigma_1) \cap f(\sigma_2)=\sigma'_1 \cap \sigma'_2$ hold.

Therefore $\sigma'_1 \cap \sigma'_2 \in \Face \sigma'_1$.
By symmetry, we also have $\sigma'_1 \cap \sigma'_2 \in \Face \sigma'_2$.
Thus (b) holds.

Now (a) and (b) imply that $\Sigma'$ is a generalized fan.

(2)
The definition of $\Sigma_\times$ gives that $\Psi_2 \Psi_1$ is identity.
Thus $\Psi_1$ is injective.
Moreover, the definition of $\Sigma'$ implies $\Psi_1$ is surjective.
Therefore $\Psi_1$ is bijective, and $\Psi_2=\Psi_1^{-1}$.

(3)
is obvious.
\end{proof}

\subsection{Stability conditions}

We recall stability conditions and related notions in this subsection.

As usual, $K_0(\proj A)$ denotes the Grothendieck group of 
the exact category $\proj A$.
It is well-known that $K_0(\proj A)$ is naturally identified with
the Grothendieck group $K_0(\sfK^\rmb(\proj A))$ of 
$\sfK^\rmb(\proj A)$ as a triangulated category.
The Grothendieck group $K_0(\mod A)$ of $\mod A$ is also defined,
and we regard $K_0(\sfD^\rmb(\mod A))=K_0(\mod A)$.

We mainly consider the \emph{real Grothendieck group}
$K_0(\proj A)_\R:=K_0(\proj A) \otimes_\Z \R$ in this paper.
In the same way, we set $K_0(\mod A)_\R:=K_0(\mod A) \otimes_\Z \R$.
For these real Grothendieck groups, we have the \emph{Euler bilinear form} 
as the $\R$-bilinear form 
\begin{align*}
\langle !,? \rangle \colon K_0(\proj A)_\R \times K_0(\mod A)_\R \to \R
\end{align*}
such that 
any $U \in \sfK^\rmb(\proj A)$ and $X \in \sfD^\rmb(\mod A)$ satisfy
\begin{align}\label{Euler U X}
\langle [U],[X] \rangle
=\sum_{\ell \in \Z}\dim_K \Hom_{\sfD(A)}(U,X[\ell]).
\end{align}

The Grothendieck group $K_0(\proj A)$ has a free basis $[P(1)],\ldots,[P(n)]$ 
given by the nonisomorphic indecomposable projective modules.
By this basis, we identify
$K_0(\proj A)=\bigoplus_{i=1}^n \Z[P(i)]$ with $\Z^n$
and $K_0(\proj A)_\R=\bigoplus_{i=1}^n \R[P(i)]$ with $\R^n$.
Similarly, the canonical basis of $K_0(\mod A)$ is given by
the nonisomorphic simple modules $[L(1)],\ldots,[L(n)]$,
so we see
$K_0(\mod A)=\bigoplus_{i=1}^n \Z[L(i)] \simeq \Z^n$ and 
$K_0(\mod A)_\R=\bigoplus_{i=1}^n \R[L(i)] \simeq \R^n$.

These elements satisfy
\begin{align*}
\langle [P(i)],[L(j)] \rangle=\delta_{i,j} \dim_K \End_A(L(j)),
\end{align*}
so $[P(1)],\ldots,[P(n)] \in K_0(\proj A)$ and 
$[L(1)],\ldots,[L(n)] \in K_0(\mod A)$
give dual bases of $K_0(\proj A)_\R$ and $K_0(\mod A)_\R$
(up to rescaling) with respect to the Euler bilinear form.

Each $\theta \in K_0(\proj A)_\R$ is
identified with the $\R$-linear form
$\theta:=\langle \theta,? \rangle \colon K_0(\mod A)_\R \to \R$.
For $M \in \mod A$, we shortly write $\theta(M)$
as the meaning of the value $\theta([M])=\langle \theta,[M] \rangle \in \R$.
The linear form given by $\theta \in K_0(\proj A)_\R$ 
can be used to define various notions in $\mod A$.
Among them, we focus on the following class of torsion pairs.

\begin{Def}
\cite[Subsection 3.1]{BKT}
Let $\theta \in K_0(\proj A)_\R$.
Then we define two \emph{semistable torsion pairs} 
$(\ovcalT_\theta,\calF_\theta)$ and $(\calT_\theta,\ovcalF_\theta)$
in $\mod A$ by
\begin{align*}
\ovcalT_\theta&:=\{ M \in \mod A \mid
\text{$\theta(X) \ge 0$ for all factor modules $X$ of $M$} \},\\
\calF_\theta&:=\{ M \in \mod A \mid
\text{$\theta(Y) < 0$ for all nonzero submodules $Y \ne 0$ of $M$} \},\\
\calT_\theta&:=\{ M \in \mod A \mid
\text{$\theta(X) > 0$ for all nonzero factor modules $X \ne 0$ of $M$} \},\\
\ovcalF_\theta&:=\{ M \in \mod A \mid
\text{$\theta(Y) \le 0$ for all submodules $Y$ of $M$} \}.
\end{align*}
\end{Def}

Semistable torsion pairs give rise to 
the following equivalence relation on $K_0(\proj A)_\R$,
which is the main topic throughout the paper.

\begin{Def}\cite[Definition 2.13]{A}
We define an equivalence relation on $K_0(\proj A)_\R$ 
called the \emph{TF equivalence} as follows.
Let $\theta,\eta \in K_0(\proj A)_\R$.
Then we say that $\theta$ and $\eta$ are \emph{TF equivalent} if 
\begin{align*}
(\ovcalT_\theta,\calF_\theta)=(\ovcalT_\eta,\calF_\eta)
\quad \text{and} \quad
(\calT_\theta,\ovcalF_\theta)=(\calT_\eta,\ovcalF_\eta).
\end{align*}
\end{Def}

Note that $\theta$ and $\eta$ are TF equivalent if and only if
$\calT_\theta=\calT_\eta$ and $\calF_\theta=\calF_\eta$.
We write $\TF_A$ for the set of all TF equivalence classes 
in $K_0(\proj A)_\R$.

In addition to semistable torsion pairs, 
we also use the following subcategory.

\begin{Def}\cite[Definition 1.1]{K}
For each $\theta \in K_0(\proj A)$,
we set \emph{the $\theta$-semistable subcategory} by
\begin{align*}
\calW_\theta:=\ovcalT_\theta \cap \ovcalF_\theta
=\{M \in \mod A \mid \text{$\theta(M)=0$ and $\theta(X) \ge 0$
for any factor module $X$ of $M$}\}.
\end{align*}
\end{Def}

It is well-known that $\calW_\theta$ is a wide subcategory of $\mod A$;
see \cite{HR,Rudakov}.

By using $\calW_\theta$, 
a wall-chamber structure on $K_0(\proj A)_\R$ is defined.

\begin{Def}\cite[Definition 6.1]{BST} 
\cite[Definitions 3.2, 3.3]{Bridgeland}
We define the following.
\begin{enumerate}
\item
For each nonzero module $M \in (\mod A) \setminus \{0\}$,
we define the \emph{wall} as the rational polyhedral cone
\begin{align*}
\Theta_M:=\{ \theta \in K_0(\proj A)_\R \mid M \in \calW_\theta \}.
\end{align*}
\item
The \emph{wall-chamber structure} on $K_0(\proj A)_\R$
is defined as the collection of all walls $\Theta_M$ for nonzero modules.
Here, a \emph{chamber} in this wall-chamber structure means
a connected component of the open subset
\begin{align*}
K_0(\proj A)_\R \setminus 
\overline{\bigcup_{M \in \mod A \setminus \{0\}} \Theta_M}.
\end{align*}
\end{enumerate}
\end{Def}

Set the line segment
\begin{align*}
[\theta,\eta]:=\{(1-r)\theta+r\eta \mid r \in [0,1]\} \subset K_0(\proj A)_\R
\end{align*}
connecting $\theta,\eta \in K_0(\proj A)_\R$.
Then the relationship between the wall-chamber structures and 
the TF equivalence is stated as follows.

\begin{Prop}
\cite[Theorem 2.17]{A}
Let $\theta \ne \eta \in K_0(\proj A)_\R$ be distinct elements.
Then the following conditions are equivalent.
\begin{enumerate2}
\item
The elements $\theta$ and $\eta$ are TF equivalent.
\item
For any $\zeta \in [\theta,\eta]$, the subcategory $\calW_\zeta$ is constant.
\item
There exists no nonzero module $M \in \mod A \setminus \{0\}$ such that
$[\theta,\eta] \cap \Theta_M$ is one point.
\end{enumerate2}
\end{Prop}

\subsection{$M$-TF equivalences}

As a systematic way to coarsen the TF equivalence on $K_0(\proj A)_\R$, 
we introduced the $M$-TF equivalence for each $M \in \mod A$ in \cite{AsI2}.

For $\theta \in K_0(\proj A)_\R$,
we have the torsion pairs $(\ovcalT_\theta,\calF_\theta)$
and $(\calT_\theta,\ovcalF_\theta)$ in $\mod A$
with $\calT_\theta \subset \ovcalT_\theta$
and $\calF_\theta \subset \ovcalF_\theta$.
Thus we have two exact sequences
\begin{align*}
&0 \to \ovrmt_\theta M \to M  \to \rmf_\theta M \to 0
\quad (\ovrmt_\theta M \in \ovcalT_\theta, \ 
\rmf_\theta M \in \calF_\theta), \\
&0 \to \rmt_\theta M \to M \to \ovrmf_\theta M \to 0
\quad (\rmt_\theta M \in \calT_\theta, \ 
\ovrmf_\theta M \in \ovcalF_\theta),
\end{align*}
which are unique up to isomorphisms.
Since $\calT_\theta \subset \ovcalT_\theta$
and $\calW_\theta=\ovcalT_\theta \cap \ovcalF_\theta$,
we have $\rmt_\theta M \subset \ovrmt_\theta M$,
and define
\begin{align*}
\rmw_\theta M:=\ovrmt_\theta M/\rmt_\theta M \in \calW_\theta.
\end{align*}
Therefore $M$ is divided to the three modules
$\rmt_\theta M, \rmw_\theta M, \rmf_\theta M$.

Recall that $\calW_\theta$ is always a wide subcategory of $\mod A$.
Thus $\calW_\theta$ is an abelian length category,
and satisfies the Jordan-H\"{o}lder property.
Then for each $L \in \calW_\theta$, the set
\begin{align*}
\supp_\theta L
:=\{ \text{composition factors of $L$ in $\calW_\theta$}\} \subset 
\simple \calW_\theta
\end{align*}
is well-defined,
where $\simple \calW_\theta$ is the set of isoclasses of simple objects
in $\calW_\theta$.

\begin{Def}\label{Def_M-TF}\cite[Definition 3.1]{AsI2}
Let $M \in \mod A$.
\begin{enumerate}
\item
We define an equivalence relation on $K_0(\proj A)_\R$ called 
the \emph{$M$-TF equivalence} as follows.
For any $\theta,\eta \in K_0(\proj A)_\R$, 
they are said to be \emph{$M$-TF equivalent} if 
\begin{align*}
(\rmt_\theta M,\rmw_\theta M,\rmf_\theta M)=
(\rmt_\eta M,\rmw_\eta M,\rmf_\eta M) \quad \text{and} \quad
\supp_\theta(\rmw_\theta M)=\supp_\eta(\rmw_\eta M).
\end{align*}
\item
The sets $\TF(M)$ and $\Sigma(M)$ are defined by
\begin{align*}
\TF(M):=\{ \text{all $M$-TF equivalence classes in $K_0(\proj A)_\R$} \}, 
\quad
\Sigma(M):=\{ \overline{E} \mid E \in \TF(M) \}.
\end{align*}
\end{enumerate}
\end{Def}

Under these notations, we have the following fundamental property.

\begin{Prop}\label{Prop_M-TF_fan}\cite[Theorems 3.3, 3.5 (a)]{AsI2}
Let $M \in \mod A$.
\begin{enumerate}
\item
There exists a bijection $\TF(M) \to \Sigma(M)$
such that $E \mapsto \overline{E}$.
The inverse map $\Sigma(M) \to \TF(M)$ is given 
by $\sigma \mapsto \sigma^\circ$.
\item
The set $\Sigma(M)$ is a finite complete generalized fan in $K_0(\proj A)_\R$.
In particular, each $\sigma \in \Sigma(M)$ is a rational polyhedral cone.
\end{enumerate}
\end{Prop}

Direct sums of modules and $M$-TF equivalences have 
a nice relationship.

\begin{Rem}\label{Rem_M-TF_oplus}\cite[Proposition 3.2]{AsI2}
Let $M=M_1 \oplus M_2 \in \mod A$.
Then being $M$-TF equivalent is precisely
being both $M_1$-TF equivalent and $M_2$-TF equivalent.
\end{Rem}

For each $\theta \in K_0(\proj A)_\R$, 
we set $[\theta]_M$ as the $M$-TF equivalence class of $\theta$.
If $E=[\theta]_M$, then we define
\begin{align*}
&\rmt_E M:=\rmt_\theta M, \quad
\rmw_E M:=\rmw_\theta M, \quad
\rmf_E M:=\rmf_\theta M, \quad
\ovrmt_E M:=\ovrmt_\theta M, \quad
\ovrmf_E M:=\ovrmf_\theta M, \\
&\supp_E(\rmw_E M):=\supp_\theta(\rmw_\theta M).
\end{align*}
Each $M$-TF equivalence class $E$ and its closure are described by
$\rmt_E M, \supp_E(\rmw_E M), \rmf_E M$ as follows.
Here, $\R \calX$ means the vector subspace of $K_0(\mod A)_\R$
spanned by $\{[X] \mid X \in \calX\}$ for any set $\calX \subset \mod A$.

\begin{Lem}\label{Lem_E_twf}\cite[Proposition 3.14]{AsI2}
Let $M \in \mod A$, $E \in \TF(M)$, and $\sigma:=\overline{E}$.
\begin{enumerate}
\item
We have
\begin{align*}
E=\{\theta \in K_0(\proj A)_\R \mid
\rmt_E M \in \calT_\theta, \ \supp_E(\rmw_E M) \subset \simple \calW_\theta, \ 
\rmf_E M \in \calF_\theta \}.  
\end{align*}
Thus $E$ is open in the vector subspace
$\Ker \langle ?,\supp_E(\rmw_E M) \rangle$.
\item
We have
\begin{align*}
\sigma=\{\theta \in K_0(\proj A)_\R \mid
\rmt_E M \in \ovcalT_\theta, \ 
\supp_E(\rmw_E M) \subset \calW_\theta, \ 
\rmf_E M \in \ovcalF_\theta \}.  
\end{align*}
Thus $\R \sigma=\Ker \langle ?,\supp_E(\rmw_E M) \rangle$
and $\dim_\R \sigma=n-\dim_\R (\R \supp_E(\rmw_E M))$ hold.
\end{enumerate}
\end{Lem}

The following properties of 
faces of each element in $\Sigma(M)$ are important.

\begin{Prop}\label{Prop_M-TF_face}\cite[Corollary 3.6]{AsI2}
Let $M \in \mod A$, $\sigma \in \Sigma(M)$
and $\sigma' \in \Face \sigma$.
\begin{enumerate}
\item
We have $(\sigma')^\circ \in \TF(M)$ and $\sigma' \in \Sigma(M)$.
Thus the decomposition 
\begin{align*}
\sigma=\bigsqcup_{\tau \in \Face \sigma}\tau^\circ
\end{align*}
is the decomposition into $M$-TF equivalence classes.
\item
Set $E':=(\sigma')^\circ$. Then 
$\tau=\{\eta \in \sigma \mid \eta(\supp_{E'}(\rmw_{E'} M))=0\}$
holds.
\end{enumerate}
\end{Prop}

We set $\TF_n(M)$ and $\Sigma_n(M)$ as the sets of all full-dimensional 
elements in $\TF(M)$ and $\Sigma(M)$ respectively.
Then $\TF_n(M)$ is the set of all $M$-TF equivalence classes
which are open in $K_0(\proj A)_\R$,
and $\Sigma_n(M)$ is the set of maximal elements of $\Sigma(M)$.

\begin{Lem}\label{Lem_M-TF_open}\cite[Lemma 4.1]{AsI2}
Let $M \in \mod A$.
\begin{enumerate}
\item 
We have mutually inverse bijections $\TF_n(M) \leftrightarrow \Sigma_n(M)$
given by $E \mapsto \overline{E}$ and $\sigma \mapsto \sigma^\circ$.
\item
Let $E \in \TF(M)$.
Then $E \in \TF_n(M)$ holds if and only if $\rmw_\eta M=0$.
\end{enumerate}
\end{Lem}

The boundary of a full-dimensional element $\sigma \in \Sigma_n(M)$
satisfies the following properties.

\begin{Prop}\label{Prop_M-TF_+-_no_facet}
\cite[Theorem 4.4]{AsI2}
Let $M \in \mod A$ and $E \in \TF_n(M)$, 
$\sigma:=\overline{E} \in \Sigma_n(M)$.
Set
\begin{align*}
\partial^+\sigma:=\{\theta \in \sigma \mid \rmt_E M \notin \calT_\theta\}, 
\quad
\partial^-\sigma:=\{\theta \in \sigma \mid \rmf_E M \notin \calF_\theta\}.
\end{align*}
Then $\partial^+\sigma \cap \partial^-\sigma$ contains no facets of $\sigma$.
\end{Prop}

We also have the following property.

\begin{Lem}\label{Lem_M-TF_max_min}\cite[Lemma 3.19]{AsI2}
Let $M \in \mod A$ and $E \in \TF(M)$, $\sigma:=\overline{E}$.
For any $\theta \in \sigma$, 
we have 
\begin{align*}
\theta(\rmt_\theta M)=\theta(\ovrmt_\theta M)=
\theta(\rmt_E M)=\theta(\ovrmt_E M)
&=\max\{ \theta(L) \mid \text{$L$ is a submodule of $M$}\},\\
\theta(\rmf_\theta M)=\theta(\ovrmf_\theta M)=
\theta(\rmf_E M)=\theta(\ovrmf_E M)
&=\min\{ \theta(L) \mid \text{$L$ is a factor module of $M$}\}.
\end{align*}
\end{Lem}

\subsection{Silting cones and the $g$-fan}

To study 2-term silting complexes,
using the real Grothendieck group $K_0(\proj A)_\R$ is effective.
We recall relevant properties in this subsection.

For each $U=\bigoplus_{i=1}^m U_i \in \twopsilt A$
with each $U_i$ indecomposable,
we define the \emph{silting cones} $C^\circ(U)$ and $C(U)$
in the real Grothendieck group $K_0(\proj A)_\R$ by
\begin{align*}
C^\circ(U):=\sum_{i=1}^m \R_{>0}[U_i] \quad \text{and} \quad
C(U):=\sum_{i=1}^m \R_{\ge 0}[U_i].
\end{align*}
Then $C(U)$ is an $m$-dimensional rational polyhedral cone,
since $[U_1],\ldots,[U_m]$ can be extended to a free basis of $K_0(\proj A)$ 
\cite[Theorem 2.27]{AI}.

Each silting cone $C^\circ(U)$ is a typical example of TF equivalence classes.
For (1), see also \cite[Proposition 3.3]{Y} and \cite[Proposition 3.27]{BST}.

\begin{Prop}\label{Prop_cone_TF}
We have the following assertions.
\begin{enumerate}
\item
\cite[Proposition 3.11]{A} 
Let $U \in \twopsilt A$.
Then $C^\circ(U)$ is a TF equivalence class satisfying
\begin{align*}
C^\circ(U)&=\{\theta\in K_0(\proj A)_\R\mid
(\ovcalT_\theta,\calF_\theta)=(\ovcalT_U,\calF_U), \ 
(\calT_\theta,\ovcalF_\theta)=(\calT_U,\ovcalF_U)\}\\
&=\{\theta\in K_0(\proj A)_\R\mid\calT_\theta=\calT_U,\ 
\calF_\theta=\calF_U\}.
\end{align*}
\item
The correspondence $U \mapsto C^\circ(U)$ gives 
an injection $\twopsilt A \to \TF_A$. 
\end{enumerate}
\end{Prop}

\begin{proof}
(2)
The map is well-defined by (1), 
and the injectivity follows from Lemma \ref{Lem_T_U_T_V} (1).
\end{proof}

As a consequence of \cite{AIR,DK,DIJ,DF,P},
the set $\{C(U) \mid U \in \twopsilt A\}$ is a fan in $K_0(\proj A)_\R$
called the $g$-fan; see \cite[Proposition 4.2]{AHIKM} for the detail. 
This fan is not complete in general.
We set a symbol denoting its support.

\begin{Def}\label{Def_rigid}
We define a subset $\Cone \subset K_0(\proj A)_\R$ by
\begin{align*}
\Cone=\Cone_A:=\bigsqcup_{U \in \twopsilt A} C^\circ(U)
=\bigcup_{U \in \twopsilt A} C(U)
=\bigcup_{T \in \twosilt A} C(T).
\end{align*}
\end{Def}

There is the following characteization of the completeness of the $g$-fan.

\begin{Def-Prop}\label{Def-Prop_brick_fin}
\cite[Theorem 4.2]{DIJ} \cite[Theorem 4.7]{A}
A finite dimensional algebra $A$ is said to be \emph{brick finite}
if the following equivalent conditions hold.
\begin{enumerate2}
\item
There exist finitely many isoclasses of bricks in $\mod A$.
\item
The set $\twosilt A$ is finite.
\item
The $g$-fan is complete; that is, $\Cone=K_0(\proj A)_\R$.
\end{enumerate2}
\end{Def-Prop}

For each indecomposable $V \in \twopsilt A$, we set 
\begin{align*}
R_V:=\End_{\sfD(A)}(V)/{\rad\End_{\sfD(A)}(V)}, \quad
d_V:=\dim_K R_V.
\end{align*}
Then Definition-Proposition \ref{Def-Prop_SH}
and Lemma \ref{Lem_R_V} give the following duality.

\begin{Prop}\label{Prop_dual_basis}
Let $S=\bigoplus_{i=1}^n S_i \in \silt A$,
and set $X=\bigoplus_{i=1}^n X_i:=\SH(X) \in \smc A$.
Then for each $i,j \in \{1,\ldots,n\}$, 
\begin{align*}
[S_i](X_j)&=\delta_{i,j}\dim_k \End_{\sfD(A)}(X_j)=\delta_{i,j}d_{U_i}.
\end{align*}
\end{Prop}

Then we have the following equality.

\begin{Prop}\label{Prop_dual_basis_mut}
Let $U=\bigoplus_{i=1}^m U_i \in \twopsilt A$.
Under Notations \ref{Nota_S_T} and \ref{Nota_X_Y}, 
for any $j \in J$ and $i \in \{1,\ldots,n\} \setminus J$,
we define integers $a_{j,i},b_{j,i} \in \Z_{\ge 0}$ 
so that the following hold.
\begin{enumerate2}
\item
In the triangle $S_j \to U'_j \to T_j \to S_j$ in $\sfK^\rmb(\proj A)$, 
we have
$U'_j \simeq \bigoplus_{i=1}^m U_i^{\oplus a_{j,i}}$ in $\add U$.
\item
In the triangle $X_i[-1] \to W_i \to Y_i \to X_i$ in $\sfD(A)$,
the object $X_j$ appears exactly $b_{j,i}$ times as composition factors
of $W_i$ in the abelian length category $\Filt\{X_{m+1},\ldots,X_n\}$.
\end{enumerate2}
Then we have
\begin{align*}
a_{j,i}d_{U_i}=b_{j,i}d_{U_j}.
\end{align*}
\end{Prop}

\begin{proof}
Let $j \in J$ and $i \in \{1,\ldots,n\} \setminus J$.
By definition, we have
\begin{align*}
[T_j]&=-[S_j]+[U'_j]=-[S_j]+\sum_{i'=1}^m a_{j,i'}[U_i] \quad 
\text{in $K_0(\proj A)$}, \\
[Y_i]&=[W_i]+[X_i]=\sum_{j'=m+1}^n b_{j',i}[X_j]+[X_i] \quad 
\text{in $K_0(\mod A)$}.
\end{align*}
Therefore Proposition \ref{Prop_dual_basis} gives
\begin{align*}
0=\langle [T_j],[Y_i] \rangle
&=\left\langle -[S_j]+\sum_{i'=1}^m a_{j,i'}[U_i], 
\sum_{j'=m+1}^n b_{j',i}[X_j]+[X_i] \right\rangle 
=-b_{j,i}d_{U_j}+a_{j,i}d_{U_i}.
\qedhere
\end{align*}
\end{proof}

\subsection{The interval neighborhood $D(U)$}
\label{subsec reduc stab}

This subsection is devoted to giving 
basic properties of the interval neighborhoods $D(U),D^\circ(U)$.
They are defined as follows.

\begin{Def-Prop}\label{Def-Prop_D(U)}
Let $U \in \twopsilt A$. 
We define the \emph{closed interval neighborhood} $D(U)$ 
of $C^\circ(U)$ by
\begin{align*}
D(U)&:=\{ \theta \in K_0(\proj A)_\R \mid \calT_U \subset \ovcalT_\theta, \ 
\calF_U \subset \ovcalF_\theta \}\\
&=\{ \theta \in K_0(\proj A)_\R \mid H^0(U)\in \ovcalT_\theta, \ 
H^{-1}(\nu U)\in \ovcalF_\theta \}.
\end{align*}
This is a rational polyhedral cone 
in $K_0(\proj A)_\R$ by the second equality.

We also define the \emph{open interval neighborhood} 
$D^\circ(U)$ of $C^\circ(U)$ by
\begin{align*}
D^\circ(U)&:=\{ \theta \in K_0(\proj A)_\R \mid 
\calT_U \subset \calT_\theta, \ \calF_U \subset \calF_\theta \}\\
&=\{ \theta \in K_0(\proj A)_\R \mid H^0(U) \in\calT_\theta, \ 
H^{-1}(\nu U) \in \calF_\theta \}.
\end{align*}
This is open by the second equality. 
Moreover, it is an open neighborhood of $C^\circ(U)$ 
by the first equality and Proposition \ref{Prop_cone_TF}. 

Moreover $D^\circ(U)$ is the interior of $D(U)$,
and $D(U)$ is the closure of $D^\circ(U)$.
\end{Def-Prop}

These are often called just the interval neighborhoods of $C^\circ(U)$.
We use the word ``interval'', because
\begin{align*}
D^\circ(U)=\{\theta \in K_0(\proj A)_\R \mid 
[\calT_\theta,\ovcalT_\theta] \subset [\calT_U,\ovcalT_U] \}
\end{align*}
holds.
Note that the interval $[\calT_U,\ovcalT_U] \subset \tors A$
appeared in Proposition \ref{Prop_reduc}.
We also remark that $D^\circ(U)$ is denoted by $N_U$ 
in \cite[Subsection 4.1]{A}.

If $U=U' \oplus U'' \in \twopsilt A$, then we have
\begin{align*}
D(U)=D(U') \cap D(U''), \quad
D^\circ(U)=D^\circ(U') \cap D^\circ(U''),
\end{align*}
from the second descriptions of $D^\circ(U)$ and $D(U)$ 
in Definition-Proposition \ref{Def-Prop_D(U)}.

The interval neighborhoods and direct sums of 2-term presilting complexes
are related as follows.

\begin{Lem}\label{Lem_D(U)_incl}
Let $U,V \in \twopsilt A$. 
\begin{enumerate}
\item 
The following conditions are equivalent.
\begin{enumerate2}
\item
The complex $U$ is a direct summand of $V$.
\item
Both $\calT_U \subset \calT_V$ and $\calF_U \subset \calF_V$ hold. 
\item
We have $C^\circ(V) \subset D^\circ(U)$. 
\item
We have $C^\circ(V) \cap D^\circ(U) \ne \emptyset$
(or equivalently, $C(V) \cap D^\circ(U) \ne \emptyset$). 
\item
We have $D^\circ(V) \subset D^\circ(U)$.
\end{enumerate2}
\item The following conditions are equivalent.
\begin{enumerate2}
\item
The direct sum $U \oplus V$ is 2-term presilting.
\item
Both $\calT_U \subset \ovcalT_V$ and $\calF_U \subset \ovcalF_V$ hold.
\item
We have $C^\circ(V) \subset D(U)$
(or equivalently, $C(V) \subset D(U)$).
\item
We have $C^\circ(V) \cap D(U) \ne \emptyset$.
\item
We have $D^\circ(U) \cap D^\circ(V) \ne \emptyset$ 
(or equivalently, $D(U) \cap D^\circ(V) \ne \emptyset$).
\end{enumerate2}
\end{enumerate}
\end{Lem}

\begin{proof}
(1) 
(a)$\Leftrightarrow$(b) is Lemma \ref{Lem_T_U_T_V} (1).

(b)$\Leftrightarrow$(c)
By Proposition \ref{Prop_cone_TF} (1), 
$\calT_V=\calT_\theta$ and $\calF_V=\calF_\theta$
hold for any $\theta \in C^\circ(U)$.
Thus (b) is equivalent to that 
both $\calT_U \subset \calT_\theta$ and $\calF_U \subset \calF_\theta$ hold
for any $\theta \in C^\circ(U)$,
which is nothing but (c).

(c)$\Leftrightarrow$(d)
By Proposition \ref{Prop_cone_TF} (1), $C^\circ(V)$ is a TF equivalence class.
By definition, $D^\circ(U)$ is a union of TF equivalence classes.
Thus (c) and (d) are equivalent.

(b)$\Rightarrow$(e) follows from $D^\circ(V)=D^\circ(U) \cap D^\circ(V/U)$.

(e)$\Rightarrow$(c) is obvious,
since $C^\circ(V) \subset D^\circ(V)$.

(2)
(a)$\Leftrightarrow$(b) is Lemma \ref{Lem_T_U_T_V} (2).

(b)$\Leftrightarrow$(c)
is shown in a similar way to (1) (b)$\Leftrightarrow$(c).

(c)$\Rightarrow$(d)$\Rightarrow$(e) are obvious.

(e)$\Rightarrow$(b)
Take $\theta \in D^\circ(U) \cap D^\circ(V)$.
Then we get 
$\calT_U \subset \calT_\theta \subset \ovcalT_\theta 
\subset \ovcalT_V$.
Similarly $\calF_U \subset \ovcalF_V$ holds.
\end{proof}

In Definition \ref{Def_rigid}, we defined
\begin{align*}
\Cone:=\bigcup_{T \in \twosilt A} C(T)=
\bigsqcup_{U \in \twopsilt A} C^\circ(U) \subset K_0(\proj A)_\R.
\end{align*}
Then $\Cone=K_0(\proj A)_\R$ is equivalent to that 
$A$ is brick finite by Definition-Proposition \ref{Def-Prop_brick_fin}.

Lemma \ref{Lem_D(U)_incl} and Proposition \ref{Prop_U_S}
give the following equalities.

\begin{Prop}\label{Prop_D(U)_Cone}
Let $U \in \twopsilt A$.
We have
\begin{align*}
D^\circ(U) \cap \Cone=\bigsqcup_{V \in \twopsilt_U A} C^\circ(V),\quad
D(U) \cap \Cone=\bigcup_{V \in \twopsilt_U A} C(V)
=\bigcup_{V \in \twosilt_U A} C(V).
\end{align*}
\end{Prop}

\begin{proof}
For $V \in \twopsilt_U A$, by Lemma \ref{Lem_D(U)_incl} (1),
$C^\circ(V) \subset D^\circ(U)$ if $U \in \add V$,
and $D^\circ(U) \cap C^\circ(V)=\emptyset$ otherwise.
Thus the equality on $D^\circ(U)$ follows. 

We next show the first equality on $D(U) \cap \Cone$.
If $\theta \in D(U) \cap \Cone$, 
then there exists $V' \in \twopsilt A$ such that 
$\theta \in C^\circ(V')$ and $U \oplus V'$ is 2-term presilting
by Lemma \ref{Lem_D(U)_incl} (2),
so we get $\theta \in C(V)$,
where $V \in \twopsilt_U A$ is the basic 2-term presilting complex 
such that $\add V=\add(U \oplus V')$.
Conversely, if $V \in \twopsilt_U A$,
then $C(V) \subset D(U) \cap \Cone$ holds, 
again by Lemma \ref{Lem_D(U)_incl} (2).
Thus the first equality on $D(U) \cap \Cone$ has been proved.

The second equality on $D(U) \cap \Cone$ follows from 
Proposition \ref{Prop_U_S}.
\end{proof}

We use the following observation later.
Recall that a polyhedral cone $C \subset K_0(\proj A)_\R$
is said to be \emph{strongly convex} if $C \cap (-C)=0$.

\begin{Prop}\label{Prop_strong_convex}
Let $U \in \twopsilt A$.
Set $I \subset \{1,\ldots,n\}$ and $H \subset K_0(\proj A)_\R$ by
\begin{align*}
I:=\{i \in \{1,\ldots,n\} \mid
\text{$L(i)$ is not a composition factor of
$H^0(U) \oplus H^{-1}(\nu U)$ in $\mod A$}\}
\end{align*}
and $H:=\bigoplus_{i \in I}\R[P(i)]$.
Then the following statements hold.
\begin{enumerate}
\item
We have $D(U) \cap (-D(U))=H$.
In particular, $D(U)$ is strongly convex if and only if
$H^0(U) \oplus H^{-1}(\nu U)$ is sincere.
\item
Set $H':=\bigoplus_{i \notin I'} \R[P(i)]$.
Then $D(U) \cap H'$ is strongly convex,
and $D(U)=(D(U) \cap H')+H$.
\end{enumerate}
\end{Prop}

\begin{proof}
(1)
First, let $\theta \in D(U) \cap (-D(U))$.
Then for any factor module $M$ of $H^0(U)$,
we have $\theta(M) \ge 0$ and $-\theta(M) \ge 0$;
hence $\theta(M)=0$.
Thus if $L$ is a composition factor of $H^0(U)$ in $\mod A$,
then $\theta(L)=0$.
Similarly, if $L$ is a composition factor of $H^{-1}(\nu U)$ in $\mod A$,
then $\theta(L)=0$.
Therefore for any $i \notin I$, we have $\theta(L)=0$,
which implies $\theta \in H$.

Conversely, assume $\theta \in H$.
If $M$ is a factor module of $H^0(U)$ or a submodule of $H^{-1}(\nu U)$,
then $M$ is filtered by the simple modules $L(i)$ for $i \notin I$,
so $\theta(M)=-\theta(M)=0$.
Thus $\theta \in D(U) \cap (-D(U))$. 

(2) is clear by (1).
\end{proof}

We will explain the relationship 
between interval neighborhoods and $\tau$-tilting reduction.

Let $U=\bigoplus_{i=1}^m U_i \in \twopsilt A$.
Take its maximal completion $S=\bigoplus_{i=1}^n S_i$
as in Notation \ref{Nota_S_T}.
Then we define the algebra $B$ by \eqref{def B};
namely, $B:=\End_{\sfD(A)}(S)/\langle e \rangle$
with $e$ the idempotent $S \to U \to S$.

By Proposition \ref{Prop_reduc_silt},
there exists a triangle functor 
$\reduc \colon \sfK^\rmb(\proj A) \to \sfK^\rmb(\proj B)$
which gives a bijection
$\reduc \colon \twopsilt_U A \to \twopsilt B$
such that $\reduc(S)=B$ and $\reduc(U)=0$.
As in Proposition \ref{Prop_sim_W_U} (1),
\begin{align}\label{Eq_proj_B}
P(1)_B:=\reduc(S_{m+1}),\quad \ldots,\quad P(n-m)_B:=\reduc(S_n)
\end{align} 
are the nonisomorphic indecomposable projective $B$-modules.
Thus they give a canonical basis of $K_0(\proj B)_\R$.
The functor $\reduc$ induces 
a linear map $\pi \colon K_0(\proj A)_\R \to K_0(\proj B)_\R$ as follows.

\begin{Def-Prop}\label{Def-Prop_pi}
\cite[Theorem 4.9]{AHIKM}
\cite[Subsection 4.1]{A}
Let $U \in \twopsilt A$, and use the setting above.
Then there uniquely exists an $\R$-linear map 
\begin{align*}
\pi=\pi_U \colon K_0(\proj A)_\R \to K_0(\proj B)_\R
\end{align*}
satisfying the following conditions.
\begin{enumerate2}
\item
For any $V \in \sfK^\rmb(\proj A)$,
we have $\pi([V])=[\reduc(V)]$ in $K_0(\proj B)$.
\item
The map $\pi$ is a surjection with $\Ker \pi=\R C(U)$.
\item
If $i \in \{1,\ldots,m\}$, then $\pi([S_i])=\pi([U_i])=0$.
If $j \in \{m+1,\ldots,n\}$, then $\pi([S_j])=[P(j-m)_B]$.
\end{enumerate2}
\end{Def-Prop}

Then we have the following fundamental properties of $\pi$,
where $\TF_A^{D^\circ(U)}$ is 
the set of all TF equivalence classes contained in $D^\circ(U)$,
which is a union of TF equivalence classes in $K_0(\proj A)_\R$
by definition.

\begin{Prop}\label{Prop_reduc_K_0}\cite[Lemma 4.4, Theorem 4.5]{A}
Let $U \in \twopsilt A$.
\begin{enumerate}
\item
The map $\pi|_{D^\circ(U)} \colon D^\circ(U) \to K_0(\proj B)_\R$ 
is surjective.
\item
There exists a bijection
\begin{align*}
\TF_A^{D^\circ(U)} \to \TF_B; \quad E \mapsto \pi(E).
\end{align*}
In particular, if $\theta,\eta \in D^\circ(U)$ 
satisfies $\eta-\theta \in \R C(U)$,
then $\theta$ and $\eta$ are TF equivalent.
\item
For each $\theta \in D^\circ(U)$ and $M \in \calW_U$, 
we have $(\pi(\theta))(\Phi(M))=\theta(M)$ and
\begin{align*}
(\Phi(\ovcalT_\theta \cap \calW_U),\Phi(\calF_\theta \cap \calW_U))
&=(\ovcalT_{\pi(\theta)},\calF_{\pi(\theta)}), \\
(\Phi(\calT_\theta \cap \calW_U),\Phi(\ovcalF_\theta \cap \calW_U))
&=(\calT_{\pi(\theta)},\ovcalF_{\pi(\theta)}), \\
\Phi(\calW_\theta)&=\calW_{\pi(\theta)} \ \text{with} \ 
\calW_\theta \subset \calW_U.
\end{align*}
\item
Any $V \in \twopsilt_U A$ satisfies $\pi(C^\circ(V))=C^\circ(\reduc(V))$.
\end{enumerate}
\end{Prop}

We remark the following equality.

\begin{Prop}\label{Prop_pi_Cone}
Let $U \in \twopsilt A$.
Then we have $\pi(D^\circ(U) \cap \Cone)=\Cone_B$.
\end{Prop}

\begin{proof}
This follows as 
\begin{align*}
\pi(D^\circ(U) \cap \Cone) 
&\stackrel{\text{Prop.~\ref{Prop_D(U)_Cone}}}{=}
\pi\left(\bigcup_{V \in \twopsilt_U A} C^\circ(V)\right)
=\bigcup_{V \in \twopsilt_U A} \pi(C^\circ(V))\\
&\stackrel{\text{Prop.~\ref{Prop_reduc_K_0}}}{=}
\bigcup_{V \in \twopsilt_U A} C^\circ(\reduc(V))
\stackrel{\text{Prop.~\ref{Prop_reduc}}}{=}
\bigcup_{V' \in \twopsilt B} C^\circ(V')=\Cone_B. \qedhere
\end{align*}
\end{proof}

We can characterize the brick finiteness of $B$ as follows.

\begin{Prop}\label{Prop_B_brick_fin}
Let $U \in \twopsilt A$.
Then the following conditions are equivalent.
\begin{enumerate2}
\item
The algebra $B$ is brick finite.
\item
The set $\twosilt_U A$ is a finite set.
\item
The inclusion $D^\circ(U) \subset \Cone$ holds.
\item
The inclusion $D(U) \subset \Cone$ holds.
\end{enumerate2}
In this case, we have 
\begin{align*}
D^\circ(U)=\bigsqcup_{V \in \twopsilt_U A} C^\circ(V),\quad
D(U)=\bigcup_{V \in \twopsilt_U A} C(V)=\bigcup_{V \in \twosilt_U A} C(V).
\end{align*}
\end{Prop}

\begin{proof}
(a)$\Leftrightarrow$(b) 
By Definition-Proposition \ref{Def-Prop_brick_fin},
(a) is equivalent to that $\twosilt B$ is finite.
Moreover $\twosilt B$ is finite if and only if (b) holds, 
by the bijection
$\twosilt_U A \simeq \twosilt B$ in Proposition \ref{Prop_reduc_silt} (b).

(a)$\Leftrightarrow$(c)
Proposition \ref{Prop_reduc_K_0} (1)(2) yield that
$\pi|_{D^\circ(U)} \colon D^\circ(U) \to K_0(\proj B)_\R$
is a surjection preserving TF equivalence classes.
Moreover $D^\circ(U) \cap \Cone$ is a union of some TF equivalence classes
contained in $D^\circ(U)$ by Proposition \ref{Prop_cone_TF}.
By these, $D^\circ(U) \cap \Cone=D^\circ(U)$ holds
if and only if $\pi(D^\circ(U) \cap \Cone)=K_0(\proj B)_\R$.
The former is nothing but (c).
The latter is $\Cone_B=K_0(\proj B)_\R$ by Proposition \ref{Prop_pi_Cone},
and it is equivalent to (a) 
by Definition-Proposition \ref{Def-Prop_brick_fin}.

((b) and (c))$\Rightarrow$(d)
By (c) and Proposition \ref{Prop_D(U)_Cone}, 
$D^\circ(U)$ is the union of all $C^\circ(V)$ for $V \in \twopsilt_U A$.
This union is a finite union by (b).
Thus taking closures, we get that
$D(U)$ is the union of all $C(V)$ for $V \in \twopsilt_U A$.
Thus (d) holds.

(d)$\Rightarrow$(c) is obvious.

The last equalities come from Proposition \ref{Prop_D(U)_Cone}.
\end{proof}

We end this section by giving a concrete example of interval neighborhoods.

\begin{Ex}\label{Ex_A4_D(U)}
Let $A$ be the path algebra of the quiver $1 \to 2 \to 3 \to 4$,
and $U=U_1 \oplus U_2  \in \twopsilt A$
with $U_1=(P(4) \to P(3))$ and $U_2=P(1)$.
In this case, the maximal completion $S$ of $U$ is given by
\begin{align*}
S=U \oplus S_3 \oplus S_4 \quad \text{with} \quad
S_3=P(2), \quad S_4=P(3).
\end{align*}
Minimal left ($\add U$)-approximations of $S_3$ and $S_4$ are
\begin{align*}
S_3 \to U_2, \quad S_4 \to U_1 \oplus U_2,
\end{align*}
and their mapping cones are the 2-term complexes 
$(P(2) \to P(1))$ and $(P(4) \to P(1))$.
Therefore the minimal completion $T$ of $U$ is 
\begin{align*}
T=U \oplus T_3 \oplus T_4 \quad \text{with} \quad
T_3=(P(2) \to P(1)), \quad T_4=(P(4) \to P(1)).
\end{align*}

Then we have
\begin{align*}
H^0(U_1)&=\begin{smallmatrix}3\end{smallmatrix}, &
H^0(U_2)&=\begin{smallmatrix}1\\2\\3\\4\end{smallmatrix}, &
H^{-1}(\nu U_1)&=\begin{smallmatrix}4\end{smallmatrix}, &
H^{-1}(\nu U_2)&=0, \\
H^0(S_3)&=\begin{smallmatrix}2\\3\\4\end{smallmatrix}, &
H^0(S_4)&=\begin{smallmatrix}3\\4\end{smallmatrix}, &
H^{-1}(\nu S_3)&=0, &
H^{-1}(\nu S_4)&=0, \\
H^0(T_3)&=\begin{smallmatrix}1\end{smallmatrix}, &
H^0(T_4)&=\begin{smallmatrix}1\\2\\3\end{smallmatrix}, &
H^{-1}(\nu T_3)&=\begin{smallmatrix}2\end{smallmatrix}, &
H^{-1}(\nu T_4)&=\begin{smallmatrix}2\\3\\4\end{smallmatrix}.
\end{align*}

Setting $e$ as the idempotent $H^0(S) \to H^0(U) \to H^0(S)$,
we can check that 
$B=H^0(S)/\langle e \rangle$ is isomorphic to the path algebra $K(1 \to 2)$. 
Note also that the indices of the vertices are 
consistent with \eqref{Eq_proj_B}.
By Proposition \ref{Prop_rad_soc},
the 2-term simple-minded collections $X=\bigoplus_{i=1}^4 X_i := \SH(S)$ 
and $Y=\bigoplus_{i=1}^4 Y_i := \SH(T)$ are given by
\begin{align*}
X_1&=\begin{smallmatrix}4\end{smallmatrix}[1], &
X_2&=\begin{smallmatrix}1\end{smallmatrix}, &
X_3&=\begin{smallmatrix}2\end{smallmatrix}, &
X_4&=\begin{smallmatrix}3\\4\end{smallmatrix}, \\
Y_1&=\begin{smallmatrix}3\end{smallmatrix}, &
Y_2&=\begin{smallmatrix}1\\2\\3\\4\end{smallmatrix}, &
Y_3&=\begin{smallmatrix}2\end{smallmatrix}[1], &
Y_4&=\begin{smallmatrix}3\\4\end{smallmatrix}[1].
\end{align*}
Therefore Proposition \ref{Prop_sim_W_U} implies 
that the simple objects of $\calW_U$ are 
$X_3=\begin{smallmatrix}2\end{smallmatrix}$ and 
$X_4=\begin{smallmatrix}3\\4\end{smallmatrix}$.
Under Notation \ref{Nota_X_Y}, the modules $W_1,W_2$ are
\begin{align}\label{W_1 W_2}
W_1=\begin{smallmatrix}3\\4\end{smallmatrix}, \quad
W_2=\begin{smallmatrix}2\\3\\4\end{smallmatrix}.
\end{align}

We also set $V:=(P(4) \to P(2)) \in \twopsilt A$,
and then $U \oplus V \in \twopsilt_U A$ holds.
Thus by Proposition \ref{Prop_reduc_silt},
we define $V_B:=\reduc(U)=\reduc(U \oplus V) \in \twopsilt B$,
which is indecomposable.
We have 
$[V]=[P(2)]-[P(4)]=[U_1]+[S_3]-[S_4]$ in $K_0(\proj A)_\R$.
Thus $[\reduc(V)]=\pi([V])=[P(1)_B]-[P(2)_B]$ 
by Definition-Proposition \ref{Def-Prop_pi} and \eqref{Eq_proj_B}.
In particular, $V^B=(P(2)_B \to P(1)_B)$.
We have
\begin{align*}
H^0(V)&=\begin{smallmatrix}2\\3\end{smallmatrix}, &
H^{-1}(\nu V)&=\begin{smallmatrix}3\\4\end{smallmatrix}, \\
H^0(V^B)&=\begin{smallmatrix}1\end{smallmatrix}, &
H^{-1}(\nu V^B)&=\begin{smallmatrix}2\end{smallmatrix}.
\end{align*}

The set $\twosilt_U A$ has exactly five elements,
and so does $\twosilt B$ by Proposition \ref{Prop_reduc_silt}.
Their exchange quivers with labels are given as follows,
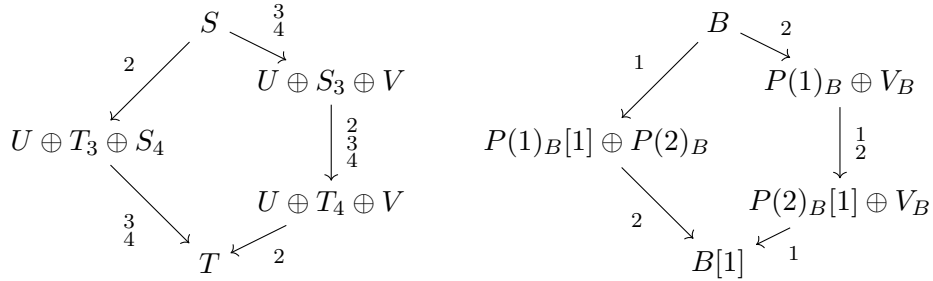
\begin{align*}
\begin{tikzpicture}[->,baseline=0pt,scale=0.8]
\node (1) at ( 0, 2) {$S$};
\node (2) at ( 2, 1) {$U \oplus S_3 \oplus V$};
\node (3) at ( 2,-1) {$U \oplus T_4 \oplus V$};
\node (4) at (-2, 0) {$U \oplus T_3 \oplus S_4$};
\node (5) at ( 0,-2) {$T$};
\draw (1) to [edge label={$\begin{smallmatrix}3\\4\end{smallmatrix}$}] (2);
\draw (2) to [edge label={$\begin{smallmatrix}2\\3\\4\end{smallmatrix}$}] (3);
\draw (3) to [edge label={$\begin{smallmatrix}2\end{smallmatrix}$}] (5);
\draw (1) to [edge label'={$\begin{smallmatrix}2\end{smallmatrix}$}] (4);
\draw (4) to [edge label'={$\begin{smallmatrix}3\\4\end{smallmatrix}$}] (5);
\end{tikzpicture} \qquad
\begin{tikzpicture}[->,baseline=0pt,scale=0.8]
\node (1) at ( 0, 2) {$B$};
\node (2) at ( 2, 1) {$P(1)_B \oplus V_B$};
\node (3) at ( 2,-1) {$P(2)_B[1] \oplus V_B$};
\node (4) at (-2, 0) {$P(1)_B[1] \oplus P(2)_B$};
\node (5) at ( 0,-2) {$B[1]$};
\draw (1) to [edge label={$\begin{smallmatrix}2\end{smallmatrix}$}] (2);
\draw (2) to [edge label={$\begin{smallmatrix}1\\2\end{smallmatrix}$}] (3);
\draw (3) to [edge label={$\begin{smallmatrix}1\end{smallmatrix}$}] (5);
\draw (1) to [edge label'={$\begin{smallmatrix}1\end{smallmatrix}$}] (4);
\draw (4) to [edge label'={$\begin{smallmatrix}2\end{smallmatrix}$}] (5);
\end{tikzpicture}
\end{align*}

The equivalence $\Phi \colon \calW_U \to \mod B$ sends 
$X_3$ and $X_4$ to the simple $B$-modules $L(1)_B$ and $L(2)_B$
in $\mod B$ respectively by Proposition \ref{Prop_sim_W_U} (1).
Thus we can see that the brick labels are compatible with $\Phi$
as in Proposition \ref{Prop_reduc_label}.
Moreover \eqref{W_1 W_2} gives 
$\Phi(W_1)=L(2)_B$ and $\Phi(W_2)=P(1)_B$.

Let $\theta=x_1[P_1]+x_2[P_2]+x_3[P_3]+x_4[P_4] \in K_0(\proj A)_\R$.
Then since 
\begin{align*}
Y^+=\begin{smallmatrix}3\end{smallmatrix} \oplus
\begin{smallmatrix}1\\2\\3\\4\end{smallmatrix}, \quad
X^-=\begin{smallmatrix}4\end{smallmatrix},
\end{align*} 
the condition $\theta \in D^\circ(U)$ holds if and only if 
all the following inequalities are satisfied:
\begin{align*}
x_3 \ge 0,\quad
x_1+x_2+x_3+x_4 \ge 0,\quad
x_1+x_2+x_3 \ge 0,\quad
x_1+x_2 \ge 0,\quad
x_1 \ge 0,\quad
x_4 \le 0.
\end{align*}
Among them, $x_1+x_2+x_3 \ge 0$ is redundant,
because it is implied by $x_1+x_2 \ge 0$ and $x_3 \ge 0$.
The other five inequalities give the facets of $D(U)$.
By direct calculation, we have 
$D(U)=\R_{\ge 0}D$, 
where $D$ is the following pyramid in the affine hyperplane 
$\{x_1[P_1]+x_2[P_2]+x_3[P_3]+x_4[P_4] \mid 2x_1+x_2+x_3=0\}$.
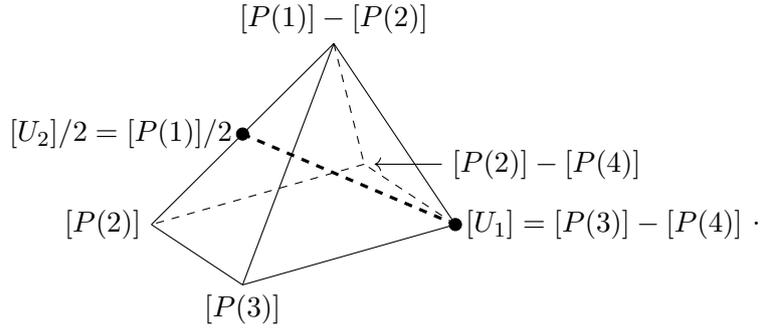
\begin{align*}
\begin{tikzpicture}[baseline=0pt,scale=0.8]
\node (F3)[coordinate,label= 90:{$[P(1)]-[P(2)]$}] at ( 0  , 3  ) {};
\node (F1)[coordinate,label=180:{$[P(2)]$}]        at (-3  , 0  ) {};
\node (F2)[coordinate,label=270:{$[P(3)]$}]        at (-1.5,-1  ) {};
\node (F5)[coordinate,label=  0:{$[U_1]=[P(3)]-[P(4)]$}] 
at ( 2  , 0  ) {};
\node (F4)[coordinate]                                      at ( 0.5, 1  ) {};
\node (U2)[coordinate,label=180:{$[U_2]/2 =[P(1)]/2$}]      at (-1.5, 1.5) {};
\node (F4')  at ( 0.5, 1  ) {};
\node (F4'') at ( 3.5, 1  ) {$[P(2)]-[P(4)]$};
\draw (F3) to (F1);
\draw (F3) to (F2);
\draw[dashed] (F3) to (F4);
\draw (F3) to (F5);
\draw (F1) to (F2);
\draw (F2) to (F5);
\draw[dashed] (F5) to (F4);
\draw[dashed] (F4) to (F1);
\draw[dashed,very thick] (U2) to (F5);
\draw[->] (F4'') to (F4');
\draw[fill=black] (F5) circle [radius=0.1];
\draw[fill=black] (U2) circle [radius=0.1];
\end{tikzpicture}.
\end{align*}
Note that $[U_2]$ itself is not on this pyramid.
The bold dashed line is the intersection of $C(U)$ and $D$.

Since $B=K(1 \to 2)$ is brick finite,
Proposition \ref{Prop_B_brick_fin} implies 
$D^\circ(U)=\bigcup_{V \in \twosilt_U A} C^\circ(V)$.
By $\twosilt_U A$ given above,
the wall-chamber structure on $D$ is depicted as follows.
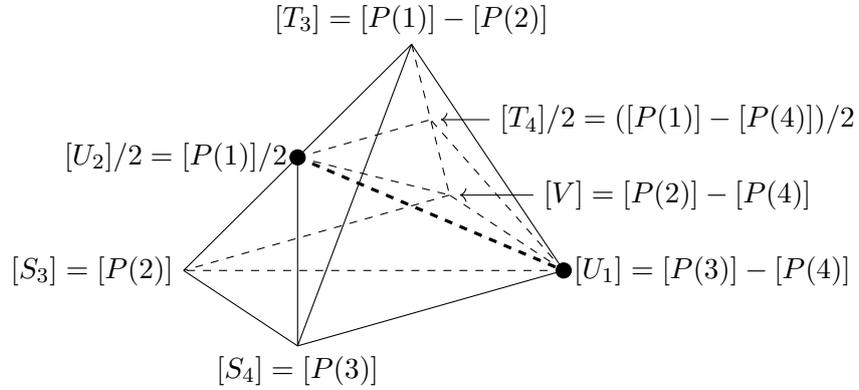
\begin{align*}
\begin{tikzpicture}[baseline=0pt]
\node (F3)[coordinate,label= 90:{$[T_3]=[P(1)]-[P(2)]$}] at ( 0  , 3  ) {};
\node (F1)[coordinate,label=180:{$[S_3]=[P(2)]$}]        at (-3  , 0  ) {};
\node (F2)[coordinate,label=270:{$[S_4]=[P(3)]$}]        at (-1.5,-1  ) {};
\node (F5)[coordinate,label=  0:{$[U_1]=[P(3)]-[P(4)]$}] at ( 2  , 0  ) {};
\node (F4)[coordinate]                                   at ( 0.5, 1  ) {};
\node (U2)[coordinate,label=180:{$[U_2]/2=[P(1)]/2$}]    at (-1.5, 1.5) {};
\node (F4')  at ( 0.5, 1  ) {};
\node (F4'') at ( 3.5, 1  ) {$[V]=[P(2)]-[P(4)]$};
\node (V) [coordinate]                                   at (0.25, 2  ) {};
\node (V')   at (0.25, 2  ) {};
\node (V'')  at ( 3.5, 2  ) {$[T_4]/2=([P(1)]-[P(4)])/2$};
\draw (F3) to (F1);
\draw (F3) to (F2);
\draw[dashed] (F3) to (F4);
\draw (F3) to (F5);
\draw (F1) to (F2);
\draw (F2) to (F5);
\draw[dashed] (F5) to (F4);
\draw[dashed] (F4) to (F1);
\draw[dashed] (F5) to (F1);
\draw[dashed] (F5) to (V);
\draw (U2) to (F2);
\draw[dashed] (U2) to (F4);
\draw[dashed] (U2) to (V);
\draw[dashed,very thick] (U2) to (F5);
\draw[->] (F4'') to (F4');
\draw[->] (V'')  to (V');
\draw[fill=black] (F5) circle [radius=0.1];
\draw[fill=black] (U2) circle [radius=0.1];
\end{tikzpicture}
\end{align*}

Compare this to the wall-chamber structure on $K_0(\proj B)_\R$ given in
the following, where the gray region is $\pi(D)$.
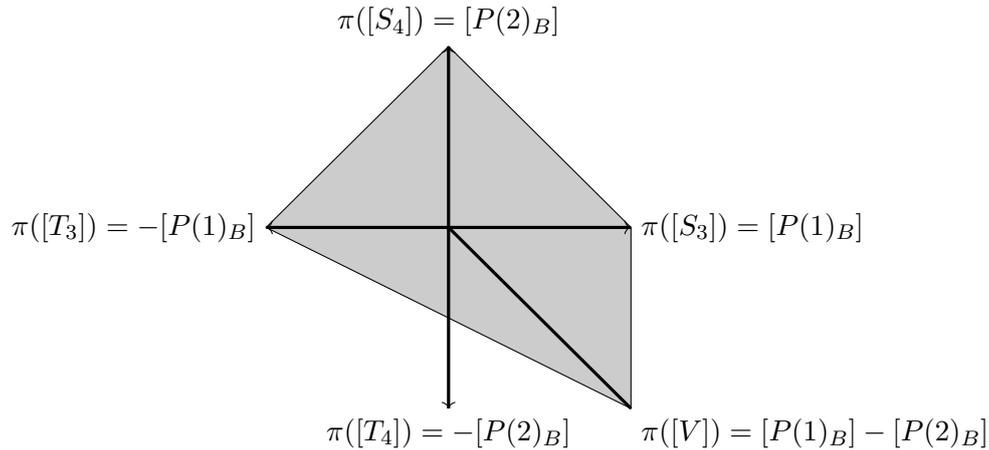
\begin{align*}
\begin{tikzpicture}[baseline=0pt,scale=0.8]
\node (00)[coordinate] at ( 0, 0) {};
\node (+0)[coordinate,label=  0:{$\pi([S_3])=[P(1)_B]$}]        at ( 3, 0) {};
\node (++)[coordinate]                                          at ( 3, 3) {};
\node (0+)[coordinate,label= 90:{$\pi([S_4])=[P(2)_B]$}]        at ( 0, 3) {};
\node (-+)[coordinate]                                          at (-3, 3) {};
\node (-0)[coordinate,label=180:{$\pi([T_3])=-[P(1)_B]$}]       at (-3, 0) {};
\node (--)[coordinate]                                          at (-3,-3) {};
\node (0-)[coordinate,label=270:{$\pi([T_4])=-[P(2)_B]$}]       at ( 0,-3) {};
\node (+-)[coordinate,label=315:{$\pi([V])=[P(1)_B]-[P(2)_B]$}] at ( 3,-3) {};
\draw[fill=black!20] (+0)--(0+)--(-0)--(+-)--cycle;
\draw[dashed,->] (00) to (+0);
\draw[dashed,->] (00) to (0+);
\draw[dashed,->] (00) to (-0);
\draw[dashed,->] (00) to (0-);
\draw[very thick] (00) to (+0);
\draw[very thick] (00) to (0+);
\draw[very thick] (00) to (-0);
\draw[very thick] (00) to (0-);
\draw[very thick] (00) to (+-);
\end{tikzpicture}
\end{align*}
\end{Ex}

\section{Facets of the interval neighborhood $D(U)$}\label{Sec_facet}

\subsection{Our strategy}

Let $U \in \twopsilt A$.
Recall that the closed interval neighborhood $D(U)$ is given by
\begin{align}\notag
D(U)&=
\{\theta \in K_0(\proj A)_\R \mid 
\calT_U \subset \ovcalT_\theta,\ 
\calF_U \subset \ovcalF_\theta\} \\ \label{overline DcU}
&=\{\theta \in K_0(\proj A)_\R \mid 
H^0(U) \in \ovcalT_\theta,\ 
H^{-1}(\nu U) \subset \ovcalF_\theta\}.
\end{align}
The aim of this section is to study the boundary $\partial D(U)$
and the set $\Facet D(U)$ of facets of $D(U)$.

By \eqref{overline DcU}, we have
\begin{align}\label{MM'}
D(U)=\bigcap_{M} \{\theta \in K_0(\proj A)_\R \mid \theta(M) \ge 0\}
\cap \bigcap_{M'} \{\theta \in K_0(\proj A)_\R \mid \theta(M') \le 0\},
\end{align}
where $M$ runs over all factor modules of $H^0(U)$ and $M'$ 
runs over all submodules of $H^{-1}(\nu U)$.
Thus for each facet $F \in \Facet D(U)$,
one of the following conditions holds.
\begin{enumerate2}
\item
An inner normal vector of $F$ in $D(U)$ is $[M] \in K_0(\mod A)$ 
for a factor module $M$ of $H^0(U)$.
\item
An outer normal vector of $F$ in $D(U)$ is $[M] \in K_0(\mod A)$ 
for a submodule $M'$ of $H^{-1}(\nu U)$.
\end{enumerate2}
However, not every factor module $M$ of $H^0(U)$
(resp.~not every submodule $M'$ of $H^{-1}(\nu U)$) gives 
a normal vector of a facet of $D(U)$ in general. 
In other words, 
it is inefficient to consider all factor modules of $H^0(U)$
and submodules of $H^{-1}(\nu U)$ to determine $D(U)$.

To reduce this redundancy,
we use the semibricks $Y^+=H^0(Y)$ and $X^-=H^{-1}(X)$
of the direct summands of
$X=\SH(S),\ Y=\SH(T) \in \twosmc A$,
which correspond to 
the maximal and the minimal completions $S,T \in \twosilt A$ of $U$.

Let $U=\bigoplus_{i=1}^m U_i$ with each $U_i$ indecomposable,
and apply Notations \ref{Nota_S_T} and \ref{Nota_X_Y}.
Then, by Propositions \ref{Prop_rad_soc_U} and \ref{Prop_T_U_T(Y^+)},
they satisfy
\begin{align*}
Y^+&=\bigoplus_{i=1}^m Y_i^+, &  
X^-&=\bigoplus_{i=1}^m X_i^-, \\
Y_i^+&=H^0(U_i)/\sum_{f \in \rad_A(H^0(U),H^0(U_i))} \Im f, & 
X_i^-&=\bigcap_{f \in \rad_A(H^{-1}(\nu U_i),H^{-1}(\nu U))} \Ker f, \\
\calT_U&=\sfT(Y^+), & \calF_U&=\sfF(X^-).
\end{align*}
Then $D(U)$ and $D^\circ(U)$ can be written with these (semi)bricks.

\begin{Prop}\label{Prop_D(U)_Y_X}
For any $U \in \twopsilt A$, we have
\begin{align*}
D(U)&=\{ \theta \in K_0(\proj A)_\R \mid 
Y^+ \in \ovcalT_\theta, \ X^- \in \ovcalF_\theta \}
=\bigcap_{i=1}^m \{ \theta \in K_0(\proj A)_\R \mid 
Y_i^+ \in \ovcalT_\theta, \ X_i^- \in \ovcalF_\theta \}, \\
D^\circ(U)&=\{ \theta \in K_0(\proj A)_\R \mid 
Y^+ \in \calT_\theta, \ X^- \in \calF_\theta \}
=\bigcap_{i=1}^m \{ \theta \in K_0(\proj A)_\R \mid 
Y_i^+ \in \calT_\theta, \ X_i^- \in \calF_\theta \}.
\end{align*}
\end{Prop}

\begin{proof}
We prove the equalities on $D(U)$. 
By Proposition \ref{Prop_T_U_T(Y^+)}, 
$\calT_U=\Fac H^0(U)=\sfT(Y^+)$ and 
$\calF_U=\Sub H^{-1}(\nu U)=\sfF(X^-)$ hold. 
Thus $\theta \in K_0(\proj A)_\R$ belongs to $D(U)$ if and only if 
$H^0(U) \in \ovcalT_\theta$ and $H^{-1}(\nu U) \in \ovcalF_\theta$ 
if and only if $Y^+ \in \ovcalT_\theta$ and $X^- \in \ovcalF_\theta$. 
Thus the left equality holds.
The right equality is immediate 
from $Y^+=\bigoplus_{i=1}^m Y_i^+$ and $X^-=\bigoplus_{i=1}^m X_i^-$.

A similar argument give the equalities on $D^\circ(U)$.
\end{proof}

Based on these descriptions, 
for each $i \in \{1,\ldots,m\}$ and $\epsilon \in \{\pm\}$, 
we set
\begin{align}\label{define partial}
\partial_i^+:=\{ \theta \in D(U) \mid Y_i^+ \notin \calT_\theta\}, \quad
\partial_i^-:=\{ \theta \in D(U) \mid X_i^- \notin \calF_\theta\}
\end{align}
and moreover
\begin{align}\label{define partial 2}
\partial_i:=\partial_i^+ \cup \partial_i^-, \quad
\partial^\epsilon:=\bigcup_{i=1}^m \partial_i^\epsilon.
\end{align}
Since $D^\circ(U)$ is open, we have 
\begin{align}\label{all partial_i union}
\partial D(U)=D(U)\setminus D^\circ(U) 
\stackrel{\text{Prop.~\ref{Prop_D(U)_Y_X}}}{=}
\bigcup_{i=1}^m \partial_i=
\bigcup_{i=1}^m (\partial_i^+\cup \partial_i^-).
\end{align}

By Proposition \ref{Prop_D(U)_Y_X}, 
every $\theta \in D(U)$ satisfies
$Y_i^+ \in \ovcalT_\theta$ and $X_i^- \in \ovcalF_\theta$ for all $i$.
Thus we have
\begin{align}\label{partial_i}
\begin{array}{l}
\partial_i^+=\{ \theta \in D(U) \mid 
\text{there exists a factor module $M \ne 0$ of $Y_i^+$
such that $\theta(M)=0$} \}, \\
\partial_i^-=\{ \theta \in D(U) \mid 
\text{there exists a submodule $M' \ne 0$ of $X_i^-$
such that $\theta(M')=0$} \}.
\end{array}
\end{align}
As a consequence, we have the following.

\begin{Lem}\label{Lem_partial_face}
Let $U=\bigoplus_{i=1}^m U_i \in \twopsilt A$.
\begin{enumerate}
\item 
Let $i \in \{1,\ldots,m\}$ and $\epsilon \in \{\pm\}$. 
Then $\partial_i^\epsilon$ is a union of faces of $D(U)$.
\item 
For each $F \in \Face D(U)$ with $F \ne D(U)$, 
there exists $(i,\epsilon) \in \{1,\ldots,m\} \times \{\pm\}$ 
such that $F\subset\partial_i^\epsilon$.
\item
Let $F \in \Face D(U)$. 
\begin{enumerate}
\item
If $i \in \{1,\ldots,m\}$ satisfies $F \subset \partial_i$,
then there exists $\epsilon \in \{\pm\}$ such that
$F \subset \partial_i^\epsilon$.
\item
If $\epsilon \in \{\pm\}$ satisfies $F \subset \partial^\epsilon$,
then there exists $i \in \{1,\ldots,m\}$ such that
$F \subset \partial_i^\epsilon$.
\end{enumerate}
\end{enumerate}
\end{Lem}

\begin{proof}
(1) follows from \eqref{partial_i},
(2) comes from \eqref{all partial_i union} and (1), and 
(3) comes from \eqref{define partial 2} and (1).
\end{proof}

In this context, 
for each $i \in \{1,\ldots,m\}$ and $\epsilon \in \{\pm\}$, we set
\begin{align*}
\Face_i^\epsilon D(U)
&:=\{F \in \Face D(U) \mid F \subset \partial_i^\epsilon \},\\
\Face_i D(U)
&:=\{F \in \Face D(U) \mid F \subset \partial_i \}, &
\Face^\epsilon D(U)
&:=\{F \in \Face D(U) \mid F \subset \partial^\epsilon \},\\
\Facet_i^\epsilon D(U)
&:=\{F \in \Facet D(U) \mid F \subset \partial_i^\epsilon \},\\
\Facet_i D(U)
&:=\{F \in \Facet D(U) \mid F \subset \partial_i \}, &
\Facet^\epsilon D(U)
&:=\{F \in \Facet D(U) \mid F \subset \partial^\epsilon \}.
\end{align*}
Then Lemma \ref{Lem_partial_face} (3) implies
\begin{align*}
\Face_i D(U)&=\Face_i^+ D(U) \cup \Face_i^- D(U), &
\Face^\epsilon D(U)&=\bigcup_{i=1}^m \Face_i^\epsilon D(U), \\
\Facet_i D(U)&=\Facet_i^+ D(U) \cup \Facet_i^- D(U), &
\Facet^\epsilon D(U)&=\bigcup_{i=1}^m \Facet_i^\epsilon D(U).
\end{align*}

By Lemma \ref{Lem_partial_face} (1), $\partial_i^\epsilon$ is 
the union of all faces $F \in \Face_i^\epsilon D(U)$.
More strongly, we will show that $\partial_i^\epsilon$ is 
the union of all facets $F \in \Facet_i^\epsilon D(U)$
in Theorem \ref{Thm_partial_facet} of the next subsection.

We caution the following on these symbols.

\begin{Rem}
By definition, we have
$\partial_i^\epsilon \subset \partial_i \cap \partial^\epsilon$, and hence
$\Face_i^\epsilon D(U) \subset \Face_i D(U) \cap \Face^\epsilon D(U)$.
However, these inclusions can be proper;
see \eqref{Eq_d1_d+_proper_d1+} in Example \ref{Ex_A4_L_F}.
Nevertheless,
$\Facet_i^\epsilon D(U)=\Facet_i D(U) \cap \Facet^\epsilon D(U)$
always holds by Theorem \ref{Thm_partial_facet}.
\end{Rem}

The property below more directly follows from \eqref{MM'}.

\begin{Lem}\label{Lem_setminus_face}
Let $U=\bigoplus_{i=1}^m U_i \in \twopsilt A$.
For any $i \in \{1,\ldots,m\}$,
the set $D(U) \setminus D^\circ(U_i)$ is a union of faces of $D(U)$.
\end{Lem}

\begin{proof}
Let $\theta \in D(U)$.
Then we have
$H^0(U_i) \in \ovcalT_\theta$ and $H^{-1}(\nu U_i) \in \ovcalF_\theta$.
Thus $\theta \in D(U) \setminus D^\circ(U_i)$ is equivalent to
that at least one of the following holds.
\begin{enumerate2}
\item
There exists a factor module $M \ne 0$ of $H^0(U_i)$ such that 
$\theta(M)=0$.
\item
There exists a submodule $M' \ne 0$ of $H^{-1}(\nu U_i)$ such that 
$\theta(M)=0$.
\end{enumerate2}
Clearly, every factor module of $H^0(U_i)$ is a factor module of $H^0(U)$,
and every submodule of $H^{-1}(\nu U_i)$ is a submodule of $H^{-1}(\nu U)$.
Thus \eqref{MM'} gives that
$D(U) \setminus D^\circ(U_i)$ is a union of faces of $D(U)$.
\end{proof}

It is easy to see $\partial_i \subset D(U) \setminus D^\circ(U_i)$
as in Lemma \ref{Lem_partial_setminus_oneside}.
The converse $D(U) \setminus D^\circ(U_i) \subset \partial_i$
is far from trivial, 
so the next subsection is mainly devoted to proving this part.

\subsection{A decomposition theorem of the set of facets}

This subsection has two main results.
The first one is a description of $\partial_i$ only by interval neighborhoods.

\begin{Thm}\label{Thm_partial_setminus}
Let $U=\bigoplus_{i=1}^m U_i \in \twopsilt A$. 
For each $i \in \{1,\ldots,m\}$, we have the equality
\begin{align*}
\partial_i&=D(U) \setminus D^\circ(U_i).
\end{align*}
\end{Thm}

The other one means the purities of $\partial_i^\epsilon$.

\begin{Thm}\label{Thm_partial_facet}
Let $U=\bigoplus_{i=1}^m U_i \in \twopsilt A$.
Then the following statements hold.
\begin{enumerate}
\item
For each $i \in \{1,\ldots,m\}$ and $\epsilon \in \{\pm\}$,
the subsets $\partial_i^\epsilon,\partial_i,\partial^\epsilon$ 
are union of facets; 
hence we have
\begin{align*}
\partial_i^\epsilon=\bigcup_{F \in \Facet_i^\epsilon D(U)} F, \quad
\partial_i=\bigcup_{F \in \Facet_i D(U)} F, \quad
\partial^\epsilon=\bigcup_{F \in \Facet^\epsilon D(U)} F.
\end{align*}
Moreover $\Facet_i D(U) \ne \emptyset$ holds.
\item
For each facet $F$ of $D(U)$,
there uniquely exists $(i,\epsilon) \in \{1,\ldots,m\} \times \{\pm\}$
such that $F \in \Facet_i^\epsilon D(U)$.
Thus we have a decomposition
\begin{align*}
\Facet D(U)
=\bigsqcup_{i=1}^m(\Facet_i^+ D(U) \sqcup \Facet_i^- D(U)).
\end{align*}
In particular, we have
\begin{align*}
\Facet_i^\epsilon D(U)=\Facet_i D(U) \cap \Facet^i D(U).
\end{align*}
\end{enumerate}
\end{Thm}

We first aim to prove Theorem \ref{Thm_partial_setminus}.
The inclusion $\partial_i \subset D(U) \setminus D^\circ(U_i)$
is easily obtained.

\begin{Lem}\label{Lem_partial_setminus_oneside}
Let $U=\bigoplus_{i=1}^m U_i \in \twopsilt A$ and $i \in \{1,\ldots,m\}$.
Then we have $\partial_i \subset D(U) \setminus D^\circ(U_i)$.
\end{Lem}

\begin{proof}
Let $\theta \in \partial_i$.
Then \eqref{define partial} gives $Y_i^+ \notin \calT_\theta$ 
or $X_i^- \notin \calF_\theta$.
Since $Y_i^+$ is a factor module of $H^0(U_i)$
and $X_i^-$ is a submodule of $H^{-1}(\nu U_i)$
by Proposition \ref{Prop_rad_soc_U},
we have $\theta \notin D^\circ(U_i)$.
Then we have the assertion, since $\theta \in \partial_i \subset D(U)$.
\end{proof}

To show the converse $D(U) \setminus D^\circ(U_i) \subset \partial_i$,
fortunately, we already know that both sides are union of faces
by Lemmas \ref{Lem_partial_face} and \ref{Lem_setminus_face}.
Thus it is enough to show the following characterization of $\Face_i D(U)$.

\begin{Prop}\label{Prop_facet_between}
Let $U=\bigoplus_{i=1}^m U_i \in \twopsilt A$. 
For each $i \in \{1,\ldots,m\}$ and $F \in \Face D(U)$,
the following conditions are equivalent.
\begin{enumerate2}
\item
We have $F \subset \partial_i$;
that is, $F \in \Face_i D(U)$.
\item
The inclusion $F \subset D(U) \setminus D^\circ(U_i)$ holds.
\item
We have $[U_i] \notin F$.
\item
There exists $F' \in \Facet D(U)$ such that 
$F \subset F'$ and $[U_i] \notin F'$.
\item
There exist $F' \in \Facet D(U)$ and $\epsilon \in \{\pm\}$ such that 
$F \subset F' \subset \partial_i^\epsilon$.
\end{enumerate2}
\end{Prop}

We have (a)$\Rightarrow$(b) from Lemma \ref{Lem_partial_setminus_oneside}.
Moreover the implications 
(e)$\Rightarrow$(a) and (b)$\Rightarrow$(c)$\Rightarrow$(d)
are almost clear.
Thus we focus on showing the nontrivial part (d)$\Rightarrow$(e).
For this purpose,
we see $D^\circ(U)$ as an $M$-TF equivalence class for $M:=Y^+ \oplus X^-$.

\begin{Lem}\label{Lem_D(U)_M-TF}
Let $U \in \twopsilt A$ and set $M:=Y^+ \oplus X^-$.
Then $D^\circ(U)$ is a full-dimensional $M$-TF equivalence class 
with $\rmt_{D^\circ(U)} M=Y^+$ and $\rmf_{D^\circ(U)} M=X^-$.
\end{Lem}

\begin{proof}
Take $\eta \in D^\circ(U)$.
Then $Y^+ \in \calT_\eta$ and $X^- \in \calF_\eta$ 
by Proposition \ref{Prop_D(U)_Y_X}.
Thus $\rmt_\eta M=Y^+$, $\rmw_\eta M=0$ and $\rmf_\eta M=X^-$,
so Lemma \ref{Lem_E_twf} (1) gives
\begin{align*}
[\eta]_M =\{\theta \in K_0(\proj A)_\R \mid 
Y^+\in \calT_\theta, \ X^- \in \calF_\theta\}.
\end{align*} 
The right-hand side is $D^\circ(U)$ again by Proposition \ref{Prop_D(U)_Y_X},
so we have $[\eta]_M=D^\circ(U) \in \TF_n(M)$. 
Now $\rmt_{D^\circ(U)} M=Y^+$ and $\rmf_{D^\circ(U)} M=X^-$ are clear.
\end{proof}

Thus the following properties come 
from basic properties of $M$-TF equivalences.

\begin{Lem}\label{Lem_face_M-TF}
Let $U=\bigoplus_{i=1}^m U_i \in \twopsilt A$.
Set $M:=Y^+ \oplus X^-$.
\begin{enumerate}
\item
For any $F \in \Face D(U)$, 
we have $F^\circ \in \TF(M)$ and $F \in \Sigma(M)$.
\item
Let $i \in \{1,\ldots,m\}$, $F \in \Face D(U)$ and $\theta \in F^\circ$.
Then none of
\begin{align*}
\rmt_\theta Y_i^+, \ \rmw_\theta Y_i^+, \ \supp_\theta(\rmw_\theta Y_i^+), \ 
\supp_\theta(\rmw_\theta X_i^-), \ \rmw_\theta X_i^-, \ \rmf_\theta X_i^-. 
\end{align*}
depends on the choice of $\theta \in F^\circ$.
Moreover $\rmf_\theta Y_i^+=0$ and $\rmt_\theta X_i^-=0$ hold.
\end{enumerate}
\end{Lem}

\begin{proof}
(1) follows from Lemma \ref{Lem_D(U)_M-TF} and 
Proposition \ref{Prop_M-TF_face} (1).

(2)
By (1), we have $F^\circ \in \TF(M)$.
Since $Y_i^+$ is a direct summand of $M$,
we get $F^\circ \subset \TF(Y_i^+)$
by Remark \ref{Rem_M-TF_oplus}.
Thus $\rmt_\theta Y_i^+$, $\rmw_\theta Y_i^+$, 
$\supp_\theta(\rmw_\theta Y_i^+)$ are constant
for any $\theta \in F^\circ$.
Moreover $\rmf_\theta Y_i^+=0$, 
since $Y_i^+ \in \ovcalT_\theta$ for any $\theta \in F^\circ$.
The remaining parts are dually obtained.
\end{proof}

Thus for each $F \in \Face D(U)$ and $i \in \{1,\ldots,m\}$, 
we write
\begin{align*}
\rmt_F Y_i^+, \ \rmw_F Y_i^+, \ \supp_F(\rmw_F Y_i^+), \ 
\supp_F(\rmw_F X_i^-), \ \rmw_F X_i^-, \ \rmf_F X_i^-
\end{align*}
for the six sets appearing in Lemma \ref{Lem_face_M-TF} (2) respectively.
Then we have the short exact sequences
\begin{align}\label{w_theta F}
0 \to \rmt_F Y_i^+ \to Y_i^+ \to \rmw_F Y_i^+ \to 0, \quad
0 \to \rmw_F X_i^- \to X_i^- \to \rmf_F X_i^- \to 0.
\end{align}

The sets $\Face_i^\epsilon D(U)$ are characterized by these exact sequences.

\begin{Lem}\label{Lem_w_F_nonzero}
Let $U=\bigoplus_{i=1}^m U_i \in \twopsilt A$, 
$i \in \{1,\ldots,m\}$
and $F \in \Face D(U)$.
Then
$F \in \Face_i^+ D(U)$ holds if and only if $\rmw_F Y_i^+ \ne 0$,
and $F \in \Face_i^- D(U)$ holds if and only if $\rmw_F X_i^- \ne 0$.
\end{Lem}

\begin{proof}
We prove the first statement only.
The condition $\rmw_F Y_i^+ \ne 0$ is equivalent to that
any $\theta \in F^\circ$ satisfies $Y_i^+ \notin \calT_\theta$
by the definition of the first exact sequence in \eqref{w_theta F}.
The latter is nothing but $F^\circ \subset \partial_i^+$.
Since $F$ is closed, $F^\circ \subset \partial_i^+$ is equivalent to 
$F \subset \partial_i^+$.
\end{proof}

To investigate the exact sequences \eqref{w_theta F} more, 
the following property is crucial.

\begin{Lem}\label{Lem_Y_W_almost}
Let $U=\bigoplus_{i=1}^m U_i \in \twopsilt A$. 
For each $i \in \{1,\ldots,m\}$,
we have $Y_i^+ \in \calT_{U_i} \cap \calW_{U/U_i}$ and
$X_i^- \in \calF_{U_i} \cap \calW_{U/U_i}$.
\end{Lem}

\begin{proof}
We only show $Y_i^+ \in \calT_{U_i} \cap \calW_{U/U_i}$.
Since $Y_i^+$ is a factor module of $H^0(U_i) \in \calT_{U_i}$ 
by Proposition \ref{Prop_rad_soc_U}, we have $Y_i^+ \in \calT_{U_i}$.
This and Lemma \ref{Lem_T_U_T_V} (2) imply 
$Y_i^+ \in \calT_{U_i} \subset \ovcalT_{U/U_i}$.
On the other hand, Definition-Proposition \ref{Def-Prop_SH} gives
$\Hom_{\sfD(A)}(U/U_i,Y_i^+)=0$.
Thus $\Hom_A(H^0(U/U_i),Y_i^+)=0$ holds, 
which means $Y_i^+ \in \ovcalF_{U/U_i}$.
Hence we obtain $Y_i^+ \in \calW_{U/U_i}$.
\end{proof}

Then we observe the following relationship 
between $C(U/U_i)$ and $\partial_i^\pm$.

\begin{Prop}\label{Prop_partial_nonempty}
Let $U=\bigoplus_{i=1}^m U_i \in \twopsilt A$
and $i \in \{1,\ldots,m\}$.
\begin{enumerate}
\item
The set $\partial_i^+$ is nonempty if and only if $Y_i^+ \neq 0$.
In this case, $C(U/U_i) \subset \partial_i^+$ holds.
\item
The set $\partial_i^-$ is nonempty if and only if $X_i^- \neq 0$.
In this case, $C(U/U_i) \subset \partial_i^-$ holds.
\item
The set $\partial_i$ is nonempty, and $C(U/U_i) \subset \partial_i$ holds.
\end{enumerate}
\end{Prop}

\begin{proof}
(1)
The ``only if'' part is obvious.
We prove the ``if'' part. 
For each $\theta \in C^\circ(U/U_i) \subset D(U)$, 
by Lemma \ref{Lem_Y_W_almost} and Proposition \ref{Prop_cone_TF} (1), 
we have $Y_i^+\in\calW_{U/U_i}=\calW_\theta$ and hence $\theta(Y_i^+)=0$. 
Thus if $Y_i^+\ne 0$, 
then $\theta \in \partial_i^+$ by \eqref{define partial}. 
Thus $C^\circ(U/U_i) \subset \partial_i^+$ holds,
and since $\partial_i^+$ is closed by Lemma \ref{Lem_partial_face} (1),
we get $C(U/U_i) \subset \partial_i^+$.
In particular, $\partial_i^+\neq\emptyset$ holds.

(2) is dual to (1).

(3) follows from (1), (2) and Proposition \ref{Prop_X_i^-_Y_i^+} (1).
\end{proof}

This and the short exact sequences \eqref{w_theta F} 
give the following crucial property of $\partial_i^\pm$.

\begin{Lem}\label{Lem_F+C_almost}
Let $U=\bigoplus_{i=1}^m U_i \in \twopsilt A$ 
and $F \in \Face_i^\epsilon D(U)$ with
$(i,\epsilon) \in \{1,\ldots,m\} \times \{\pm\}$.
\begin{enumerate}
\item 
All terms of the short exact sequences \eqref{w_theta F} 
belong to $\calW_{U/U_i}$.
\item
We have $F+C(U/U_i)\subset\partial_i^\epsilon$.
\end{enumerate}
\end{Lem}

\begin{proof}
We prove the assertions for $Y_i^+$, since those for $X_i^-$ are the duals.

(1) 
By Lemma \ref{Lem_Y_W_almost}, we have $Y_i^+\in\calW_{U/U_i}$. 
Since $\rmt_F Y_i^+$ is a submodule of $Y_i^+ \in \calW_{U/U_i}$, 
we have $\rmt_F Y_i^+ \in \ovcalF_{U/U_i}$. 
On the other hand, taking $\theta \in F^\circ$, we get
$\rmt_F Y_i^+ =\rmt_\theta Y_i^+ \in \calT_\theta \subset \ovcalT_U 
\subset \ovcalT_{U/U_i}$.
Thus $\rmt_F Y_i^+ \in \calW_{U/U_i}$ holds. 
This implies $\rmw_F Y_i^+ \in \calW_{U/U_i}$,
since $\calW_{U/U_i}$ is a wide subcategory.

(2) 
Let $\theta \in F^\circ$ and $\eta \in C(U/U_i)$.
Then $\rmw_F Y_i^+=\rmw_\theta Y_i^+ \in \calW_\theta$ holds.
By (1) and Proposition \ref{Prop_cone_TF} (1), 
we have $\rmw_F Y_i^+ \in \calW_{U/U_i}=\calW_\eta$.
Thus we obtain $\rmw_F Y_i^+ \in \calW_\theta \cap \calW_\eta \subset 
\calW_{\theta+\eta}$.
Since $F \subset \partial_i^+$, we have $\rmw_F Y_i^+ \ne 0$
by Lemma \ref{Lem_w_F_nonzero}.
Thus $\rmw_F Y_i^+$ is a nonzero factor module of $Y_i^+$
which belongs to $\calW_{\theta+\eta}$.
Then $\theta+\eta \in \partial_i^+$ follows from \eqref{partial_i}. 
\end{proof}

Then we have the following property on maximal faces 
contained in $\partial_i^\epsilon$.
For any subset $\frakF \subset \Face D(U)$,
we write $\max \frakF$ for the set of maximal elements in $\frakF$.

\begin{Prop}\label{Prop_C_almost_F}
Let $U=\bigoplus_{i=1}^m U_i \in \twopsilt A$ 
and $(i,\epsilon) \in \{1,\ldots,m\} \times \{\pm\}$.
If $F \in \max \Face_i^\epsilon D(U)$, then we have
\begin{align*}
C(U/U_i) \subset F, \quad [U_i] \notin F.
\end{align*}
\end{Prop}

\begin{proof}
Lemma \ref{Lem_Y_W_almost} and Proposition \ref{Prop_cone_TF} (1) imply
$Y_i^+ \in \calT_{U_i}=\calT_{[U_i]}$.
Thus we have $[U_i] \notin \partial_i^+\supset F$ and hence $[U_i] \notin F$.

By Lemma \ref{Lem_F+C_almost} (2), 
we have $F+C(U/U_i) \subset \partial_i^\epsilon$.
The set $F+C(U/U_i)$ is clearly convex,
and $\partial_i^\epsilon$ is a union of faces of $D(U)$
by Lemma \ref{Lem_partial_face} (1).
Thus we obtain $F' \subset \partial_i^\epsilon$, 
where $F'$ is the smallest face containing $F+C(U/U_i)$.
We have proved
$F \subset F+C(U/U_i) \subset F' \subset \partial_i^\epsilon$.
By maximality, $F+C(U/U_i)=F$ holds. Thus $C(U/U_i) \subset F$.
\end{proof}

Now we are able to show Proposition \ref{Prop_facet_between}
and Theorem \ref{Thm_partial_setminus}.

\begin{proof}[Proof of Proposition \ref{Prop_facet_between}]
(a)$\Rightarrow$(b)
follows from Lemma \ref{Lem_partial_setminus_oneside}.

(b)$\Rightarrow$(c) is true, because $[U_i] \in D^\circ(U_i)$.

(c)$\Rightarrow$(d) 
follows from that $F$ is the intersection of all facets containing $F$.

(d)$\Rightarrow$(e)
Take $F' \in \Facet D(U)$ in (d).
By Lemma \ref{Lem_partial_face} (2),
we also take $(i',\epsilon) \in \{1,\ldots,m\} \times \{\pm\}$ 
such that $F' \subset \partial_{i'}^\epsilon$.
Since $F'$ is a facet, we have $F' \in \max \Face_{i'}^\epsilon D(U)$.
Then $C(U/U_{i'}) \subset F'$ follows 
from Proposition \ref{Prop_C_almost_F}.
On the other hand, the choice of $F'$ gives $[U_i] \notin F$.
Thus $i=i'$ must hold, and we obtain $F' \subset \partial_i^\epsilon$.

(e)$\Rightarrow$(a) is obvious.
\end{proof}

\begin{proof}[Proof of Theorem \ref{Thm_partial_setminus}]
The ``$\subset$'' part is Lemma \ref{Lem_partial_setminus_oneside}.

It remains to get the ``$\supset$'' part.
By Proposition \ref{Prop_facet_between} (b)$\Rightarrow$(a),
if $F \in \Face D(U)$ satisfies $F \subset D(U) \setminus D^\circ(U_i)$, 
then $F \subset \partial_i$ holds.
This and Lemma \ref{Lem_setminus_face} imply
$D(U) \setminus D^\circ(U_i) \subset \partial_i$.
\end{proof}

We next move to prove Theorem \ref{Thm_partial_facet}.
We begin with the following property.

\begin{Lem}\label{Lem_partial^+-_no_facet}
Let $U \in \twopsilt A$.
Then
the intersection $\partial^+ \cap \partial^-$ contains no facets of $D(U)$.
\end{Lem}

\begin{proof}
Set $M:=Y^+ \oplus X^-$.
By Lemma \ref{Lem_D(U)_M-TF}, we have $D^\circ(U) \in \TF_n(M)$ 
with $\rmt_{D^\circ(U)} M=Y^+$ and $\rmf_{D^\circ(U)} M=X^-$.
Thus Proposition \ref{Prop_M-TF_+-_no_facet} implies the assertion.
\end{proof}

Then we have $\Facet_i^\epsilon D(U)$ are disjoint.

\begin{Prop}\label{Prop_Facet_disjoint}
Let $U=\bigoplus_{i=1}^m U_i \in \twopsilt A$.
For any distinct pairs 
$(i,\epsilon) \ne (i',\epsilon') \in \{1,\ldots,m\} \times \{\pm\}$,
we have
$\Facet_i^\epsilon D(U) \cap \Facet_{i'}^{\epsilon'} D(U)
=\emptyset$.
\end{Prop}

\begin{proof}
Let $F \in \Facet_i^\epsilon D(U) \cap \Facet_{i'}^{\epsilon'} D(U)$.
By assumption, $i \ne i'$ or $\epsilon \ne \epsilon'$ holds.

If $i \ne i'$, then 
Proposition \ref{Prop_C_almost_F} implies
$[U_i] \notin F$ and $[U_{i'}] \in C(U/U_i) \subset F$,
a contradiction.

If $\epsilon \ne \epsilon'$, 
then the facet $F$ is contained in $\partial^+ \cap \partial^-$,
which contradicts Lemma \ref{Lem_partial^+-_no_facet}.
\end{proof}

Now we are ready to prove Theorem \ref{Thm_partial_facet}.

\begin{proof}[Proof of Theorem \ref{Thm_partial_facet}]
(1)
By Proposition \ref{Prop_facet_between} (a)$\Rightarrow$(e), 
$\partial_i^\epsilon$ is a union of facets,
and so are $\partial_i$ and $\partial^\epsilon$.
Then the displayed equalities are clear.
Since $\partial_i \ne \emptyset$ 
by Proposition \ref{Prop_partial_nonempty} (3),
we get $\Facet_i D(U) \ne \emptyset$.

(2)
The ``$\supset$'' part is clear,
and the ``$\subset$'' part is Lemma \ref{Lem_partial_face} (2).
By Proposition \ref{Prop_Facet_disjoint},
the right-hand side is a disjoint union.
\end{proof}

\subsection{The brick labeling of facets}

For each $F \in \Facet D(U)$,
we write $(i_F,\epsilon_F)$ for the unique pair $(i,\epsilon)$ such that 
$F \in \Facet_i^\epsilon D(U)$ by Theorem \ref{Thm_partial_facet} (2).
Then we define a nonzero module $L_F$ by Lemma \ref{Lem_w_F_nonzero}
as follows.

\begin{Def}\label{Def_L_F}
Let $U=\bigoplus_{i=1}^m U_i \in \twopsilt A$.
For each facet $F \in \Facet D(U)$, we set 
\begin{align*}
L_F:=\begin{cases}
\rmw_F Y_{i_F}^+ & (\epsilon_F={+}) \\
\rmw_F X_{i_F}^- & (\epsilon_F={-})
\end{cases}.
\end{align*}
\end{Def}

In this subsection, 
we show that $L_F$ satisfies the following nice properties.
Thus we call $L_F$ the \emph{brick label} of each facet $F$.
Recall that we defined the division algebra $R_V$ and 
a positive integer $d_V \in \Z_{\ge 1}$ by
\begin{align*}
R_V:=\End_{\sfD(A)}(V)/{\rad\End_{\sfD(A)}(V)}, \quad
d_V:=\dim_K R_V
\end{align*}
for any indecomposable $V \in \twopsilt A$
and before Proposition \ref{Prop_dual_basis}.

\begin{Thm}\label{Thm_L_F_inn_vec}
Let $U=\bigoplus_{i=1}^m U_i \in \twopsilt A$
and $F \in \Facet_i^\epsilon D(U)$
with $(i,\epsilon) \in \{1,\ldots,m\} \times \{\pm\}$.
\begin{enumerate}
\item
The module $L_F$ is a simple object of $\calW_\theta$ 
for any $\theta \in F^\circ$, and $\End_A(L_F) \simeq R_{U_i}$ holds.
\item
The element $\epsilon[L_F] \in K_0(\mod A)$ is 
an inner normal vector of $F$ such that
\begin{align*}
F=\{ \theta \in D(U) \mid \theta(L_F)=0\}.
\end{align*}
\item
For each $i' \in \{1,\ldots,m\}$, we have
$[U_{i'}](L_F)=\epsilon \cdot \delta_{i,i'}d_{U_i}$.
\end{enumerate}
\end{Thm}

It is immediate that $D(U)$ can be described by the bricks $L_F$.

\begin{Cor}\label{Cor_D(U)_ineq}
Let $U=\bigoplus_{i=1}^m U_i \in \twopsilt A$. 
Then we have
\begin{align*}
D(U)=\bigcap_{F \in \Facet D(U)}
\{ \theta \in K_0(\proj A)_\R \mid \epsilon_F \cdot\theta(L_F) \ge 0 \}.
\end{align*}
\end{Cor}

We first show that $L_F$ is a brick by determining 
the endomorphism algebra of $L_F$.

\begin{Prop}\label{Prop_End_L_F}
Let $U=\bigoplus_{i=1}^m U_i \in \twopsilt A$
and $F \in \Facet_i D(U)$ with $i \in \{1,\ldots,m\}$. 
\begin{enumerate}
\item
The module $L_F$ is a brick with an isomorphism
$R_{U_i} \to \End_A(L_F)$ of division algebras.
\item
For any $\theta \in K_0(\proj A)$, we have $\theta(L_F)=d_{U_i}\Z$.
\end{enumerate}
\end{Prop}

\begin{proof}
(1)
We prove in the case $\epsilon_F={+}$.
The other case is dual.
Set $L:=L_F$.

By Lemma \ref{Lem_R_V} and $Y_i^+ \ne 0$,
we have an isomorphism $R_{U_i} \to \End_A(Y_i^+)$.
Thus it suffices to construct an isomorphism $\End_A(Y_i^+) \to \End_A(L)$.

By Proposition \ref{Prop_rad_soc_U}, 
we have a short exact sequence
$0 \to N \to H^0(U_i) \to Y_i^+ \to 0$, where 
\begin{align*}
N:=\sum_{f \in \rad_A(H^0(U),H^0(U_i))} \Im f.
\end{align*}
Then $h(N) \subset N$ holds for any $h \in \End_A(H^0(U_i))$.
Thus we have a canonical homomorphism 
$\psi_1 \colon \End_A(H^0(U_i)) \to \End_A(Y_i^+)$.

We also have a short exact sequence
$0 \to \rmt_F Y_i^+ \to Y_i^+ \to L \to 0$.
Since $h(\rmt_F Y_i^+) \subset \rmt_F Y_i^+$ 
for any $h \in \End_A(Y_i^+)$,
we have a canonical homomorphism $\psi_2 \colon \End_A(Y_i^+) \to \End_A(L)$.

Now we consider the composite
$\psi_2 \psi_1 \colon \End_A(H^0(U_i)) \to \End_A(L)$.
This coincides with the canonical homomorphism
induced by the surjection $H^0(U_i) \to L$.
We show that $\psi_2 \psi_1$ is surjective.

We set $N'$ as the kernel of the canonical surjection $H^0(U_i) \to L$.
Then there exists an exact sequence $0 \to N' \to H^0(U_i) \to L \to 0$.
To get the surjectivity of $\psi_2 \psi_1$,
it suffices to show $\Ext_A^1(H^0(U_i),N')=0$.

By the definition of $N'$, we have another short exact sequence 
$0 \to N \to N' \to \rmt_F Y_i^+ \to 0$.
Since $N \in \calT_U \subset \ovcalT_U \subset \ovcalT_{U_i}$
and $\rmt_F Y_i^+ \in \calT_\theta \subset \ovcalT_U 
\subset \ovcalT_{U_i}$
for any $\theta \in F^\circ \subset D(U)$,
we have $N' \in \ovcalT_{U_i}$.
Thus we get $\Hom_A(N',H^{-1}(\nu U_i))=0$; hence $\Ext_A^1(H^0(U_i),N')=0$.
Thus $\psi_2\psi_1$ is surjective, and so is $\psi_2$.

Since $\rmw_F Y_i^+=L \ne 0$, we have $Y_i^+ \ne 0$.
Thus $Y_i^+$ is a brick and $\End_A(Y_i^+)$ is a division algebra.
By this and $L \ne 0$,
the surjective homomorphism $\psi_2$ of algebras
is actually an isomorphism.
By the second paragraph, we obtain the assertion.

(2)
Let $P \in \proj A$.
Then we have $[P](L_F)=\dim_K \Hom_A(P,L_F)$ by \eqref{Euler U X}.
By (1), $\End_A(L_F)$ is a division algebra,
so $\Hom_A(P,L_F)$ is a free $\End_A(L_F)$-module.
Thus (1) implies $\dim_K \Hom_A(P,L_F) \in d_{U_i}\Z$,
and $[P](L_F) \in d_{U_i}\Z$ follows.
Then the assertion is immediate.
\end{proof}

We also observe the following ``all-but-one'' properties of facets.

\begin{Prop}\label{Prop_facet_X_Y_almost}
Let $U=\bigoplus_{i=1}^m U_i \in \twopsilt A$
and $F \in \Facet_i^\epsilon D(U)$
with $(i,\epsilon) \in \{1,\ldots,m\} \times \{\pm\}$.
\begin{enumerate}
\item
We have $C(U/U_i) \subset F$, $[U_i] \notin F$ and $F \cap C(U)=C(U/U_i)$.
\item
For any $\theta \in F^\circ$, we have
\begin{align*}
\begin{cases}
Y_i^+ \notin \calT_\theta, \ Y^+/Y_i^+ \in \calT_\theta, \ 
X^- \in \calF_\theta &
(\epsilon=+) \\
X_i^+ \notin \calF_\theta, \ X^+/X_i^- \in \calF_\theta, \ 
Y^- \in \calT_\theta &
(\epsilon=-)
\end{cases}.
\end{align*}
\item
If $\epsilon=+$, then $(\rmw_F Y^+,\rmw_F X^-)=(L_F,0)$.
If $\epsilon=-$, then $(\rmw_F Y^+,\rmw_F X^-)=(0,L_F)$.
\end{enumerate}
\end{Prop}

\begin{proof}
(1) 
Both $C(U/U_i) \subset F$ and $[U_i] \notin F$ follow 
from Proposition \ref{Prop_C_almost_F}.
These imply $F \cap C(U)=C(U/U_i)$, since $F \cap C(U)$ is a face of $C(U)$.

(2) 
We prove in the case $\epsilon={+}$.
Since $F \in \Facet_i^+ D(U)$, we have $Y_i^+ \notin \calT_\theta$.

Moreover Proposition \ref{Prop_Facet_disjoint} implies
$F \notin \Facet_{i'}^+ D(U)$ for any $i' \ne i$.
Then Lemma \ref{Lem_w_F_nonzero} gives $\rmw_F Y_{i'}^+=0$, 
and hence $Y_{i'}^+=\rmt_F Y_{i'} \in \calT_\theta$ for any $i' \ne i$.
Thus $Y^+/Y_i^+ \in \calT_\theta$ holds.

We also obtain $\rmw_F X^-=0$ and 
$X^- \in \calF_\theta$ by a similar argument.

(3) 
We only consider the case $\epsilon={+}$.
As in the proof of (2), 
$\rmw_F Y_{i'}^+=0$ holds for any $i' \ne i$,
and moreover $\rmw_F X^-=0$.
Thus we have $(\rmw_F Y^+,\rmw_F X^-)=(\rmw_F Y_i^+,0)=(L_F,0)$.
\end{proof}

To describe faces, the following properties of $M$-TF equivalence
classes are useful.

\begin{Prop}\label{Prop_face_dim}
Let $U=\bigoplus_{i=1}^m U_i \in \twopsilt A$ and $F \in \Face D(U)$.
Set $M:=Y^+ \oplus X^-$.
Then we have
\begin{align*}
F=\{ \theta \in D(U) \mid \theta(\supp_F(\rmw_F M))=0 \}, \quad
\dim_\R F=n-\dim_\R \R(\supp_F(\rmw_F M)).
\end{align*}
\end{Prop}

\begin{proof}
By Lemma \ref{Lem_face_M-TF} (1), we have $F^\circ \in \TF(M)$,
so the set $\supp_F(\rmw_F M)$ is well-defined.
Then Proposition \ref{Prop_M-TF_face} (2) and Lemma \ref{Lem_E_twf} (2) imply
the assertions.
\end{proof}

We provide more information on the short exact sequences \eqref{w_theta F}.

\begin{Lem}\label{Lem_proper_sub_Y}
Let $U=\bigoplus_{i=1}^m U_i \in \twopsilt A$. 
For each $i \in \{1,\ldots,m\}$ and $F \in \Face D(U)$,
the following assertions hold.
\begin{enumerate}
\item
Each proper submodule of $Y_i^+$ is in $\ovcalF_U$, 
and each proper factor module of $X_i^-$ is in $\ovcalT_U$.
\item
If $\rmw_F Y_i^+\neq 0$, then $\rmt_F Y_i^+ \in \calW_U$. 
If $\rmw_F X_i^-\neq 0$, then $\rmf_F X_i^- \in \calW_U$.
\end{enumerate}
\end{Lem}

\begin{proof}
We prove the first statements only.

(1)
Let $L$ be a proper submodule of $Y_i^+$, and $f \in\Hom_A(H^0(U),L)$. 
By Proposition \ref{Prop_T_U_T(Y^+)}, $H^0(U) \in \sfT(Y^+)$.
Then the composite 
$H^0(U) \xrightarrow{f} L \xrightarrow{\mathrm{incl.}} Y_i^+$ 
is zero or surjective by \cite[Lemma 2.7 (1)]{A-semi}.
Since $L \subsetneq Y_i^+$, we get $f=0$. 
Thus $L \in H^0(U)^\perp=\ovcalF_U$.

(2)
If $\rmw_F Y_i^+\ne 0$, then $\rmt_F Y_i^+ \subsetneq Y_i^+$.
Hence $\rmt_F Y_i^+ \in \ovcalF_U$ by (1). 
On the other hand, by taking $\theta \in F^\circ \subset D(U)$, 
we have $\rmt_F Y_i^+ \in \calT_\theta \subset \ovcalT_U$.
Thus $\rmt_F Y_i^+ \in \calW_U$ holds.
\end{proof}

Then we have the following key proposition.
Note that $\supp_F(L_F)$ is well-defined,
since $L_F \in \calW_\theta$ for any $\theta \in F^\circ$ by definition.

\begin{Prop}\label{Prop_L_F_simple}
Let $U=\bigoplus_{i=1}^m U_i \in \twopsilt A$ 
and $F \in \Facet D(U)$.
\begin{enumerate}
\item
We have $\dim_\R \R (\supp_F(L_F))=1$.
\item
For any $\theta \in F^\circ$, we have $L_F \in \simple \calW_\theta$.
\end{enumerate}
\end{Prop}

\begin{proof}
We set $(i,\epsilon):=(i_F,\epsilon_F)$.
We only prove in the case $\epsilon={+}$.

(1)
Propositions \ref{Prop_facet_X_Y_almost} (3) and \ref{Prop_face_dim} imply
$\dim_\R F=n-\dim_\R (\R \supp_F(L_F))$.
Since $\dim_\R F=n-1$, we have $\dim_\R (\R \supp_F(L_F))=1$.

(2)
Set $L:=L_F$.
By $L \ne 0$,
we take a subobject $N \subset L$ in $\calW_\theta$
such that $L/N$ is a simple object in $\calW_\theta$.
It suffices to show $N=0$, so assume $N \ne 0$.

In $\calW_\theta$, the objects $N \ne 0$ and $L/N \ne 0$ 
are filtered by objects in $\supp_F(L)$.
Then (1) implies $[N] \in \Q_{>0}[L/N]$ in $K_0(\mod A)_\R$.
By Lemma \ref{Lem_Y_W_almost} and Proposition \ref{Prop_cone_TF} (1),
we have $Y_i^+ \in \calT_{U_i}=\calT_{[U_i]}$.
Since $L/N$ is a factor module of $Y_i^+$, we have $[U_i](L/N)>0$.
This and $[N] \in \Q_{>0}[L/N]$ give $[U_i](N)>0$.

Now we take the submodule $N' \subset Y_i^+$
such that $Y_i^+/N' \simeq L/N$.
There exists a short exact sequence
$0 \to \rmt_\theta Y_i^+ \to N' \to N \to 0$.
Thus $[U_i](N)=[U_i](N')-[U_i](\rmt_\theta Y_i^+)$ holds.
Since $N \ne 0$, we have $\rmt_\theta Y_i^+ \subsetneq Y_i^+$.
Then Lemma \ref{Lem_proper_sub_Y} (2) implies
$\rmt_\theta Y_i^+=\rmt_F Y_i^+ \in \calW_U \subset \calW_{U_i}$.
Thus Proposition \ref{Prop_cone_TF} (1) gives
$\rmt_\theta Y_i^+ \in \calW_{[U_i]}$, and hence $[U_i](\rmt_\theta Y_i^+)=0$.
Therefore we get $[U_i](N)=[U_i](N')$.

Here $N'$ is a proper submodule of $Y_i^+$,
since $Y_i^+/N' \simeq L/N \ne 0$.
Then Lemma \ref{Lem_proper_sub_Y} (1) gives 
$N' \in \ovcalF_U \subset \ovcalF_{U_i}$.
Again Proposition \ref{Prop_cone_TF} (1) implies $N' \in \ovcalF_{[U_i]}$,
and hence $[U_i](N') \le 0$.
This and the previous paragraph give $[U_i](N) \le 0$.

We have obtained $[U_i](N)>0$ and $[U_i](N) \le 0$, a contradiction.
Thus $N=0$ must hold.
\end{proof}

Now we are able to prove Theorem \ref{Thm_L_F_inn_vec}.

\begin{proof}[Proof of Theorem \ref{Thm_L_F_inn_vec}]
We consider the case $\epsilon={+}$; the other case is dual.
Set $L:=L_F$.

(1) is Propositions \ref{Prop_L_F_simple} (2) and \ref{Prop_End_L_F} (1).

(2)
Set $M:=Y^+ \oplus X^-$.
By Proposition \ref{Prop_facet_X_Y_almost} (3),
we have $\rmw_F M=L$.
Moreover Proposition \ref{Prop_L_F_simple} (2) implies $\supp_F(L)=\{L\}$.
Thus Proposition \ref{Prop_face_dim} gives
$F=\{ \theta \in D(U) \mid \theta(L)=0\}$.
In particular, $[L]$ is a normal vector of $F$.
For any $\theta \in D(U)$,
we have $\theta(L) \ge 0$,
because $L$ is a factor module of 
$Y_i^+ \in \ovcalT_\theta$.
Thus $[L]$ is an inner normal vector of $F$.

(3)
If $i' \ne i$, 
then $L=\rmw_F Y_i^+ \in \calW_{U/U_i} \subset \calW_{U_{i'}}$
hold by Lemma \ref{Lem_F+C_almost} (1),
and we get $[U_{i'}](L)=0$ by Proposition \ref{Prop_cone_TF} (1).

Let $i'=i$.
Then $\rmt_F Y_i^+ \in \calW_U \subset \calW_{U_i}$ holds
by Lemma \ref{Lem_proper_sub_Y} (2), 
so we get $[U_i](\rmt_F Y_i^+)=0$
by Proposition \ref{Prop_cone_TF} (1). 
Then \eqref{w_theta F} implies
\begin{align*}
[U_i](L)=[U_i](\rmw_F Y_i^+)=[U_i](Y_i^+)-[U_i](\rmt_F Y_i^+)=[U_i](Y_i^+).
\end{align*}
By Proposition \ref{Prop_dual_basis}, this value is $d_{U_i}$.
\end{proof}

\subsection{Comparison to the exchange quiver}
\label{Subsec_facet_ex_quiver}

This subsection is devoted to showing
that the brick $L_F$ associated to each facet $F$ 
is an extension of the brick labeling 
in Definition \ref{Def_exchange}
of the exchange quiver $Q(\twosilt A)$.
In this subsection, for each $V \in \twosilt_U A$,
we always assume a decomposition $V=\bigoplus_{i=1}^n V_i \in \twosilt_U A$ 
with $V_i=U_i$ for any $i \in \{1,\ldots,m\}$,
and we set $Z_V=\bigoplus_{i=1}^n Z_{V,i}:=\SH(V)$.

\begin{Thm}\label{Thm_L_F_exch}
Let $U=\bigoplus_{i=1}^m U_i \in \twopsilt A$,
$V=\bigoplus_{i=1}^n V_i \in \twosilt_U A$
and $i \in \{1,\ldots,m\}$.
\begin{enumerate}
\item
There uniquely exists $F \in \Facet D(U)$ such that $C(V/U_i) \subset F$.
\item
The facet $F$ in (1) and the mutation $V'$ of $V$ at $U_i$
satisfy the following.
\begin{enumerate}
\item
If $V'$ is a left mutation of $V$, then $(i_F,\epsilon_F)=(i,+)$, 
and $L_F$ is the label $Z_{V,i}^+$ of the arrow $V \to V'$.
\item
If $V'$ is a right mutation of $V$, then $(i_F,\epsilon_F)=(i,-)$,
and $L_F$ is the label $Z_{V,i}^-$ of the arrow $V' \to V$.
\end{enumerate}
\end{enumerate}
\end{Thm}

\begin{proof}
(1)
By Lemma \ref{Lem_D(U)_incl}, 
we have $C(V/U_i) \subset D(U) \setminus D^\circ(U)=\partial D(U)$.
Since $\dim_\R \R C(V/U_i)=n-1$ and $C(V/U_i)$ is convex,
we have the assertion.

(2)
We only show (i).
Since $C(U/U_i) \subset C(V/U_i) \subset F$, 
we have $i_F=i$ by Proposition \ref{Prop_facet_X_Y_almost} (1).

Since $V'$ is a left mutation of $V$,
Definition-Proposition \ref{Def-Prop_adjacent} (2) implies 
$\calF_V=\calF_{V/U_i}$ and $Z_{V,i}^+ \ne 0$.

Then Lemma \ref{Lem_Y_W_almost} and $\calF_V=\calF_{V/U_i}$ above yield
$X_i^- \in \calF_{U_i} \subset \calF_V=\calF_{V/U_i}$.
By setting $\theta:=[V/U_i]$, 
we get $X_i^- \in \calF_\theta$ from Proposition \ref{Prop_cone_TF} (1),
and hence $\theta \notin \partial_i^-$. 
This and $\theta \in C(V/U_i) \subset F$ imply $F \not \subset \partial_i^-$.
Therefore $\epsilon_F=+$.

Moreover $\calF_V=\calF_{V/U_i}$ means that
$V$ is the maximal completion of $V/U_i$.
Then Proposition \ref{Prop_sim_W_U} gives 
$\simple \calW_{V/U_i}=\{Z_{V,i}^+\}$.
On the other hand, by taking $\theta \in C^\circ(V/U_i) \subset F^\circ$,
we get $L_F \in \simple \calW_\theta=\simple \calW_{V/U_i}$ 
by Propositions \ref{Prop_L_F_simple} (2) and \ref{Prop_cone_TF} (1).
Therefore $L_F \simeq Z_{V,i}^+$ as desired.
By the definition of brick labeling,
the label of the arrow $V \to V'$ is $Z_{V,i}^+$.
\end{proof}

Thus we have the following inclusions.

\begin{Cor}\label{Cor_C(V/Ui)_Z}
Let $U=\bigoplus_{i=1}^m U_i \in \twopsilt A$.
For each $(i,\epsilon) \in \{1,\ldots,m\} \times \{\pm\}$,
we have the inclusions
\begin{align}
\bigcup_{V \in \twosilt_U A, \ Z_{V,i}^\epsilon \ne 0} C(V/U_i) 
&\subset \partial_i^\epsilon, \label{Eq_C(V/Ui)}\\
\{Z_{V,i}^\epsilon \mid V \in \twosilt_U A\} \setminus \{0\} &\subset
\{L_F \mid F \in \Facet_i^\epsilon D(U)\} \label{Eq_Z}.
\end{align}
\end{Cor}

\begin{proof}
We show in the case $\epsilon={+}$;
the dual argument works in the case $\epsilon={-}$.

Assume that $V \in \twosilt_U A$ satisfies $Z_{V,i}^+ \ne 0$.
By Definition-Proposition \ref{Def-Prop_adjacent} (2), 
the mutation of $V$ at $U_i$ is a left mutation of $V$.
Thus Theorem \ref{Thm_L_F_exch} implies that
there exists $F \in \Facet_i^+ D(U)$ such that
$C(V/U_i) \subset F \subset \partial_i^+$ and $Z_{V,i}^+=L_F$. 
Thus we have the assertions.
\end{proof}

For each inclusion in Corollary \ref{Cor_C(V/Ui)_Z},
we can characterize when it is an equality.

\begin{Thm}\label{Thm_C(V/Ui)_Z_=}
Let $U=\bigoplus_{i=1}^m U_i \in \twopsilt A$ 
and $(i,\epsilon) \in \{1,\ldots,m\} \times \{\pm\}$.
\begin{enumerate}
\item
Both sides of \eqref{Eq_C(V/Ui)} coincide if and only if
$\partial_i^\epsilon \subset \Cone$.
\item
Both sides of \eqref{Eq_Z} coincide if and only if
every $F \in \Facet_i^\epsilon D(U)$ has $V' \in \twopsilt A$
such that $C^\circ(V) \subset F$ and $|V'|=n-1$.
\item
If \eqref{Eq_C(V/Ui)} is an equality, then so is \eqref{Eq_Z}.
\item
If $B$ is brick finite, 
then \eqref{Eq_C(V/Ui)} and \eqref{Eq_Z} are equalities.
\end{enumerate}
\end{Thm}

\begin{proof}
(1)
The ``only if'' part is obvious.

Conversely, assume $\partial_i^\epsilon \subset \Cone$.
Since $\partial_i^\epsilon$ is closed and 
is a union of TF equivalence classes,
Proposition \ref{Prop_cone_TF} (1) gives that
$\partial_i^\epsilon$ is the union of $C(V')$
for all $V' \in \twopsilt A$ such that $C(V') \subset \partial_i^\epsilon$.
Then by the same argument as \cite[Proposition 4.9]{A},
there are only finitely many indecomposable $Q \in \twopsilt A$
such that $[Q] \in \partial_i^\epsilon$.
Thus $\partial_i^\epsilon$ is 
the union of $C(V')$ for all $V' \in \twopsilt A$ 
such that $|V'|=n-1$ and $C(V') \subset \partial_i^\epsilon$.

Now let $\theta \in \partial_i^\epsilon$.
Then the previous paragraph allows
us to take $V' \in \twopsilt A$ such that $|V'|=n-1$ and
$\theta \in C(V') \subset \partial_i^\epsilon$.
Since $C(V') \subset \partial_i^\epsilon \subset D(U)$, 
Lemma \ref{Lem_D(U)_incl} implies that
$U \oplus V'$ is 2-term presilting.
By Theorem \ref{Thm_partial_setminus},
we have $C(V') \subset \partial_i^\epsilon \subset 
D(U) \setminus D^\circ(U_i)$,
so Lemma \ref{Lem_D(U)_incl} again implies 
$U_i \oplus V' \in \twopsilt A$,
including that $U_i \oplus V'$ is basic.
This and $|V|=n-1$ give $U_i \oplus V' \in \twosilt A$.
Since $U \oplus V'$ is 2-term presilting and $U_i \oplus V' \in \twosilt A$,
we get $U/U_i \in \add V'$ and $V:=U_i \oplus V' \in \twosilt_U A$.

It remains to show $Z_{V,i}^\epsilon \ne 0$.
Take the unique facet $F \in \Facet D(U)$ such that
$C(V/U_i) \subset F$ by Theorem \ref{Thm_L_F_exch} (1).
Since $C(V/U)=C(V') \subset \partial_i^\epsilon$
and since $\partial_i^\epsilon$ is a union of facets 
by Theorem \ref{Thm_partial_facet} (1),
we have $C(V/U) \subset F \subset \partial_i^\epsilon$.
Thus we get $(i_F,\epsilon_F)=(i,\epsilon)$.
Then Theorem \ref{Thm_L_F_exch} (2) gives
$Z_{V,i}^\epsilon \ne 0$.

(2)
Assume that both sides of \eqref{Eq_Z} coincide.
Let $F \in \Facet_i^\epsilon D(U)$.
Then there exists $V \in \twosilt_U A$ such that $Z_{V,i}^\epsilon=L_F \ne 0$.
By Theorem \ref{Thm_L_F_exch},
there uniquely exists $F' \in \Facet D(U)$ such that $C(V/U_i) \subset F'$,
and this $F'$ satisfies $L_{F'}=Z_{V,i}^{\epsilon'} \ne 0$, 
where $\epsilon':=\epsilon_F$.
Since $Z_{V,i}^\epsilon \ne 0$ and $Z_{V,i}^{\epsilon'} \ne 0$,
we have $\epsilon=\epsilon'$.
Thus we obtain $L_{F'}=Z_{V,i}^\epsilon=L_F$.
Then Theorem \ref{Thm_L_F_inn_vec} (2) implies that $F'=F$.
Thus we get $C(V/U_i) \subset F$.
Since $|V/U_i|=n-1$, the ``only if'' part is proved.

Next we show the ``if'' part.
Let $F \in \Facet_i^\epsilon D(U)$,
and take $V' \in \twopsilt A$ such that $C(V') \subset F$ and $|V'|=n-1$.
Then we have $C(V') \subset F \subset \partial_i=D(U) \setminus D(U_i)$
by Theorem \ref{Thm_partial_setminus}.
Thus Lemma \ref{Lem_D(U)_incl} gives
$V:=U_i \oplus V'$ is a basic 2-term silting complex.
Clearly, $F \in \Facet_i^\epsilon D(U)$ satisfies $C(V'/U) \subset F$,
so Theorem \ref{Thm_L_F_exch} gives $L_F=Z_{V,i}^\epsilon \ne 0$ as desired.

(3)
Let $F \in \Facet_i^\epsilon D(U)$.
Since \eqref{Eq_C(V/Ui)} is an eqaulity, (1) implies that $F \subset \Cone$.
Since $F$ is a facet,
we take $\theta \in F$ which is not contained in any rational polyhedral cone
of dimension $\le n-2$.
Then $\theta \in F \subset \Cone$ give that 
there exists $V \in \twopsilt A$ 
such that $\theta \in C^\circ(V)$ and $|V|=n-1$.
By Lemma \ref{Lem_face_M-TF}, $F$ is a union of TF equivalence classes,
so Proposition \ref{Prop_cone_TF} and $\theta \in C^\circ(V) \cap F$ give
$C^\circ(V) \subset F$.
Thus $C(V) \subset F$ holds.
Then (2) implies that \eqref{Eq_Z} is an equality.

(4)
Since $B$ is finite, Proposition \ref{Prop_B_brick_fin} implies
$D(U) \subset \Cone$.
Thus we get $\partial_i^\epsilon \subset \Cone$.
Now (1) and (3) give the assertion.
\end{proof}

We give an example.

\begin{Ex}\label{Ex_A4_L_F}
As in Example \ref{Ex_A4_D(U)},
let $A$ be the path algebra of the quiver $1 \to 2 \to 3 \to 4$,
and $U=U_1 \oplus U_2  \in \twopsilt A$
with $U_1=(P(4) \to P(3))$ and $U_2=P(1)$.
Then we have seen in Example \ref{Ex_A4_D(U)} that
\begin{align*}
Y_1^+=\begin{smallmatrix}3\end{smallmatrix}, \quad
X_1^-=\begin{smallmatrix}4\end{smallmatrix}, \quad
Y_2^+=\begin{smallmatrix}1\\2\\3\\4\end{smallmatrix}, \quad
X_2^-=0,
\end{align*}
and $D(U)=\R_{\ge 0}D$, where 
\begin{align*}
D&=D(U) \cap
\{x_1[P(1)]+x_2[P(2)]+x_3[P(3)]+x_4[P(4)] \mid 2x_1+x_2+x_3=1,
x_1,x_2,x_3,x_4 \in \R\}
\end{align*}
and $D$ is depicted as follows.
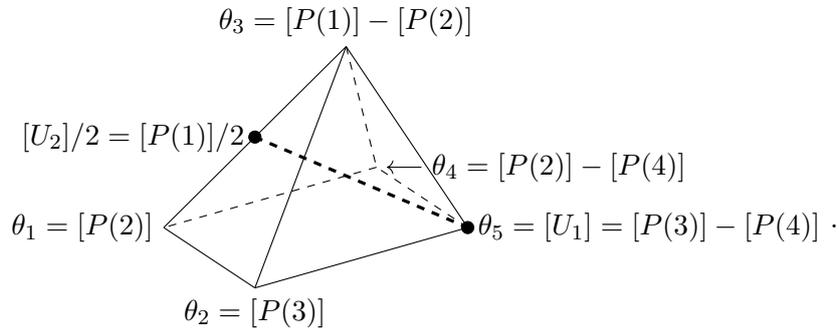
\begin{align*}
\begin{tikzpicture}[baseline=0pt,scale=0.8]
\node (F3)[coordinate,label= 90:{$\theta_3=[P(1)]-[P(2)]$}] at ( 0  , 3  ) {};
\node (F1)[coordinate,label=180:{$\theta_1=[P(2)]$}]        at (-3  , 0  ) {};
\node (F2)[coordinate,label=270:{$\theta_2=[P(3)]$}]        at (-1.5,-1  ) {};
\node (F5)[coordinate,label=  0:{$\theta_5=[U_1]=[P(3)]-[P(4)]$}] 
at ( 2  , 0  ) {};
\node (F4)[coordinate]                                      at ( 0.5, 1  ) {};
\node (U2)[coordinate,label=180:{$[U_2]/2 =[P(1)]/2$}]      at (-1.5, 1.5) {};
\node (F4')  at ( 0.5, 1  ) {};
\node (F4'') at ( 3.5, 1  ) {$\theta_4=[P(2)]-[P(4)]$};
\draw (F3) to (F1);
\draw (F3) to (F2);
\draw[dashed] (F3) to (F4);
\draw (F3) to (F5);
\draw (F1) to (F2);
\draw (F2) to (F5);
\draw[dashed] (F5) to (F4);
\draw[dashed] (F4) to (F1);
\draw[dashed,very thick] (U2) to (F5);
\draw[->] (F4'') to (F4');
\draw[fill=black] (F5) circle [radius=0.1];
\draw[fill=black] (U2) circle [radius=0.1];
\end{tikzpicture}.
\end{align*}

We set $F_i:=\R_{\ge 0}\theta_i$ for each $i \in \{1,2,3,4,5\}$.
It is easy to see that $F_1,F_2,F_3,F_4,F_5$
are all the distinct 1-dimensional faces of $D(U)$.
Moreover we abbreviate $F_1+F_3+F_4$ as $F_{134}$ and so on.

The following table shows the elements of $\Facet D(U)$
and $i_F,\epsilon_F,L_F$ for each $F \in \Facet D(U)$.
Note that $\Facet_2^- D(U)$ is empty.
\begin{center}
\begin{tabular}{c|ccc}
$F \in \Facet D(U)$ & $i_F$ & $\epsilon_F$ & $L_F$ \\
\hline
$F_{134}$ & 1 & $+$ & $\begin{smallmatrix}3\end{smallmatrix}$\\
\hline
$F_{123}$ & 1 & $-$ & $\begin{smallmatrix}4\end{smallmatrix}$\\
\hline
$F_{345}$ & 2 & $+$ &
$\begin{smallmatrix}\\1\\2\\3\\4\\\end{smallmatrix}$\\
\hline
$F_{235}$ & 2 & $+$ &
$\begin{smallmatrix}1\\2\end{smallmatrix}$\\
\hline
$F_{1245}$ & 2 & $+$ & $\begin{smallmatrix}1\end{smallmatrix}$
\end{tabular}
\end{center} 

Therefore by the purities of $\partial_i^\epsilon$ 
in Theorem \ref{Thm_partial_facet} (1),
\begin{align*}
\partial_1^+=F_{134}, \quad
\partial_1^-=F_{123}, \quad
\partial_2^+=F_{345} \cup F_{235} \cup F_{1245}, \quad
\partial_2^-=\emptyset.
\end{align*}
Thus we get
\begin{align*}
\partial_1=F_{134} \cup F_{123}, \quad
\partial_2=F_{345} \cup F_{235} \cup F_{1245}, \quad
\partial^+=F_{134} \cup F_{345} \cup F_{235} \cup F_{1245}, \quad
\partial^-=F_{123}.
\end{align*}

By direct calculations, we have
\begin{align}\label{Eq_d1_d+_proper_d1+}
\partial_1 \cap \partial^+=F_{134} \cup F_{12} \cup F_{13} \cup F_{23}
\supsetneq F_{134}=\partial_1^+,
\end{align}
but $\Facet_1^+ D(U)=\Facet_1 D(U) \cap \Facet^+ D(U)=\{F_{134}\}$
still holds as in Theorem \ref{Thm_partial_facet} (2).

We next see the equality
\begin{align*}
\{Z_{V,i}^\epsilon \mid V \in \twosilt_U A\} \setminus \{0\}=
\{L_F \mid F \in \Facet_i^\epsilon D(U)\}
\end{align*}
in Theorem \ref{Thm_C(V/Ui)_Z_=} (2).
By definition, if $Z_{V,i}^+ \ne 0$, 
then $Z_{V,i}^+$ is the label of the arrow starting at $V$
corresponding to the left mutation of $V$ at $U_i$.
This arrow ends at a vertex outside $\twosilt_U A$.
Dually if $Z_{V,i}^+ \ne 0$,
then $Z_{V,i}^-$ is the label of an arrow ending at $V$
starting outside $\twosilt_U A$.

Therefore it is enough to consider 
each arrow such that exactly one of its two endpoints is in $\twosilt_U A$
in the exchange quiver of $\twosilt A$.
Such arrows are 
the bold solid arrows and the bold dotted arrows in the next picture
corresponding to mutations at $U_1$ and $U_2$ respectively.
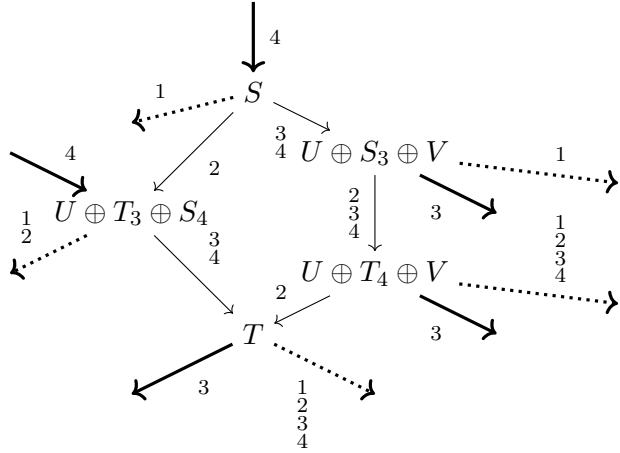
\begin{align*}
\begin{tikzpicture}[->,baseline=0pt,scale=0.8]
\node (1) at ( 0, 2) {$S$};
\node (2) at ( 2, 1) {$U \oplus S_3 \oplus V$};
\node (3) at ( 2,-1) {$U \oplus T_4 \oplus V$};
\node (4) at (-2, 0) {$U \oplus T_3 \oplus S_4$};
\node (5) at ( 0,-2) {$T$};
\draw (1) to [edge label'={$\begin{smallmatrix}3\\4\end{smallmatrix}$}](2);
\draw (2) to [edge label'={$\begin{smallmatrix}2\\3\\4\end{smallmatrix}$}](3);
\draw (3) to [edge label'={$\begin{smallmatrix}2\end{smallmatrix}$}](5);
\draw (1) to [edge label={$\begin{smallmatrix}2\end{smallmatrix}$}](4);
\draw (4) to [edge label={$\begin{smallmatrix}3\\4\end{smallmatrix}$}](5);
\draw[very thick] 
(0,3.5) to [edge label={$\begin{smallmatrix}4\end{smallmatrix}$}] (1);
\draw[very thick,dotted] 
(1) to [edge label'={$\begin{smallmatrix}1\end{smallmatrix}$}] (-2,1.5);
\draw[very thick] 
(2) to [edge label'={$\begin{smallmatrix}3\end{smallmatrix}$}] (4,0);
\draw[very thick,dotted] 
(2) to [edge label={$\begin{smallmatrix}1\end{smallmatrix}$}] (6,0.5);
\draw[very thick] 
(3) to [edge label'={$\begin{smallmatrix}3\end{smallmatrix}$}] (4,-2);
\draw[very thick,dotted] 
(3) to [edge label={$\begin{smallmatrix}1\\2\\3\\4\end{smallmatrix}$}]
(6,-1.5);
\draw[very thick] 
(-4,1) to [edge label={$\begin{smallmatrix}4\end{smallmatrix}$}] (4);
\draw[very thick,dotted] 
(4) to [edge label'={$\begin{smallmatrix}1\\2\end{smallmatrix}$}] (-4,-1);
\draw[very thick] 
(5) to [edge label={$\begin{smallmatrix}3\end{smallmatrix}$}] (-2,-3);
\draw[very thick,dotted] 
(5) to [edge label'={$\begin{smallmatrix}1\\2\\3\\4\end{smallmatrix}$}] 
(2,-3);
\end{tikzpicture}
\end{align*}
Here we use $S_3,S_4,T_3,T_4,V$ in Example \ref{Ex_A4_D(U)}; namely,
\begin{align*}
S_3=P(2), \quad S_4=P(3), \quad
T_3=(P(2) \to P(1)), \quad T_4=(P(4) \to P(1)), \quad V=(P(4) \to P(2)).
\end{align*}
For example, $\{Z_{V,2}^+ \mid V \in \twosilt_U A\} \setminus \{0\}$
is the set of labels of bold dotted arrows going outside $\twosilt A$.
By the picture above, it is 
$\left\{\begin{smallmatrix}1\\2\\3\\4\end{smallmatrix},
\begin{smallmatrix}1\\2\end{smallmatrix},
\begin{smallmatrix}1\end{smallmatrix}\right\}$,
and coincides with $\{L_F \mid F \in \Facet_2^+ D(U)\}$.
\end{Ex}

\section{Faces of the interval neighborhood $D(U)$}\label{Sec_face}

\subsection{The link $L(U) \subset D(U)$ of the silting cone $C(U)$}
\label{Subsec piecewise}

In this section, we study 
the general faces of the rational polyhedral cone $D(U)$
for $U \in \twopsilt A$ by fully using the following subset of $D(U)$.

\begin{Def}
Let $U=\bigoplus_{i=1}^m U_i \in \twopsilt A$.
Then we set the \emph{link} $L(U)$ of $C(U)$ by
\begin{align*}
L(U):=D(U) \setminus \left( \bigcup_{i=1}^m D^\circ(U_i) \right)
\stackrel{\text{Thm.~\ref{Thm_partial_setminus}}}{=}
\bigcap_{i=1}^m \partial_i.
\end{align*}
\end{Def}

The reason why we call $L(U)$ the link is the following property.

\begin{Prop}\label{Prop_L(U)_cone}
Let $U \in \twopsilt A$.
Then we have
\begin{align*}
L(U) \cap \Cone=\bigsqcup_{V \in \twopsilt_U A}C^\circ(V/U)
=\bigcup_{V \in \twosilt_U A}C(V/U).
\end{align*}
\end{Prop}

\begin{proof}
We first show the left equality.
We decompose $U$ as $\bigoplus_{i=1}^m U_i \in \twopsilt A$.

For the ``$\subset$'' part, let $\theta \in L(U) \cap \Cone$,
and take $V' \in \twopsilt A$ such that $\theta \in C^\circ(V')$.
Then by Lemma \ref{Lem_D(U)_incl}, 
$V:=U \oplus V'$ is a basic 2-term presilting complex, 
and $\theta \in C^\circ(V/U)$.

For the ``$\supset$'' part, let $V \in \twopsilt_U A$.
Then $\add(V/U) \cap \add U=\{0\}$, since $V$ is basic.
Thus by Lemma \ref{Lem_D(U)_incl} again,
$C(V/U) \subset D(U)$ and $C(V/U) \cap D^\circ(U)=\emptyset$.
Thus we get $C(V/U) \subset L(U) \cap \Cone$ as desired. 

The second equality follows from Proposition \ref{Prop_U_S}.
\end{proof}

By Propositions \ref{Prop_L(U)_cone} and \ref{Prop_D(U)_Cone},
it is easy to see that there is a bijection
\begin{align*}
C(U) \times (L(U) \cap \Cone) \to D(U) \cap \Cone, \quad 
(\eta,\eta') \mapsto \eta+\eta'.
\end{align*}
The inverse map send $\theta \in C(V) \subset D(U) \cap \Cone$
with $V \in \twopsilt_U A$
to the pair $(\eta,\eta')$ given by the unique elements
$\eta \in C(U)$ and $\eta' \in C(V/U)$ such that $\theta=\eta+\eta'$.

The aim of this subsection is to show that we can extend it to a bijection
\begin{align*}
C(U) \times L(U) \to D(U), \quad (\eta,\eta') \mapsto \eta+\eta'.
\end{align*}
Our strategy is explicitly constructing two maps $\lambda_U,\lambda'_U$
from results on facets of $D(U)$ obtained in the previous section,
and show that they give the inverse map 
\begin{align*}
D(U) \to C(U) \times L(U), \quad \theta 
\mapsto (\lambda_U(\theta),\lambda'_U(\theta)).
\end{align*}

The following property is crucial to define $\lambda_U,\lambda'_U$.
Note that $a_i$ is well-defined, 
since $\Facet_i D(U) \ne \emptyset$ for any $i$ 
by Theorem \ref{Thm_partial_facet}.

\begin{Prop}\label{Prop_lambda_coeff}
Let $U=\bigoplus_{i=1}^m U_i \in \twopsilt A$ and $\theta \in D(U)$.
We set $H \subset \R C(U)$ and $a_1,\ldots,a_m \in \R_{\ge 0}$ by
\begin{align*}
H:=\{\eta \in \R C(U) \mid \theta-\eta \in D(U)\}, \quad
a_i:=\min_{F \in \Facet_i D(U)} \frac{|\theta(L_F)|}{d_{U_i}}.
\end{align*}
\begin{enumerate}
\item
We have the equality
\begin{align*}
H=\sum_{i=1}^m a_i[U_i]+(-C(U)). 
\end{align*}
\item
Let $U'=\bigoplus_{i \in I} U_i$ with $I \subset \{1,\ldots,m\}$.
We set
\begin{align*}
H':=\{\eta \in \R C(U') \mid \theta-\eta \in D(U')\}.
\end{align*}
Then we have 
\begin{align*}
H'=H \cap \R C(U')=\sum_{i \in I} a_i[U_i]+(-C(U')).
\end{align*}
\item
If $\theta \in K_0(\proj A)$, then $a_1,\ldots,a_m \in \Z_{\ge 0}$.
\end{enumerate}
\end{Prop}

\begin{proof}
(1)
By Corollary \ref{Cor_D(U)_ineq}, we have
\begin{align*}
H&=\{ \eta \in \R C(U) \mid 
\text{$\epsilon_F \cdot (\theta-\eta)(L_F) \ge 0$
for any $F \in \Facet D(U)$} \}\\
&=\{ \eta \in \R C(U) \mid 
\text{$\epsilon_F \cdot (\theta-\eta)(L_F) \ge 0$
for any $i \in \{1,\ldots,m\}$ and $F \in \Facet_i D(U)$} \}.
\end{align*}
Let $\eta=\sum_{i=1}^m x_i[U_i] \in \R C(U)$.
By Theorem \ref{Thm_L_F_inn_vec} (3), 
for each $i \in \{1,\ldots,m\}$ and $F \in \Facet_i D(U)$, 
we have
\begin{align*}
\eta(L_F)=\sum_{i'=1}^m x_{i'}[U_{i'}](L_F)=x_i \epsilon_F d_{U_i},
\end{align*}
so we get
\begin{align*}
\epsilon_F \cdot (\theta-\eta)(L_F)=
\epsilon_F \cdot \theta(L_F)-x_i d_{U_i}=
|\theta(L_F)|-x_i d_{U_i},
\end{align*}
where $\epsilon_F \cdot \theta(L_F)=|\theta(L_F)|$ comes from
$\theta \in D(U)$ and Corollary \ref{Cor_D(U)_ineq}.
Thus $H$ is written as 
\begin{align*}
H&=\left\{ \sum_{i=1}^m x_i[U_i] \in \R C(U) \mid 
\text{$x_i \le \dfrac{|\theta(L_F)|}{d_{U_i}}$
for any $i \in \{1,\ldots,m\}$ and $F \in \Facet_i D(U)$} \right\}.
\end{align*}
Then we have the desired equality by the definition of $a_i$.

(2)
For any $\eta \in H \cap \R C(U')$,
we have $\theta-\eta \in D(U) \subset D(U')$; hence $\eta \in H'$.
Thus $H \cap \R C(U') \subset H'$ holds.

Conversely, assume that $\eta=\sum_{i \in I} x_i[U_i] \in H'$.
Let $i \in I$ and $F \in \Facet_i D(U)$. 
Then $L_F$ is a factor module of $H^0(U_i)$ if $\epsilon_F={+}$, and 
$L_F$ is a submodule of $H^{-1}(\nu U_i)$ if $\epsilon_F={-}$
by Proposition \ref{Prop_rad_soc_U}.
Thus $\theta-\eta \in D(U')$
implies $\epsilon_F \cdot (\theta-\eta)(L_F) \ge 0$.
Thus we get 
$x_i \le |\theta(L_F)|/d_{U_i}$ by the proof of (1).
By the definition of $a_i$, we have $x_i \le a_i$.
This holds for any $i \in I$, so we get $\theta \in H \cap \R C(U')$.
Therefore $H' \subset H \cap \R C(U')$.

Now $H'=H \cap \R C(U')$ is proved, and the second equality is clear.

(3) follows from Proposition \ref{Prop_End_L_F} (2).
\end{proof}

Thus we define maps $\lambda_U,\lambda'_U$ 
based on Proposition \ref{Prop_lambda_coeff}.

\begin{Def}\label{Def_lambda}
For any $U=\bigoplus_{i=1}^m U_i \in \twopsilt A$,
we define piecewise linear maps $\lambda_U,\lambda'_U \colon D(U) \to D(U)$ by
\begin{align*}
\lambda_U(\theta):=\sum_{i=1}^m \left( 
\min_{F \in \Facet_i D(U)}\frac{|\theta(L_F)|}{d_{U_i}}
\right)[U_i], \quad
\lambda'_U(\theta):=\theta-\lambda_U(\theta).
\end{align*}
\end{Def}

The following property gives the relationship 
between $\lambda'_U(\theta)$ and $(\theta+\R C(U)) \cap D(U)$,
which is a section of $D(U)$ by the affine hyperplane $\theta+\R C(U)$.

\begin{Prop}\label{Prop_affine}
Let $U \in \twopsilt A$.
For any $\theta \in D(U)$, we have
\begin{align*}
(\theta+\R C(U)) \cap D(U)=\lambda'_U(\theta)+C(U).
\end{align*}
\end{Prop}

\begin{proof}
The set $H$ in Proposition \ref{Prop_lambda_coeff} satisfies 
$H=\lambda_U(\theta)+(-C(U))$.

Let $\eta \in (\theta+\R C(U)) \cap D(U)$.
Then $\eta':=\theta-\eta \in \R C(U)$ satisfies
$\eta=\theta-\eta' \in D(U)$,
so $\eta' \in H=\lambda_U(\theta)+(-C(U))$.
Thus we have $\lambda_U(\theta)-\eta'=\lambda_U(\theta)-\theta+\eta
=-\lambda'_U(\theta)+\eta \in C(U)$,
which imply $\eta \in \lambda'_U(\theta)+C(U)$.

Conversely, let $\eta \in \lambda'_U(\theta)+C(U)$.
Since $\lambda'_U(\theta)=\theta-\lambda_U(\theta) \in \theta+(-C(U))$,
we have $\eta \in \theta+\R C(U)$.
Moreover $\eta \in \lambda'_U(\theta)+C(U) \subset D(U)+D(U)=D(U)$ also hold.
\end{proof}

The following observations are useful.
In particular, the maps 
$\lambda_U$ and $\lambda'_U$ satisfy the axiom of projections.

\begin{Lem}\label{Lem_axiom_proj}
Let $U \in \twopsilt A$ and $\theta \in D(U)$.
\begin{enumerate}
\item
If $\theta_1,\theta_2 \in D(U)$ satisfy 
$\eta:=\theta_2-\theta_1 \in \R C(U)$,
then $\lambda_U(\theta_2)=\eta+\lambda_U(\theta_1)$
and $\lambda'_U(\theta'_2)=\lambda'_U(\theta_1)$.
\item
For any $\theta \in C(U)$,
we have $\lambda_U(\theta)=\theta$ and $\lambda'_U(\theta)=0$.
\item
We have $\lambda_U(\lambda_U(\theta))=\lambda_U(\theta)$,
$\lambda'_U(\lambda_U(\theta))=0$,
$\lambda_U(\lambda'_U(\theta))=0$,
$\lambda'_U(\lambda'_U(\theta))=\lambda'_U(\theta)$.
\item
Let $U=V \oplus W$. 
Then $\lambda_U(\theta)=\lambda_V(\theta)+\lambda_W(\theta)$ and
$\lambda_W(\theta)=\lambda_W(\lambda'_V(\theta))$ hold.
\end{enumerate}
\end{Lem}

\begin{proof}
(1)
Since $\eta \in \R C(U)$,
we have $(\theta_2+\R C(U)) \cap D(U)=(\theta_1+\R C(U)) \cap D(U)$.
By Proposition \ref{Prop_affine},
we obtain $\lambda'_U(\theta_2)=\lambda'_U(\theta_1)$.
Thus $\lambda_U(\theta_2)=\theta_2-\lambda'_U(\theta_2)
=\theta_2-\lambda'_U(\theta_1)=\theta_2-\theta_1+\lambda_U(\theta_1)
=\eta+\lambda_U(\theta_1)$.

(2)
Since $0,\theta \in D(U)$ and $\theta-0 \in C(U)$,
(1) gives $\lambda_U(\theta)=\theta+\lambda_U(0)$.
By Definition \ref{Def_lambda}, $\lambda_U(0)=0$.
Thus $\lambda_U(\theta)=\theta$; hence $\lambda'_U(\theta)=0$.

(3)
The first two equalities follow from (2).
By (1), 
$\lambda_U(\theta)=\lambda_U(\lambda_U(\theta)+\lambda'_U(\theta))
=\lambda_U(\theta)+\lambda_U(\lambda'_U(\theta))$
hold, so $\lambda_U(\lambda'_U(\theta))=0$.
We also have $\lambda'_U(\lambda'_U(\theta))
=\lambda'_U(\theta)-\lambda_U(\lambda'_U(\theta))=\lambda'_U(\theta)-0
=\lambda'_U(\theta)$.

(4)
Note that $\add V \cap \add W=\{0\}$, since $U$ is basic.
We uniquely write $\lambda_U(\theta)=\eta_1+\eta_2$ 
with $\eta_1 \in C(V)$ and $\eta_2 \in C(W)$.
Then Proposition \ref{Prop_lambda_coeff} (2) implies 
$\eta_1=\lambda_V(\theta)$ and $\eta_2=\lambda_W(\theta)$.
Thus the first equality follows.

We next show the second equality.
By the first one, we have
$\lambda_V(\lambda'_V(\theta))+\lambda_W(\lambda'_V(\theta))
=\lambda_U(\lambda'_V(\theta))$,
and then (3) gives
$\lambda_W(\lambda'_V(\theta))=\lambda_U(\lambda'_V(\theta))$.
By (1), we have
$\lambda_U(\lambda'_V(\theta))=\lambda_U(-\lambda_V(\theta)+\theta)
=-\lambda_V(\theta)+\lambda_U(\theta)$,
which is $\lambda_W(\theta)$ by the first equality.
\end{proof}

For elements $\theta$ in $D(U)$ and $V \in \twopsilt A$,
we can determine whether $\theta \in D^\circ(V)$ holds
from $\lambda_U(\theta)$ and $\lambda'_U(\theta)$.

\begin{Lem}\label{Lem_lambda_DcV}
Let $U \in \twopsilt A$, and $\theta \in D(U)$.
Assume that $V \in \twopsilt A$.
\begin{enumerate}
\item
If $V \in \add U$, then $\theta \in D^\circ(V)$ holds 
if and only if $\lambda_U(\theta) \in C^\circ(V)+C(U/V)$.
\item
If $\add U \cap \add V=\{0\}$, 
then $\theta \in D^\circ(V)$ holds if and only if 
$\lambda'_U(\theta) \in D^\circ(V)$.
\item
Let $V=U' \oplus V' \in \twopsilt A$ with
$U' \in \add U$ and $\add U \cap \add V'=\{0\}$.
Then $\theta \in D^\circ(V)$ holds if and only if 
$\lambda_U(\theta) \in C^\circ(U')+C(U/U')$ and 
$\lambda'_U(\theta) \in D^\circ(V')$.
\end{enumerate}
\end{Lem}

\begin{proof}
(1)
Let $\theta \in D^\circ(V)$. 
Then there exists some $\eta \in C^\circ(V)$ 
such that $\theta-\eta \in D(V)$,
so $\lambda_V(\theta) \in C^\circ(V)$. 
Then Lemma \ref{Lem_axiom_proj} (4) implies
$\lambda_U(\theta)=\lambda_V(\theta)+\lambda_{U/V}(\theta)
\in C^\circ(V)+C(U/V)$.

Conversely, let $\lambda_U(\theta) \in C^\circ(V)+C(U/V)$.
Then by Lemma \ref{Lem_D(U)_incl} (2),
$C(U/V) \subset D(U) \subset D(V)$,
so we have $\lambda_U(\theta) \in D^\circ(V)+D(V) \subset D^\circ(V)$.
Since $\lambda'_U(\theta) \in D(U) \subset D(V)$,
we have $\theta=\lambda_U(\theta)+\lambda'_U(\theta) 
\in D^\circ(V)+D(V)=D^\circ(V)$.

(2)
Let $\theta \in D^\circ(V)$.
By assumption, we get $\theta \in D(U) \cap D^\circ(V)$.
Then Lemma \ref{Lem_D(U)_incl} (2) and $\add U \cap \add V=\{0\}$ imply
$U \oplus V \in \twopsilt A$,
and we have $\theta \in D(U \oplus V)$.
Thus $\lambda_{U \oplus V}(\theta)$ is defined, 
and $\lambda_U(\theta)+\lambda_V(\theta)=\lambda_{U \oplus V}(\theta)$
by Lemma \ref{Lem_axiom_proj} (4),
so we have $\lambda'_U(\theta)
=\lambda_V(\theta)+\lambda'_{U \oplus V}(\theta)$.
Since $\theta \in D^\circ(V)$, 
we get $\lambda_V(\theta) \in C^\circ(V) \subset D^\circ(V)$ by (1).
We also have $\lambda'_{U \oplus V}(\theta) \in D(U \oplus V) \subset D(V)$.
Thus $\lambda'_U(\theta) \in D^\circ(V)+D(V) \subset D^\circ(V)$ as desired.

Conversely, let $\lambda'_U(\theta) \in D^\circ(V)$.
From the definition of $\lambda'_U$, 
we obtain $\lambda'_U(\theta) \in D(U) \cap D^\circ(V)$ follows.
Then Lemma \ref{Lem_D(U)_incl} (2) again implies $C(U) \subset D(V)$.
Thus we have
$\theta=\lambda_U(\theta)+\lambda'_U(\theta) 
\in C(U)+D^\circ(V) \subset D(V)+D^\circ(V)=D^\circ(V)$.

(3) 
is clear by (1) and (2).
\end{proof}

We determine the images of $\lambda_U$ and $\lambda'_U$.

\begin{Prop}\label{Prop_lambda_image}
Let $U \in \twopsilt A$.
\begin{enumerate}
\item
The images of $\lambda_U$ and $\lambda'_U$ are 
$C(U)$ and $L(U)$, respectively.
\item
The map $\lambda_U$ is identity on $C(U)$,
and $\lambda'_U$ is identity on $L(U)$.
\end{enumerate}
\end{Prop}

\begin{proof}
(1)
The images of $\lambda_U$ is a subset of $C(U)$ by definition.
Conversely, $C(U)$ is contained in the image of $\lambda_U$
by Lemma \ref{Lem_axiom_proj} (2).

To determine the image of $\lambda'_U$, let $\theta \in D(U)$.
Then $\lambda_U(\lambda'_U(\theta))=0$ by Lemma \ref{Lem_axiom_proj} (3).
Thus $\lambda'_U(\theta) \notin D^\circ(U_i)$ for any $i \in \{1,\ldots,m\}$
by Lemma \ref{Lem_lambda_DcV} (1).
Therefore $\lambda'_U(\theta) \in L(U)$,
since $\lambda'_U(\theta) \in D(U)$ by definition.
We have proved $\lambda'_U(D(U)) \subset L(U)$.

Conversely, if $\theta \in L(U)$,
then $\lambda_U(\theta)=0$ holds by Lemma \ref{Lem_lambda_DcV} (1) again,
so we get $\theta=\lambda'_U(\theta) \in \lambda'_U(D(U))$.

(2)
follows from (1) and Lemma \ref{Lem_axiom_proj} (3).
\end{proof}

Now we can show that the desired bijection exists.

\begin{Thm}\label{Thm_C(U)_L(U)}
Let $U \in \twopsilt A$.
Then we have the following results.
\begin{enumerate}
\item
We have a bijection
\begin{align*}
C(U) \times L(U) &\to D(U), \quad 
(\eta,\eta') \mapsto \eta+\eta'
\end{align*}
whose inverse is given by 
$\theta\mapsto (\lambda_U(\theta),\lambda'_U(\theta)).$
\item
Let $V=U' \oplus V' \in \twopsilt A$ with
$U' \in \add U$ and $\add U \cap \add V'=\{0\}$.
Then $\Psi_1$ is restricted to a bijection
\begin{align*}
(C^\circ(U')+C(U/U')) \times (L(U) \cap D^\circ(V')) &\to
D(U) \cap D^\circ(V).
\end{align*}
In particular, we have a bijection
\begin{align*}
C^\circ(U) \times L(U) \to D^\circ(U).
\end{align*}
\end{enumerate}
\end{Thm}

\begin{proof}
(1)
We write $\Psi_1$ for
$C(U) \times L(U) \to D(U)$ by $(\eta,\eta') \mapsto \eta+\eta'$,
and $\Psi_2$ for 
$D(U) \to C(U) \times L(U)$ by 
$\theta \mapsto (\lambda_U(\theta),\lambda'_U(\theta))$.

We first check that $\Psi_1,\Psi_2$ are well-defined.
The well-definedness of $\Psi_1$ follows from 
$C(U)+L(U) \subset D(U)+D(U) \subset D(U)$.
By Proposition \ref{Prop_lambda_image} (1),
$C(U)$ and $L(U)$ are the images of $\lambda_U,\lambda'_U$ respectively,
so $\Psi_2$ is well-defined.

Now $\Psi_1\Psi_2$ is obviously identity.

It remains to show $\Psi_2\Psi_1$ is also identity.
Let $(\eta,\eta') \in C(U) \times L(U)$, and set $\theta:=\eta+\eta'$.
Then by Lemmas \ref{Lem_axiom_proj} (1) and \ref{Lem_lambda_DcV} (1),
we get $\lambda_U(\theta)=\eta+\lambda_U(\eta')$ and $\lambda_U(\eta')=0$.
Therefore $\lambda_U(\theta)=\eta$.
Then $\lambda'_U(\theta)=\theta-\lambda_U(\theta)=\theta-\eta=\eta'$ follow.
Thus $\Psi_2\Psi_1$ is identity.
 
(2)
By Lemma \ref{Lem_lambda_DcV} (3) and 
Proposition \ref{Prop_lambda_image} (1), we get 
\begin{align*}
\Psi_2(D(U) \cap D^\circ(U' \oplus V'))=
(C^\circ(U')+C(U/U')) \times (L(U) \cap D^\circ(V')).
\end{align*}
Thus we obtain the first bijection.
The second one is the case $U'=U$ and $V'=0$.
\end{proof}

\subsection{A canonical bijection $L(U) \to K_0(\proj B)_\R$}

Recall that we have defined a certain subset 
$\partial_i \subset \partial D(U)$ for each $i \in \{1,\ldots,n\}$
in \eqref{define partial 2}.
Extending its definition, for each subset $I \subset \{1,\ldots,m\}$, we set
\begin{align*}
U_I:=\bigoplus_{i \in I} U_i, \quad 
\partial_I:=\bigcap_{i \in I} \partial_i \subset D(U).
\end{align*}
Thus we have a family $(\partial_I)_{I \in 2^{\{1,\ldots,m\}}}$ 
of subsets of $D(U)$ indexed by the power set $2^{\{1,\ldots,m\}}$.
We consider this family in this subsection.
If $I \subset I'$, then $\partial_I \supset \partial_{I'}$ holds.
Moreover, we have $\partial_\emptyset=D(U)$ and 
$\partial_{\{1,\ldots,m\}}=J(U)$.

The subset $\partial_I \subset D(U)$ can be characterized by $\lambda_U$.

\begin{Prop}\label{Prop_partial_lambda}
Let $U=\bigoplus_{i=1}^m U_i \in \twopsilt A$. 
For each $I \subset \{1,\ldots,m\}$, we have the equality
\begin{align*}
\partial_I=D(U) \setminus \left( \bigcup_{i \in I} D^\circ(U_i) \right)
&=\{ \theta \in D(U) \mid \lambda_U(\theta) \in C(U/U_I)\}.
\end{align*}
\end{Prop}

\begin{proof}
The first equality follows from Theorem \ref{Thm_partial_setminus},
and the other is a result of Lemma \ref{Lem_lambda_DcV} (1).
\end{proof}

Then we have the following fundamental bijection on $\partial_I$.

\begin{Prop}\label{Prop_C_L_partial}
Let $U=\bigoplus_{i=1}^m U_i \in \twopsilt A$ and $I \subset \{1,\ldots,m\}$. 
\begin{enumerate}
\item
On $\partial_I$, the maps $\lambda_U$ and $\lambda_{U/U_I}$ coincide,
and so do $\lambda'_U$ and $\lambda'_{U/U_I}$.
\item
We have a bijection
\begin{align*}
C(U/U_I) \times L(U) \to \partial_I, \quad (\eta,\eta') \mapsto \eta+\eta',
\end{align*}
whose inverse is $\theta \mapsto (\lambda_U(\theta),\lambda'_U(\theta))
=(\lambda_{U/U_I}(\theta),\lambda'_{U/U_I}(\theta))$.
\item
We have $C(U/U_I)+\partial_I=\partial_I$; hence $C(U/U_I) \subset \partial_I$.
\end{enumerate}
\end{Prop}

\begin{proof}
(1)
Let $\theta \in \partial_I$.
By Proposition \ref{Prop_partial_lambda}, $\lambda_U(\theta) \in C(U/U_I)$.
On the other hand, 
$\lambda_U(\theta)=\lambda_{U/U_I}(\theta)+\lambda_{U_I}(\theta)$
follows from Lemma \ref{Lem_axiom_proj} (4).
Thus we get $\lambda_U(\theta)=\lambda_{U/U_I}(\theta)$.
Now the remaining part follows.

(2) 
By Proposition \ref{Prop_partial_lambda}, 
the bijections in Theorem \ref{Thm_C(U)_L(U)} (1) are restricted as desired.
By (1), $(\lambda_U(\theta),\lambda'_U(\theta))
=(\lambda_{U/U_I}(\theta),\lambda'_{U/U_I}(\theta))$ 
for any $\theta \in \partial_I$.

(3)
For any $\eta \in C(U/U_I)$ and $\theta \in \partial_I$,
$\lambda_U(\eta+\theta)=\eta+\lambda_U(\theta)$ 
by Lemma \ref{Lem_axiom_proj} (1).
Then Proposition \ref{Prop_partial_lambda} implies 
$\lambda_U(\theta) \in C(U/U_I)$,
so $\lambda_U(\eta+\theta) \in C(U/U_I)$.
By Proposition \ref{Prop_partial_lambda} again, 
$\eta+\theta \in \partial_I$.
Thus $C(U/U_I)+\partial_I \subset \partial_I$ holds,
and the converse is obvious.
\end{proof}

Before proceeding, note the following to avoid the confusion.

\begin{Rem}
The set $\partial_I+(-\partial_I)$ appearing in the next proposition 
is not necessarily a vector subspace of $K_0(\proj A)_\R$, 
because $\partial_I$ may not be convex.
We will not consider the vector subspace $\R \partial_I$ of $K_0(\proj A)_\R$ 
generated by $\partial_I$.
\end{Rem}

This set $\partial_I+(-\partial_I)$ satisfies 
the following relationship with $\R C(U)$.

\begin{Prop}\label{Prop_diff_partial}
Let $U=\bigoplus_{i=1}^m U_i \in \twopsilt A$.
For any subset $I \subset \{1,\ldots,m\}$, we have 
\begin{align*}
\R C(U) \cap (\partial_I+(-\partial_I))=\R C(U/U_I).
\end{align*}
In particular, $\R C(U) \cap (L(U)+(-L(U)))=\{0\}$ holds.
\end{Prop}

\begin{proof}
Let $\theta \in \R C(U) \cap (\partial_I+(-\partial_I))$.
Then we take $\theta_1,\theta_2 \in \partial_I$
such that $\theta=\theta_1-\theta_2$.
Since $\theta_1,\theta_2 \in D(U)$ and 
$\theta_2-\theta_1=\theta \in \R C(U)$,
Lemma \ref{Lem_axiom_proj} (1) gives
$\theta=\lambda_U(\theta_2)-\lambda_U(\theta_1)$.
By Proposition \ref{Prop_partial_lambda} and 
$\theta_1,\theta_2 \in \partial_I$, we have
$\lambda_U(\theta_1),\lambda_U(\theta_2) \in C(U/U_I)$.
Thus $\theta=\lambda_U(\theta_2)-\lambda_U(\theta_1) \in \R C(U/U_I)$ holds.
Therefore $\R C(U) \cap (\partial_I+(-\partial_I)) \subset \R C(U/U_I)$.

For the converse,
$\R C(U/U_I) \subset \R C(U)$ is clear,
and $\R C(U/U_I) \subset \partial_I+(-\partial_I)$ 
follows from $C(U/U_I) \subset \partial_I$ 
in Proposition \ref{Prop_C_L_partial} (3).
Thus we get $\R C(U/U_I) \subset \R C(U) \cap (\partial_I+(-\partial_I))$.

We have proved $\R C(U) \cap (\partial_I+(-\partial_I))=\R C(U/U_I)$.
For the last statement, set $I=\{1,\ldots,m\}$.
\end{proof}
Now we recall some properties. 
For the $\tau$-tilting reduction at $U$,
the algebra $B=B_U$ defined by \eqref{def B} 
before Proposition \ref{Prop_reduc} plays a very important role.
In Definition-Proposition \ref{Def-Prop_pi}
and Proposition \ref{Prop_reduc_K_0},
we obtained a surjective $\R$-linear map 
$\pi \colon K_0(\proj A)_\R \to K_0(\proj B)_\R$ 
which is compatible with the $\tau$-tilting reduction.
The map $\pi$ satisfies $\Ker \pi=\R C(U)$ 
and $\pi(D^\circ(U))=K_0(\proj B)_\R$.
The map $\pi$ induces a canonical bijection 
$L(U) \to K_0(\proj B)_\R$ as follows.

\begin{Thm}\label{Thm_pi_L(U)}
Let $U \in \twopsilt A$.
\begin{enumerate}
\item
We have 
$\pi \circ \lambda'_U=\pi|_{D(U)} 
\colon D(U) \to K_0(\proj B)_\R$.
\item
The restriction 
\begin{align*}
\pi|_{L(U)} \colon L(U) \to K_0(\proj B)_\R
\end{align*}
is bijective.
\end{enumerate}
\end{Thm}

\begin{proof}
(1)
Let $\theta \in D(U)$.
Then $\lambda'_U(\theta)=\theta-\lambda_U(\theta) \in \theta+\R C(U)$ hold,
so we get $\pi(\lambda'_U(\theta))=\pi(\theta)$.

(2)
The surjectivity follows as 
\begin{align*}
\pi(L(U))\stackrel{\text{Prop.~\ref{Prop_lambda_image}}}{=}
\pi(\lambda'_U(D(U)))\stackrel{\text{(1)}}{=}
\pi(D(U))=K_0(\proj B)_\R.
\end{align*}

For the injectivity, 
let $\theta,\eta \in L(U)$ satisfy $\pi(\theta)=\pi(\eta)$.
Then $\theta-\eta \in \R C(U)+(L(U)+(-L(U)))$ by (1).
This and Proposition \ref{Prop_diff_partial} give $\theta-\eta=0$ as desired.
\end{proof}

\subsection{The relationship between $C(U),L(U)$ and faces of $D(U)$}
For each $F \in \Face D(U)$, we set
\begin{align*}
I_F:=\{ i \in \{1,\ldots,m\} \mid F \subset \partial_i \}.
\end{align*}
For any subset $I \subset \{1,\ldots,m\}$, we define
\begin{align*}
\Face^\circ_I D(U):=\{F \in \Face D(U) \mid I_F=I\}, \quad
\Face_I D(U):=\{F \in \Face D(U) \mid I_F \supset I\}.
\end{align*}
For example, $\Face^\circ_\emptyset D(U)=\{D(U)\}$ and
$\Face_\emptyset D(U)=\Face D(U)$ hold.

Clearly we have a decomposition
\begin{align*}
\Face D(U)=\bigsqcup_{I \subset \{1,\ldots,m\}}
\Face^\circ_I D(U).
\end{align*}
This subsection is devoted to studying the relationship between
each $\Face^\circ_I D(U)$ and $C(U),L(U)$ in this subsection
by using the maps $\lambda_U,\lambda'_U$.

We begin with the following fundamental properties of faces.

\begin{Prop}\label{Prop_C(U/U_I_F)}
Let $U=\bigoplus_{i=1}^m U_i \in \twopsilt A$ and $F \in \Face D(U)$.
\begin{enumerate}
\item
We have 
\begin{align*}
\lambda_U(F)=F\cap C(U)=C(U/U_{I_F}), \quad
\lambda'_U(F)=F \cap L(U), \quad
\pi(F)=\pi(F \cap L(U)).
\end{align*}
\item
There exist mutually inverse bijections 
\begin{align*}
(F \cap C(U)) \times (F \cap L(U)) &\to F, 
& F &\to (F \cap C(U))\times(F \cap L(U)), \\
(\eta,\eta') & \mapsto \eta+\eta', 
& \theta &\mapsto (\lambda_U(\theta),\lambda'_U(\theta)).
\end{align*}
\end{enumerate}
\end{Prop}

\begin{proof} 
(1)
For any $\theta \in F$, 
clearly $\theta=\lambda_U(\theta)+\lambda'_U(\theta)$ holds.
Since $F$ is a face of $D(U)$ and 
$\lambda_U(\theta),\lambda'_U(\theta) \in D(U)$,
we have $\lambda_U(\theta),\lambda'_U(\theta) \in F$.
Thus $\lambda_U(F),\lambda'_U(F) \subset F$.

Since $\lambda_U(F) \subset F$, 
we have $\lambda_U(F) \subset F \cap C(U)$.
The converse $F \cap C(U) \subset \lambda_U(F)$
follows from that $\lambda_U$ is identity on $C(U)$
in Proposition \ref{Prop_lambda_image} (2).
Thus $\lambda_U(F)=F \cap C(U)$ holds.

By Proposition \ref{Prop_facet_between},
$I_F$ equals $\{ i \in \{1,\ldots,m\} \mid [U_i] \notin F\}$.
Since $C(U) \subset D(U)$ and $F \in \Face D(U)$,
we have $F \cap C(U)$ is a face of $C(U)$.
These imply $F \cap C(U)=C(U/U_{I_F})$.

Since $\lambda'_U$ is identity on $L(U)$ 
by Proposition \ref{Prop_lambda_image} (2),
$\lambda'_U(F)=F \cap L(U)$ can be proved similarly to 
$\lambda_U(F)=F \cap C(U)$.

Then $\pi(F \cap L(U))=\pi(\lambda'_U(F))$, 
and it is $\pi(F)$ by Theorem \ref{Thm_pi_L(U)} (1).

(2) 
In view of Theorem \ref{Thm_C(U)_L(U)} (1),
it suffices to show that, for $\theta \in D(U)$,
the condition $\theta \in F$ holds if and only if 
$\lambda_U(\theta) \in F \cap C(U)$ 
and $\lambda'_U(\theta) \in F \cap L(U)$.
The ``if'' part is obvious, since $F$ is closed under sums.
The ``only if'' part is (1).
\end{proof}

Then $\Face^\circ_I D(U)$ is characterized by intersections with $C(U)$
as follows.

\begin{Prop}\label{Prop_Face_I_iff}
Let $U=\bigoplus_{i=1}^m U_i \in \twopsilt A$,
$F \in \Face D(U)$ and $I \subset \{1,\ldots,m\}$.
\begin{enumerate}
\item
The condition $F \in \Face^\circ_I D(U)$ holds if and only if 
$F \cap C(U)=C(U/U_I)$.
\item
The condition $F \in \Face_I D(U)$ holds if and only if 
$F \cap C(U) \subset C(U/U_I)$.
\end{enumerate}
\end{Prop}

\begin{proof}
Let $F \in \Face D(U)$.
Then by Proposition \ref{Prop_C(U/U_I_F)} (1),
$F \cap C(U)=C(U/U_I)$ is equivalent to $I_F=I$, and similarly, 
$F \cap C(U) \subset C(U/U_I)$ is equivalent to $I_F \supset I$.
Thus we have the equalities on $\Face^\circ_I D(U)$
and $\Face_I D(U)$.
\end{proof}

The set $\Face^\circ_I D(U)$ has another description
in terms of $\lambda_U(\theta)$ for $\theta \in F^\circ$.

\begin{Prop}\label{Prop_lambda_F^circ}
Let $U=\bigoplus_{i=1}^m U_i \in \twopsilt A$.
For any $F \in \Face D(U)$ and $I \subset \{1,\ldots,m\}$,
the following conditions are equivalent.
\begin{enumerate2}
\item
The face $F$ belongs to $\Face^\circ_I D(U)$.
\item
For any $\theta \in F^\circ$, we have $\lambda_U(\theta) \in C^\circ(U/U_I)$.
\item
There exists $\theta \in F^\circ$ such that 
$\lambda_U(\theta) \in C^\circ(U/U_I)$.
\end{enumerate2}
In this case, we have $\lambda_U(F^\circ)=C^\circ(U/U_I)$.
\end{Prop}

\begin{proof}
(a)$\Rightarrow$(b)
Let $F \in \Face^\circ_I D(U)$.
Proposition \ref{Prop_C(U/U_I_F)} (1) implies 
$\lambda_U(F^\circ) \subset \lambda_U(F)=C(U/U_I)$.
Moreover Proposition \ref{Prop_Face_I_iff} 
implies $C^\circ(U/U_I) \subset F$.
Thus each element $\theta \in F^\circ$ has $\eta \in C^\circ(U/U_I)$
such that $\theta-\eta \in F^\circ$.
Then Lemma \ref{Lem_axiom_proj} (1) gives
$\lambda_U(\theta)=\eta+\lambda_U(\theta-\eta)
\in C^\circ(U/U_I)+C(U/U_I)=C^\circ(U/U_I)$.

(b)$\Rightarrow$(c) is obvious.

(c)$\Rightarrow$(a)
Take such $\theta \in F^\circ$.
By definition, we have $F \in \Face^\circ_{I_F} D(U)$.
Then by applying (a)$\Rightarrow$(b) to $I_F$ instead of $I$,
we have $\lambda_U(\theta) \in C^\circ(U/U_{I_F})$.
This and (c) yield $I_F=I$, and we get (a).

For the last statement, 
we have $\lambda_U(F^\circ) \subset C^\circ(U/U_I)$ by (b).
To show the converse, let $\eta \in C^\circ(U/U_I)$.
By (b), we can take $\theta \in F^\circ$ such that 
$\eta-\lambda_U(\theta) \in C(U/U_I)$.
Then we have $\lambda_U((\eta-\lambda_U(\theta))+\theta)=
(\eta-\lambda_U(\theta))+\lambda_U(\theta)=\eta$
by Lemma \ref{Lem_axiom_proj} (1).
Moreover $(\eta-\lambda_U(\theta))+\theta
\in C(U/U_I)+F^\circ \subset F+F^\circ=F^\circ$ hold,
where the inclusion is by Proposition \ref{Prop_C(U/U_I_F)} (1).
\end{proof}

We remark that, even if $F \in \Face^\circ_I D(U)$,
the condition $C^\circ(U/U_I) \subset F^\circ$ does not hold in general.
In Example \ref{Ex_A4_Sigma_I}, $F_{345} \in \Face^\circ_{\{2\}} F(U)$ holds,
but $[U/U_2]=[U_1]=\theta_5$ is not in its relative interior $F_{345}^\circ$.

The maps $\lambda_U,\lambda'_U$ can used also to study the set $\partial_I$.
By definition, we have 
\begin{align*}
\Face_I D(U)=\{F \in \Face D(U) \mid F \subset \partial_I\}.
\end{align*}
Since $\partial_I$ is a union of faces in $D(U)$
by Lemma \ref{Lem_partial_face} (1), it is easy to see
\begin{align}\label{partial_I union face}
\partial_I=
\bigcup_{F \in \Face_I D(U)} F
=\bigcup_{F \in \max \Face_I D(U)} F.
\end{align}
We show that $\partial_I$ is also 
the union of all $F \in \Face^\circ_I D(U)$ as follows.

\begin{Prop}\label{Prop_max_Face_circ}
Let $U=\bigoplus_{i=1}^m U_i \in \twopsilt A$
and $I \subset \{1,\ldots,m\}$.
\begin{enumerate}
\item
For any $F \in \Face_I D(U)$, we have 
$\lambda_U(F+C(U/U_I))=C(U/U_I)$ and $F+C(U/U_I) \subset \partial_I$.
\item
The inclusion $\max \Face_I D(U) \subset \Face^\circ_I D(U)$ holds.
Thus we get
\begin{align*}
\partial_I=\bigcup_{F \in \Face^\circ_I D(U)} F.
\end{align*}
\end{enumerate}
\end{Prop}

\begin{proof}
(1)
We have $\lambda_U(F)=F \cap C(U) \subset C(U/U_I)$ 
by Proposition \ref{Prop_C(U/U_I_F)} (1).
Then we have
\begin{align*}
\lambda_U(F+C(U/U_I))\stackrel{\text{Lem.~\ref{Lem_axiom_proj} (1)}}{=}
\lambda_U(F)+C(U/U_I)=C(U/U_I).
\end{align*}
By Proposition \ref{Prop_partial_lambda}, 
we get $F+C(U/U_I) \subset \partial_I$.

(2)
Let $F \in \max \Face_I D(U)$.
Take the smallest face $F'$ containing $F+C(U/U_I)$.
Since $\partial_I$ is a union of faces
by Lemma \ref{Lem_partial_face} (1) and $F+C(U/U_I)$ is convex,
(1) implies $F \subset F+C(U/U_I) \subset F' \subset \partial_I$.
By the maximality of $F$ in $\Face_I D(U)$,
we have $F=F+C(U/U_I)$.
By (1) again, $\lambda_U(F)=\lambda_U(U/U_I)=C(U/U_I)$ hold,
so Proposition \ref{Prop_Face_I_iff} gives $F \in \Face^\circ_I D(U)$.
\end{proof}

The converse 
$\Face^\circ_I D(U) \subset \max \Face_I D(U)$ does not hold in general.
In Example \ref{Ex_A4_Sigma_I} later,
$F_5=C(U_1)$ belongs to $\Face^\circ_{\{2\}} D(U)$,
but it is not maximal in $\Face_{\{2\}} D(U)$.

\subsection{Generalized fans in $K_0(\proj B)_\R$ 
induced by some faces of $D(U)$}\label{Subsec fan}

For each $I \subset \{1,\ldots,m\}$, 
we defined the subsets 
$\Face^\circ_I D(U)$ and $\Face_I D(U)$ of $\Face D(U)$.
We set
\begin{align*}
\Sigma_I:=\{ \pi(F) \mid F \in \Face^\circ_I D(U) \}.
\end{align*}
In this subsection,
we will show that $\Sigma_I$ is 
a finite complete generalized fan in $K_0(\proj B)_\R$.

We first observe that 
$\Face_I D(U)$ is a generalized fan 
whose support is $\partial_I$ in $K_0(\proj A)_\R$.

\begin{Prop}\label{Prop_Face_I_fan}
Let $U=\bigoplus_{i=1}^m U_i \in \twopsilt A$.
For any subset $I \subset \{1,\ldots,m\}$, 
the set $\Face_I D(U)$ is a generalized fan in $K_0(\proj A)_\R$
whose support is $\partial_I$.
\end{Prop}

\begin{proof}
The set $\Face_I D(U)$ is clearly closed under taking faces.
Moreover$F_1 \cap F_2$ of any two $F_1,F_2 \in \Face_I D(U)$ 
is a face of $F_1$ and a face of $F_2$, 
because $D(U)$ is a polyhedral cone.
Thus $\Face_I D(U)$ is a generalized fan.
Its support is $\partial_I$ by \eqref{partial_I union face}.
\end{proof}

The following property of intersections of faces is important.

\begin{Lem}\label{Lem_pi_cap}
Let $U \in \twopsilt A$.
For any subset $\frakF \subset \Face D(U)$, we have
\begin{align*}
\pi\left( \bigcap_{F \in \frakF} F \right)=\bigcap_{F \in \frakF} \pi(F).
\end{align*}
\end{Lem}

\begin{proof}
By Proposition \ref{Prop_C(U/U_I_F)} (1), it suffices to show 
\begin{align*}
\pi\left( \left( \bigcap_{F \in \frakF} F \right) \cap L(U) \right)
=\bigcap_{F \in \frakF} \pi(F \cap L(U)).
\end{align*}
This follows as
\begin{align*}
\pi\left( \left( \bigcap_{F \in \frakF} F \right) \cap L(U) \right)
=\pi\left( \bigcap_{F \in \frakF} (F \cap L(U)) \right)
\stackrel{\text{Thm.~\ref{Thm_pi_L(U)} (2)}}{=}
\bigcap_{F \in \frakF} \pi(F \cap L(U)),
\end{align*}
from the injectivity of $\pi|_{L(U)}$ proved 
in Theorem \ref{Thm_pi_L(U)} (2).
\end{proof}

By using this, we have the following property. 
In paticular, $\Sigma_I$ is refined if $I$ gets larger.

\begin{Prop}\label{Prop_Sigma_I_cap}
Let $U=\bigoplus_{i=1}^m U_i \in \twopsilt A$,
and $I_1,I_2 \subset \{1,\ldots,m\}$ be subsets.
Set $I:=I_1 \cup I_2$.
\begin{enumerate}
\item
There exists a surjection
\begin{align*}
\Face^\circ_{I_1} D(U) \times \Face^\circ_{I_2} D(U) \to
\Face^\circ_I D(U); \quad
(F_1,F_2) \mapsto F_1 \cap F_2.
\end{align*}
\item
We have
\begin{align*}
\Sigma_I=\{\sigma_1 \cap \sigma_2 \mid 
\sigma_1 \in \Sigma_{I_1}, \ \sigma_2 \in \Sigma_{I_2}\}.
\end{align*}
In particular, 
$\Sigma_I$ is a refinement of both $\Sigma_{I_1}$ and $\Sigma_{I_2}$.
\end{enumerate}
\end{Prop}

\begin{proof}
(1) 
This map is well-defined by Proposition \ref{Prop_Face_I_iff}.

To check the surjectivity, let $F \in \Face^\circ_I D(U)$.
We set $\frakF:=\{F' \in \Facet D(U) \mid F \subset F'\}$.
Then $F=\bigcap_{F' \in \frakF} F'$ is clear.
For each $j \in \{1,2\}$, we set
$\frakF_j:=\{F' \in \frakF \mid i_{F'}\in I_j\}$.
We show $\frakF=\frakF_1 \cup \frakF_2$.

The inclusion $\frakF_1 \cup \frakF_2 \subset \frakF$ is obvious.

To show the converse, let $F' \in \frakF$, 
and take the unique $i \in \{1,\ldots,m\}$ such that $F' \in \Facet_i D(U)$
by Theorem \ref{Thm_partial_facet} (2).
It suffices to show $i \in I$, because $I=I_1 \cup I_2$ by definition.
We have $F' \cap C(U)=C(U/U_i)$ 
by Proposition \ref{Prop_facet_X_Y_almost} (1).
Since $F \in \Face^\circ_I D(U)$, 
we get $F \cap C(U)=C(U/U_I)$ by Proposition \ref{Prop_C(U/U_I_F)} (1).
Then we have $C(U/U_I)=F \cap C(U) \subset F' \cap C(U)=C(U/U_i)$,
so we get $i \in I$.

Thus $\frakF=\frakF_1 \cup \frakF_2$ holds.
Setting $F_j:=\bigcap_{F' \in \frakF_j} F'$ for each $j \in \{1,2\}$,
we get $F=F_1 \cap F_2$.

Let $j \in \{1,2\}$.
It remains to show $F_j \in \Face^\circ_{I_j} D(U)$.
By assumption, $F \subset \partial_I \subset \partial_{I_j}$,
so for any $i \in I_j$, there exists $F' \in \Facet_i D(U)$
such that $F \subset F'$ by Theorem \ref{Thm_partial_facet} (1).
Thus $\{i_{F'} \mid F' \in \frakF_j\}=I_j$ holds.
Then we have $F_j \in \Face^\circ_{I_j} D(U)$ 
by Proposition \ref{Prop_Face_I_iff},
because $I_{F'}=\{i_{F'}\}$ for any $F' \in \Facet D(U)$.

(2) follows as
\begin{align*}
\Sigma_I&=\{\pi(F) \mid F \in \Face^\circ_I D(U)\}\\
&\stackrel{(1)}{=}
\{\pi(F_1 \cap F_2) \mid 
F_1 \in \Face^\circ_{I_1} D(U), \ F_2 \in \Face^\circ_{I_2} D(U)\}\\
&\stackrel{\text{Lem.~\ref{Lem_pi_cap}}}{=}
\{\pi(F_1) \cap \pi(F_2) \mid 
F_1 \in \Face^\circ_{I_1} D(U), \ F_2 \in \Face^\circ_{I_2} D(U)\}\\
&=\{\sigma_1 \cap \sigma_2 \mid 
\sigma_1 \in \Sigma_{I_1}, \ \sigma_2 \in \Sigma_{I_2}\}.
\qedhere
\end{align*}
\end{proof}

To show that $\Sigma_I$ is a generalized fan, 
we need the following general property.

\begin{Lem}\label{Lem_fiber}
Let $U=\bigoplus_{i=1}^m U_i \in \twopsilt A$.
For each subset $I \subset \{1,\ldots,m\}$, the following hold.
\begin{enumerate}
\item
The map $\pi|_{\partial_I} \colon \partial_I \to K_0(\proj B)_\R$ 
is surjective.
\item
Let $\theta \in D(U)$.
Then we have 
$\pi^{-1}(\pi(\theta)) \cap \partial_I=C(U/U_I)+\lambda'_U(\theta)$.
\item
Let $F \in \Face_I D(U)$.
Then $F \in \Face^\circ_I D(U)$ holds if and only if
$\pi^{-1}(\pi(F)) \cap \partial_I=F$.
\end{enumerate}
\end{Lem}

\begin{proof}
(1) 
By definition,
$L(U)=\partial_{\{1,\ldots,m\}}\subset \partial_I$.
Thus Theorem \ref{Thm_pi_L(U)} (2) gives the assertion.

(2)
Note that $\theta \in \partial_I$ implies $\lambda_U(\theta) \in C(U/U_I)$
by Proposition \ref{Prop_partial_lambda}.

First, let $\eta \in \pi^{-1}(\pi(\theta)) \cap \partial_I$.
Then we have $\eta=\lambda_U(\theta)+\lambda'_U(\theta) \in 
C(U/U_I)+\lambda'_U(\theta)$.
Therefore $\pi^{-1}(\pi(\theta)) \cap \partial_I \subset 
C(U/U_I)+\lambda'_U(\theta)$ holds.

Conversely, assume $\eta \in C(U/U_I)+\lambda'_U(\theta)$.
Since $\lambda_U(\theta) \in C(U/U_I)$, we have 
\begin{align*}
\eta \in C(U/U_I)+\lambda'_U(\theta)
=C(U/U_I)-\lambda_U(\theta)+\theta \subset \R C(U/U_I)+\theta.
\end{align*}
Thus Definition-Proposition \ref{Def-Prop_pi} gives 
$\pi(\eta)=\pi(\theta)$, so $\eta \in \pi^{-1}(\pi(\theta))$.
Moreover we get
$\lambda'_U(\theta) \in L(U) \subset \partial_I$
by Proposition \ref{Prop_lambda_image} (1).
Thus Proposition \ref{Prop_C_L_partial} (3) implies
$\eta \in C(U/U_I)+\lambda'_U(\theta) \subset C(U/U_I)+\partial_I=\partial_I$.
Thus $\eta \in \pi^{-1}(\pi(\theta)) \cap \partial_I$ holds,
so we have $C(U/U_I)+\lambda'_U(\theta) \subset 
\pi^{-1}(\pi(\theta)) \cap \partial_I$.

(3)
Under the assumption $F \in \Face_I D(U)$,
Proposition \ref{Prop_Face_I_iff} implies that
$F \in \Face^\circ_I D(U)$ is equivalent to $C(U/U_I) \subset F$.
Since $\lambda'_U(F)=F \cap L(U) \subset F$ 
by Proposition \ref{Prop_C(U/U_I_F)} (1),
$C(U/U_I) \subset F$ holds if and only if $C(U/U_I)+\lambda'_U(F) \subset F$.
The latter is equivalent to 
$\pi^{-1}(\pi(F)) \cap \partial_I \subset F$ by (2),
and also to $\pi^{-1}(\pi(F)) \cap \partial_I=F$,
because $F \subset \pi^{-1}(\pi(F)) \cap \partial_I$ always holds
as long as $F \in \Face_I D(U)$.
\end{proof}

Then we can show that $\Sigma_I$ is a generalized fan in $K_0(\proj B)_\R$.

\begin{Prop}\label{Prop_Sigma_I_fan}
Let $U=\bigoplus_{i=1}^m U_i \in \twopsilt A$.
For each subset $I \subset \{1,\ldots,m\}$, 
the set $\Sigma_I$ is a finite complete generalized fan in $K_0(\proj B)_\R$,
and there exist mutually inverse bijections
\begin{align*}
\Face^\circ_I D(U) & \to \Sigma_I, &
F & \mapsto \pi(F), \\
\Sigma_I & \to \Face^\circ_I D(U), &
\sigma & \mapsto \pi^{-1}(\sigma) \cap \partial_I
\end{align*}
which preserve inclusions and intersections.
\end{Prop}

\begin{proof}
By Proposition \ref{Prop_Face_I_fan},
$\Face_I D(U)$ is a generalized fan in $K_0(\proj A)_\R$
whose support is $\partial_I$.
By Lemma \ref{Lem_fiber} (3), we have
\begin{align*}
\Face^\circ_I D(U)=\{F \in \Face_I D(U) \mid
\pi^{-1}(\pi(F)) \cap \partial_I=F\}.
\end{align*}
Thus by Lemma \ref{Lem_induce_fan} (1)(2), 
$\Sigma_I$ is a finite generalized fan in $K_0(\proj B)_\R$,
and we obtain the desired bijections.
These bijections clearly preserve inclusions,
and preserve intersections by Lemma \ref{Lem_pi_cap}.

By Propositon \ref{Prop_max_Face_circ} (2),
the support of $\Sigma_I$ in $K_0(\proj B)_\R$ is $\pi(\partial_I)$ 
by Lemma \ref{Lem_induce_fan} (3).
Then Lemma \ref{Lem_fiber} (1) gives $\pi(\partial_I)=K_0(\proj B)_\R$,
that is, $\Sigma_I$ is complete.
\end{proof}

As in Theorem \ref{Thm_pi_L(U)} (2),
the restriction $\pi|_{L(U)} \colon L(U) \to K_0(\proj B)_\R$ is bijective,
so we define a map $\rho \colon K_0(\proj B)_\R \to L(U)$ as its inverse.
Then $\rho$ is a piecewise $\R$-linear map,
since $L(U)$ is a union of faces of $D(U)$.

We have the following relationship on the faces
in $\Face^\circ_I D(U)$ and $\Sigma_I$.

\begin{Prop}\label{Prop_Sigma_I_dim}
Let $U=\bigoplus_{i=1}^m U_i \in \twopsilt A$,
and $I \subset \{1,\ldots,m\}$ be a subset. 
For any $F \in \Face^\circ_I D(U)$, we have
\begin{align*}
\dim_\R F=(m-\#I)+\dim_\R \pi(F), \quad F=C(U/U_I)+\rho(\pi(F)).
\end{align*}
\end{Prop}

\begin{proof}
On the dimensions, we have
\begin{align*}
\dim_\R F=\dim_\R \R F
&=\dim_\R (\Ker \pi \cap \R F)+\dim_\R \pi(\R F)\\
&\stackrel{\text{Def.-Prop.~\ref{Def-Prop_pi}}}{=}
\dim_\R (\R C(U) \cap \R F)+\dim_\R \pi(\R F)\\
&=\dim_\R (\R C(U) \cap \R F)+\dim_\R \pi(F).
\end{align*}
Moreover, $C(U/U_I)=F \cap C(U)$ follows 
from Proposition \ref{Prop_Face_I_iff},
so we get
\begin{align*}
\R C(U/U_I) \subset \R C(U) \cap \R F
\stackrel{F \in \Face^\circ_I D(U)}{\subset} 
\R C(U) \cap (\partial_I+(-\partial_I)) 
\stackrel{\text{Prop.~\ref{Prop_diff_partial}}}{\subset}
\R C(U/U_I).
\end{align*}
Thus $\R C(U) \cap \R F=\R C(U/U_I)$, so its dimension is $m-\#I$.

We next show $F=C(U/U_I)+\rho(\pi(F))$.
Theorem \ref{Thm_pi_L(U)} (1) 
and $L(U)=\lambda'_U(D(U))$ from Proposition \ref{Prop_lambda_image} (1) imply
$\pi|_{L(U)} \circ \lambda'_U=\pi \colon D(U) \to K_0(\proj B)_\R$,
so we have $\lambda'_U=\rho \circ \pi \colon D(U) \to L(U)$, and
\begin{align*}
F &\stackrel{\text{Prop.~\ref{Prop_C(U/U_I_F)}}}{=} C(U/U_I)+\lambda'_U(F)
\stackrel{\lambda'_U=\rho \circ \pi}{=} C(U/U_I)+\rho(\pi(F)).
\qedhere
\end{align*}
\end{proof}

Now we have the main result in this subsection.

\begin{Thm}\label{Thm_Sigma_I}
Let $U=\bigoplus_{i=1}^m U_i \in \twopsilt A$,
and $I \subset \{1,\ldots,m\}$ be a subset. 
Then $\Sigma_I$ is a finite complete generalized fan in $K_0(\proj B)_\R$,
and there exist mutually inverse bijections
\begin{align*}
\Face^\circ_I D(U) & \to \Sigma_I, &
F & \mapsto \pi(F), \\
\Sigma_I & \to \Face^\circ_I D(U), &
\sigma & \mapsto \pi^{-1}(\sigma) \cap \partial_I=C(U/U_I)+\rho(\sigma)
\end{align*}
which preserve inclusions and intersections.
Moreover for any $F \in \Face^\circ_I D(U)$, we have
\begin{align*}
\dim_\R F=(m-\#I)+\dim_\R \pi(F).
\end{align*}
\end{Thm}

\begin{proof}
The assertions follow from 
Propositions \ref{Prop_Sigma_I_fan} and \ref{Prop_Sigma_I_dim}.
\end{proof}

Applying Theorem \ref{Thm_Sigma_I} 
to all subsets $I \subset \{1,\ldots,m\}$,
we have the following results.

\begin{Cor}\label{Cor_Sigma_I_all}
Let $U=\bigoplus_{i=1}^m U_i \in \twopsilt A$.
\begin{enumerate}
\item
There exists a bijection
\begin{align*}
\{ (I,\sigma) \mid I \subset \{1,\ldots,m\}, \ \sigma \in \Sigma_I \}
&\to \Face D(U), \\
(I,\sigma) &\mapsto C(U/U_I)+\rho(\sigma),
\end{align*}
whose inverse map sends $F \in \Face D(U)$ to 
$(I_F,\pi(F))$.
\item
Let $F \in \Face D(U)$ correspond to $(I,\sigma)$ in (1).
Then we have 
\begin{align*}
&I_F=I,\quad F \in \Face^\circ_I D(U), \quad \pi(F)=\sigma,\\
&\dim_\R F=(m-\#I)+\dim_\R \sigma, \quad
\lambda_U(F)=C(U/U_I), \quad \lambda'_U(F)=\rho(\sigma).
\end{align*}
\item
Let $F_1,F_2 \in \Face D(U)$ correspond to 
$(I_1,\sigma_1),(I_2,\sigma_2)$ in (1).
Then $F_1 \subset F_2$ holds if and only if
$I_1 \supset I_2$ and $\sigma_1 \subset \sigma_2$. 
\end{enumerate}
\end{Cor}

\begin{proof}
(1) follows from Theorem \ref{Thm_Sigma_I}, since $\Face D(U)$ 
is the disjoint union of all $\Face^\circ_I D(U)$.

(2)
We have $I_F=I$, $F \in \Face^\circ_I D(U)$ and $\pi(F)=\sigma$ by (1).
Then $\dim_\R F=(m-\#I)+\dim_\R \sigma$ follows 
from Theorem \ref{Thm_Sigma_I}.
Moreover $F \in \Face^\circ_I D(U)$ implies
$\lambda_U(F)=C(U/U_I)$ by Proposition \ref{Prop_Face_I_iff}.

It remains to show $\lambda'_U(F)=\rho(\sigma)$. 
By Proposition \ref{Prop_lambda_image} (1),
$\lambda'_U(F) \subset L(U)$ holds,
so we get $\lambda'_U(F)=\rho(\pi(\lambda'_U(F))$.
Its right-hand side is $\rho(\pi(F))$ by Theorem \ref{Thm_pi_L(U)} (1),
and it is $\rho(\sigma)$.

(3)
The ``if'' part follows as
$F_1=C(U/U_{I_1})+\rho(\sigma_1) \subset C(U/U_{I_2})+\rho(\sigma_2)=F_2$.

For the ``only if'' part,
$F_1 \subset F_2$ and (2) imply $C(U/U_{I_1})
=\lambda_U(F_1) \subset \lambda_U(F_2)=C(U/U_{I_2})$, 
so $I_1 \supset I_2$.
Also $F_1 \subset F_2$ gives
$\sigma_1=\pi(F_1) \subset \pi(F_2)=\sigma_2$.
\end{proof}

We give an explicit example of Corollary \ref{Cor_Sigma_I_all}.

\begin{Ex}\label{Ex_A4_Sigma_I}
As in Example \ref{Ex_A4_D(U)}, 
let $A$ be the path algebra of the quiver $1 \to 2 \to 3 \to 4$,
and $U=U_1 \oplus U_2  \in \twopsilt A$
with $U_1=(P(4) \to P(3))$ and $U_2=P(1)$.
\begin{align*}
\begin{tikzpicture}[baseline=0pt,scale=0.8]
\node (F3)[coordinate,label= 90:{$\theta_3=[P(1)]-[P(2)]$}] at ( 0  , 3  ) {};
\node (F1)[coordinate,label=180:{$\theta_1=[P(2)]$}]        at (-3  , 0  ) {};
\node (F2)[coordinate,label=270:{$\theta_2=[P(3)]$}]        at (-1.5,-1  ) {};
\node (F5)[coordinate,label=  0:{$\theta_5=[U_1]=[P(3)]-[P(4)]$}] 
at ( 2  , 0  ) {};
\node (F4)[coordinate]                                      at ( 0.5, 1  ) {};
\node (U2)[coordinate,label=180:{$[U_2]/2 =[P(1)]/2$}]      at (-1.5, 1.5) {};
\node (F4')  at ( 0.5, 1  ) {};
\node (F4'') at ( 3.5, 1  ) {$\theta_4=[P(2)]-[P(4)]$};
\draw (F3) to (F1);
\draw (F3) to (F2);
\draw[dashed] (F3) to (F4);
\draw (F3) to (F5);
\draw (F1) to (F2);
\draw (F2) to (F5);
\draw[dashed] (F5) to (F4);
\draw[dashed] (F4) to (F1);
\draw[dashed,very thick] (U2) to (F5);
\draw[->] (F4'') to (F4');
\draw[fill=black] (F5) circle [radius=0.1];
\draw[fill=black] (U2) circle [radius=0.1];
\end{tikzpicture}.
\end{align*}

As in Example \ref{Ex_A4_L_F},
we set $F_i:=\R_{\ge 0}\theta_i$ for each $i \in \{1,\ldots,5\}$.
It is easy to see that $F_1,\ldots,F_5$
are all the distinct 1-dimensional faces of $D(U)$.
Moreover we abbreviate $F_1+F_3+F_4$ as $F_{134}$.

Under this notation, we determine $\Face^\circ_I D(U)$
for all subsets $I \subset \{1,2\}$.
The smallest face containing $C(U_1),C(U_2)$ are $F_5,F_{13}$,
respectively.
Thus for each $F \in \Face D(U)$, 
the subset $I_F \subset \{1,2\}$ is determined by, 
$1 \in I_F$ if and only if $F_5 \not \subset F$; and
$2 \in I_F$ if and only if $F_{13} \not \subset F$.
Therefore we have the following table showing
the $d$-dimensional faces in $\Face^\circ_I D(U)$.
\begin{center}
\begin{tabular}{c|ccccc}
$I$ & $d=0$ & $d=1$ & $d=2$ & $d=3$ & $d=4$ \\
\hline
$\emptyset$ & & & & & $F_{12345}$ \\
$\{1\}$ & & & $F_{13}$ & $F_{123},F_{134}$ & \\
$\{2\}$ & & $F_5$ & $F_{25},F_{35},F_{45}$ & $F_{1245},F_{235},F_{345}$ & \\
$\{1,2\}$ & $\{0\}$ & $F_1,F_2,F_3,F_4$ & $F_{12},F_{14},F_{23},F_{34}$
\end{tabular}
\end{center} 

In Example \ref{Ex_A4_D(U)}, we have obtained
$\pi(\theta_1)=[P(1)_B]$, $\pi(\theta_2)=[P(2)_B]$,
$\pi(\theta_3)=-[P(1)_B]$, $\pi(\theta_4)=[P(1)_B]-[P(2)_B]$.
Therefore $\Sigma_I=\pi(\Face^\circ_I D(U))$ for 
each $I \subset \{1,2\}$ is given as follows,
where $134$ means $\sigma_{134}:=\pi(F_{134})$ and so on.
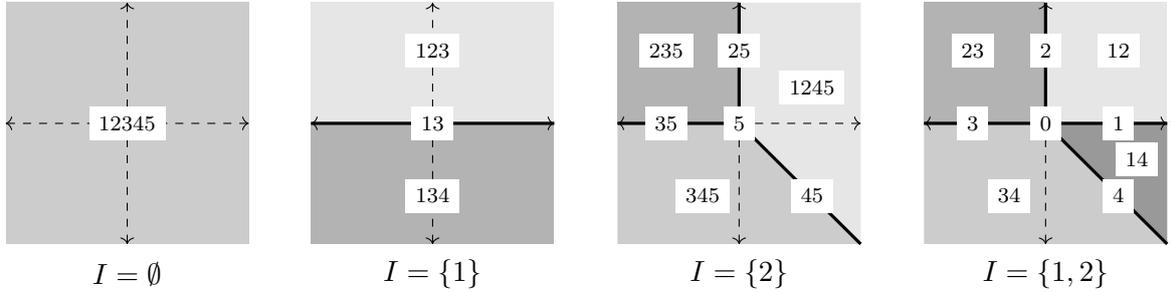
\begin{align*}
\begin{tikzpicture}[baseline=0pt,scale=0.8]
\node (00)[coordinate] at ( 0, 0) {};
\node (+0)[coordinate] at ( 2, 0) {};
\node (++)[coordinate] at ( 2, 2) {};
\node (0+)[coordinate] at ( 0, 2) {};
\node (-+)[coordinate] at (-2, 2) {};
\node (-0)[coordinate] at (-2, 0) {};
\node (--)[coordinate] at (-2,-2) {};
\node (0-)[coordinate] at ( 0,-2) {};
\node (+-)[coordinate] at ( 2,-2) {};
\node (I)              at ( 0,-2.5) {$I=\emptyset$};
\fill[black!20] (+-)--(++)--(-+)--(--)--cycle;
\draw[dashed,->] (00) to (+0);
\draw[dashed,->] (00) to (0+);
\draw[dashed,->] (00) to (-0);
\draw[dashed,->] (00) to (0-);
\node[fill=white] at (   0,   0) {$\scriptstyle{12345}$};
\end{tikzpicture}&&
\begin{tikzpicture}[baseline=0pt,scale=0.8]
\node (00)[coordinate] at ( 0, 0) {};
\node (+0)[coordinate] at ( 2, 0) {};
\node (++)[coordinate] at ( 2, 2) {};
\node (0+)[coordinate] at ( 0, 2) {};
\node (-+)[coordinate] at (-2, 2) {};
\node (-0)[coordinate] at (-2, 0) {};
\node (--)[coordinate] at (-2,-2) {};
\node (0-)[coordinate] at ( 0,-2) {};
\node (+-)[coordinate] at ( 2,-2) {};
\node (I)              at ( 0,-2.5) {$I=\{1\}$};
\fill[black!10] (+0)--(++)--(-+)--(-0)--cycle;
\fill[black!30] (+0)--(+-)--(--)--(-0)--cycle;
\draw[dashed,->] (00) to (+0);
\draw[dashed,->] (00) to (0+);
\draw[dashed,->] (00) to (-0);
\draw[dashed,->] (00) to (0-);
\draw[very thick] (-0) to (+0);
\node[fill=white] at (   0,   0) {$\scriptstyle{13}$};
\node[fill=white] at (   0, 1.2) {$\scriptstyle{123}$};
\node[fill=white] at (   0,-1.2) {$\scriptstyle{134}$};
\end{tikzpicture}&&
\begin{tikzpicture}[baseline=0pt,scale=0.8]
\node (00)[coordinate] at ( 0, 0) {};
\node (+0)[coordinate] at ( 2, 0) {};
\node (++)[coordinate] at ( 2, 2) {};
\node (0+)[coordinate] at ( 0, 2) {};
\node (-+)[coordinate] at (-2, 2) {};
\node (-0)[coordinate] at (-2, 0) {};
\node (--)[coordinate] at (-2,-2) {};
\node (0-)[coordinate] at ( 0,-2) {};
\node (+-)[coordinate] at ( 2,-2) {};
\node (I)              at ( 0,-2.5) {$I=\{2\}$};
\fill[black!10] (00)--(+-)--(++)--(0+)--cycle;
\fill[black!30] (00)--(0+)--(-+)--(-0)--cycle;
\fill[black!20] (00)--(-0)--(--)--(+-)--cycle;
\draw[dashed,->] (00) to (+0);
\draw[dashed,->] (00) to (0+);
\draw[dashed,->] (00) to (-0);
\draw[dashed,->] (00) to (0-);
\draw[very thick] (00) to (0+);
\draw[very thick] (00) to (-0);
\draw[very thick] (00) to (+-);
\node[fill=white] at (   0,   0) {$\scriptstyle{5}$};
\node[fill=white] at (   0, 1.2) {$\scriptstyle{25}$};
\node[fill=white] at (-1.2,   0) {$\scriptstyle{35}$};
\node[fill=white] at ( 1.2,-1.2) {$\scriptstyle{45}$};
\node[fill=white] at ( 1.2, 0.6) {$\scriptstyle{1245}$};
\node[fill=white] at (-1.2, 1.2) {$\scriptstyle{235}$};
\node[fill=white] at (-0.6,-1.2) {$\scriptstyle{345}$};
\end{tikzpicture}&&
\begin{tikzpicture}[baseline=0pt,scale=0.8]
\node (00)[coordinate] at ( 0, 0) {};
\node (+0)[coordinate] at ( 2, 0) {};
\node (++)[coordinate] at ( 2, 2) {};
\node (0+)[coordinate] at ( 0, 2) {};
\node (-+)[coordinate] at (-2, 2) {};
\node (-0)[coordinate] at (-2, 0) {};
\node (--)[coordinate] at (-2,-2) {};
\node (0-)[coordinate] at ( 0,-2) {};
\node (+-)[coordinate] at ( 2,-2) {};
\node (I)              at ( 0,-2.5) {$I=\{1,2\}$};
\fill[black!10] (00)--(+0)--(++)--(0+)--cycle;
\fill[black!30] (00)--(0+)--(-+)--(-0)--cycle;
\fill[black!20] (00)--(-0)--(--)--(+-)--cycle;
\fill[black!40] (00)--(+-)--(+0)--cycle;
\draw[dashed,->] (00) to (+0);
\draw[dashed,->] (00) to (0+);
\draw[dashed,->] (00) to (-0);
\draw[dashed,->] (00) to (0-);
\draw[very thick] (00) to (+0);
\draw[very thick] (00) to (0+);
\draw[very thick] (00) to (-0);
\draw[very thick] (00) to (+-);
\node[fill=white] at (   0,   0) {$\scriptstyle{0}$};
\node[fill=white] at ( 1.2,   0) {$\scriptstyle{1}$};
\node[fill=white] at (   0, 1.2) {$\scriptstyle{2}$};
\node[fill=white] at (-1.2,   0) {$\scriptstyle{3}$};
\node[fill=white] at ( 1.2,-1.2) {$\scriptstyle{4}$};
\node[fill=white] at ( 1.2, 1.2) {$\scriptstyle{12}$};
\node[fill=white] at (-1.2, 1.2) {$\scriptstyle{23}$};
\node[fill=white] at (-0.6,-1.2) {$\scriptstyle{34}$};
\node[fill=white] at ( 1.5,-0.6) {$\scriptstyle{14}$};
\end{tikzpicture}.
\end{align*}
The following table gives the $d$-dimensional faces in $\Sigma_I$.
\begin{center}
\begin{tabular}{c|ccc}
$I$ & $d=0$ & $d=1$ & $d=2$ \\
\hline
$\emptyset$ & & & $\sigma_{12345}$ \\
$\{1\}$ & & $\sigma_{13}$ & $\sigma_{123},\sigma_{134}$ \\
$\{2\}$ & $\sigma_5$ & $\sigma_{25},\sigma_{35},\sigma_{45}$ & 
$\sigma_{1245},\sigma_{235},\sigma_{345}$ \\
$\{1,2\}$ & $\{0\}$ & $\sigma_1,\sigma_2,\sigma_3,\sigma_4$ & 
$\sigma_{12},\sigma_{14},\sigma_{23},\sigma_{34}$
\end{tabular}
\end{center} 

By the two tables,
we can see the equality $\dim_\R F=(2-\#I)+\dim_\R \pi(F)$
in Corollary \ref{Cor_Sigma_I_all} (2) surely holds
for each $I \subset \{1,2\}$ and $F \in \Face^\circ_I D(U)$.

We can also check that $\Sigma_{\{1,2\}}$
consists of $\sigma_1 \cap \sigma_2$ 
for all $\sigma_1 \in \Sigma_{\{1\}}$ and $\sigma_2 \in \Sigma_{\{2\}}$
as shown in Proposition \ref{Prop_Sigma_I_cap}.
\end{Ex}

We have the following useful information on maximal elements.

\begin{Cor}\label{Cor_max_Face_dim}
Let $U=\bigoplus_{i=1}^m U_i \in \twopsilt A$,
and $I \subset \{1,\ldots,m\}$ be a subset.
For any $F \in \Face^\circ_I D(U)$,
the following conditions are equivalent.
\begin{enumerate2}
\item
The polyhedral cone $\pi(F)$ is in $\max \Sigma_I$.
\item
The face $F$ is in $\max \Face_I D(U)$.
\item
We have $\dim_\R \pi(F)=n-m$.
\item
We have $\dim_\R F=n-\#I$.
\end{enumerate2}
Therefore if $I=\{i\}$, 
then $\max \Face_{\{i\}} D(U)=\Facet_i D(U)$ holds.
\end{Cor}

\begin{proof}
Since the bijections in Theorem \ref{Thm_Sigma_I} preserve inclusions,
(a) and (b) are equivalent.
The dimension formula in Theorem \ref{Thm_Sigma_I} impies that
(c) and (d) are equivalent.
Moreover (a) and (c) are equivalent,
because $\Sigma_I$ is a finite complete generalized fan 
in $K_0(\proj B)_\R \simeq \R^{n-m}$ by Theorem \ref{Thm_Sigma_I}.

The last statement comes from (b)$\Leftrightarrow$(d).
\end{proof}

\subsection{Recovering faces of $D(U)$ from 
$M$-TF equivalences on $K_0(\proj B)_\R$}
\label{subsec recover}

In this subsection, we will explain that the faces of $D(U)$
can be recovered from the $B$-modules $M_I$ for all subsets
$I \subset \{1,\ldots,m\}$ defined below.

For this purpose, we consider the 2-term simple-minded collections
$X=\bigoplus_{i=1}^n X_i := \SH(S)$ and 
$Y=\bigoplus_{i=1}^n Y_i := \SH(T)$,
corresponding to the maximal and the minimal completions $S$ and $T$ of $U$
in Notation \ref{Nota_X_Y}.
To construct $M_I$, we use the triangle
$W_i \to Y_i \to X_i \to W_i[1]$ 
with $W_i \in \Filt \{X_{m+1},\ldots,X_n\}$ for each $i \in \{1,\ldots,m\}$.

By Proposition \ref{Prop_sim_W_U},
$\Filt \{X_{m+1},\ldots,X_n\}$ coincides with $\calW_U$ 
for the $\tau$-tilting reduction.
Using the equivalence $\Phi=\Hom_A(H^0(S),?) \colon \calW_U \to \mod B$
in Proposition \ref{Prop_reduc} (2),
for any $i \in \{1,\ldots,m\}$ and $I \subset \{1,\ldots,m\}$, we set
\begin{align*}
M_i:=\Phi(W_i), \quad
M_I:=\bigoplus_{i \in I} M_i.
\end{align*}

Then we have a finite complete generalized fan 
$\Sigma(M_I):=\{\overline{E} \mid E \in \TF(M_I) \}$
in $K_0(\proj B)_\R$ by the $M_I$-TF equivalence classes
in Proposition \ref{Prop_M-TF_fan} (2).
We will show that $\Sigma(M_I)$ coincides with $\Sigma_I$.

\begin{Thm}\label{Thm_Sigma(M_I)}
Let $U=\bigoplus_{i=1}^m U_i \in \twopsilt A$,
and $I \subset \{1,\ldots,m\}$ be a subset.
Then we have
\begin{align*}
\Sigma_I=\Sigma(M_I) 
\quad \text{and} \quad
\{ \sigma^\circ \mid \sigma \in \Sigma_I \} = \TF(M_I).
\end{align*}
Moreover there exist mutually inverse bijections
\begin{align*}
\Face^\circ_I D(U) &\to \Sigma(M_I), &
F &\mapsto \pi(F);\\
\Sigma(M_I) &\to \Face^\circ_I D(U), &
\sigma &\mapsto \pi^{-1}(\sigma) \cap \partial_I=C(U/U_I)+\rho(\sigma),
\end{align*}
which preserve inclusions and intersections.
\end{Thm}

To prove Theorem \ref{Thm_Sigma(M_I)}, 
it suffices to show in the case $\#I=1$ 
by Remark \ref{Rem_M-TF_oplus} and Proposition \ref{Prop_Sigma_I_cap}.
Moreover it is enough 
to compare the sets $\max \Sigma_I$ and $\max \Sigma(M_I)$
of maximal elements, 
since we have already proved that $\Sigma_I$ and $\Sigma(M_I)$
are finite complete generalized fans in $K_0(\proj B)_\R$
in Propositions \ref{Prop_Sigma_I_fan} and \ref{Prop_M-TF_fan} (2).

Since we consider the relative interiors of faces in this subsection, 
we first remark the following property.

\begin{Lem}\label{Lem_pi_interior}
Let $U \in \twopsilt A$.
For any face $F \in \Face D(U)$, 
we have $\pi(F)^\circ=\pi(F^\circ)$.
\end{Lem}

\begin{proof}
This is clear, because $\pi|_{\R F} \colon \R F \to \R \pi(F)$ is an open map.
\end{proof}

Let $(i,\epsilon) \in \{1,\ldots,m\} \times \{\pm\}$.
For each $F \in \Facet_i^\epsilon D(U)$,
we constructed a brick $L_F \in \mod A$ by
\begin{align*}
L_F:=\begin{cases}
\rmw_F Y_i^+ & (\epsilon={+}) \\
\rmw_F X_i^- & (\epsilon={-})
\end{cases}
\end{align*}
in Definition \ref{Def_L_F}.
Here $\rmw_F Y_i^+$ and $\rmw_F X_i^-$ are the modules 
in the short exact sequences
\begin{align*}
0 \to \rmt_F Y_i^+ \to Y_i^+ \to \rmw_F Y_i^+ \to 0, \quad
0 \to \rmw_F X_i^- \to X_i^- \to \rmf_F X_i^- \to 0.
\end{align*}
These modules, $L_F$ and $W_i$ satisfy the following properties.

\begin{Lem}\label{Lem_W_i-TF}
Let $U=\bigoplus_{i=1}^m U_i \in \twopsilt A$ and
$F \in \Facet_i^\epsilon D(U)$ 
with $(i,\epsilon) \in \{1,\ldots,m\} \times \{\pm\}$.
\begin{enumerate}
\item
There exists a short exact sequence
\begin{align*}
\begin{cases}
0 \to \rmt_F Y_i^+ \to W_i \to N_F \to 0 & (\epsilon={+}) \\
0 \to N_F \to W_i \to \rmf_F X_i^- \to 0 & (\epsilon={-})
\end{cases}
\end{align*}
whose terms are in $\calW_U$.
\item
The module $N_F$ in (1) is embedded to a non-split short exact sequence 
\begin{align*}
\begin{cases}
0 \to N_F \to L_F \to X_i^+ \to 0 
& (\epsilon=+, \ X_i \in \mod A)\\
0 \to X_i^- \to N_F \to L_F \to 0
& (\epsilon=+, \ X_i \in (\mod A)[1])\\
0 \to L_F \to N_F \to Y_i^+ \to 0 
& (\epsilon=-, \ Y_i \in \mod A)\\
0 \to Y_i^- \to L_F \to N_F \to 0
& (\epsilon=-, \ Y_i \in (\mod A)[1])
\end{cases}.
\end{align*}
\item
The short exact sequence in (1) coincides with
$0 \to \rmt_F W_i \to W_i \to \rmf_F W_i \to 0$.
\item
Let $\theta \in C^\circ(U_i)+F^\circ$.
Then the short exact sequence in (3) coincides with
$0 \to \rmt_\theta W_i \to W_i \to \rmf_\theta W_i \to 0$.
\item
In $\mod B$, there exists a short exact sequence
$0 \to \Phi(\rmt_F W_i) \to \Phi(W_i) \to \Phi(\rmf_F W_i) \to 0$.
For any $\xi \in \pi(F^\circ)$, this short exact sequence coincides with 
$0 \to \rmt_\xi M_i \to M_i \to \rmf_\xi M_i \to 0$.
\end{enumerate}
\end{Lem}

\begin{proof}
We only prove the case $\epsilon={+}$.
We may assume $Y_i^+ \ne 0$.

(1)
Since $\rmw_F Y_i^+=L_F \ne 0$, by Lemma \ref{Lem_proper_sub_Y} (2), 
we have $\rmt_F Y_i^+ \in \calW_U$.
We use the triangle 
$W_i \to Y_i^+ \to X_i \to W_i[1]$ in Notation \ref{Nota_X_Y}.
Since $\rmt_F Y_i^+ \in \calW_U=\Filt\{X_{m+1},\ldots,X_n\}$ 
by Proposition \ref{Prop_sim_W_U}, 
the composite $\rmt_F Y_i^+ \to X_i$
of the inclusion $\rmt_F Y_i^+ \to Y_i^+$ and
the morphism $Y_i^+ \to X_i$ in the triangle is zero.
Thus the inclusion $\rmt_F Y_i^+ \to Y_i^+$ factors through $W_i$,
so there exists a monomorohism $\rmt_F Y_i^+ \to W_i$ in $\calW_U$.
Thus we have a short exact sequence
$0 \to \rmt_F Y_i^+ \to W_i \to N_F \to 0$.
Since $\rmt_F Y_i^+, W_i$ are in $\calW_U$, so is $N_F$.

(2)
By the proof of (1), $N_F$ and $L_F$ are the cokernels 
of the monomorphisms $\rmt_F Y_i^+ \to W_i$ and $\rmt_F Y_i^+ \to Y_i^+$, 
respectively.
Now we consider the triangle $W_i \to Y_i^+ \to X_i \to W_i[1]$ again.

(i)
Assume $X_i \in \mod A$.
Then the triangle gives a short exact sequence 
$0 \to W_i \to Y_i^+ \to X_i^+ \to 0$ in $\mod A$,
and we have the commutative diagram of short exact sequences
\begin{align*}
\begin{tikzpicture}[baseline=0pt,->,scale=1.2]
\node (11) at ( 0, 0.5) {$0$};
\node (12) at ( 2, 0.5) {$\rmt_F Y_i^+\vphantom{Y_i^+}$};
\node (13) at ( 4, 0.5) {$\rmt_F Y_i^+\vphantom{Y_i^+}$};
\node (14) at ( 6, 0.5) {$0$};
\node (15) at ( 8, 0.5) {$0$};
\node (21) at ( 0,-0.5) {$0$};
\node (22) at ( 2,-0.5) {$W_i\vphantom{Y_i^+}$};
\node (23) at ( 4,-0.5) {$Y_i^+\vphantom{Y_i^+}$};
\node (24) at ( 6,-0.5) {$X_i^+\vphantom{Y_i^+}$};
\node (25) at ( 8,-0.5) {$0$};
\draw (11) to (12);
\draw (12) to (13);
\draw (13) to (14);
\draw (14) to (15);
\draw (21) to (22);
\draw (22) to (23);
\draw (23) to (24);
\draw (24) to (25);
\draw (12) to (22);          
\draw (13) to (23);          
\draw (14) to (24);          
\end{tikzpicture}.
\end{align*}
Taking cokernels, 
we have another short exact sequence $0 \to N_F \to L_F \to X_i^+ \to 0$.

(ii)
Otherwise, $X_i \in (\mod A)[1]$ holds.
Then the triangle gives a short exact sequence 
$0 \to X_i^- \to W_i \to Y_i^+ \to 0$ in $\mod A$,
and we have the commutative diagram of short exact sequences
\begin{align*}
\begin{tikzpicture}[baseline=0pt,->,scale=1.2]
\node (11) at ( 0, 0.5) {$0$};
\node (12) at ( 2, 0.5) {$0$};
\node (13) at ( 4, 0.5) {$\rmt_F Y_i^+\vphantom{Y_i^+}$};
\node (14) at ( 6, 0.5) {$\rmt_F Y_i^+\vphantom{Y_i^+}$};
\node (15) at ( 8, 0.5) {$0$};
\node (21) at ( 0,-0.5) {$0$};
\node (22) at ( 2,-0.5) {$X_i^-\vphantom{Y_i^+}$};
\node (23) at ( 4,-0.5) {$W_i\vphantom{Y_i^+}$};
\node (24) at ( 6,-0.5) {$Y_i^+\vphantom{Y_i^+}$};
\node (25) at ( 8,-0.5) {$0$};
\draw (11) to (12);
\draw (12) to (13);
\draw (13) to (14);
\draw (14) to (15);
\draw (21) to (22);
\draw (22) to (23);
\draw (23) to (24);
\draw (24) to (25);
\draw (12) to (22);          
\draw (13) to (23);          
\draw (14) to (24);          
\end{tikzpicture}.
\end{align*}
Taking cokernels, 
we have another short exact sequence $0 \to X_i^- \to N_F \to L_F \to 0$.

In either case, we have $N_F \in \calW_U$ by (1), 
and $L_F \in \calT_U$ because $L_F$ is a factor module of $Y_i^+ \in \calT_U$.
Thus the short exact sequence does not split.

(3)
Take $\eta \in F^\circ$.
By the uniqueness of canonical sequences of $W_i$ with respect
to the torsion pairs $(\ovcalT_\eta,\calF_\eta)$
and $(\calT_\eta,\ovcalF_\eta)$,
it suffices to prove $\rmt_F Y_i^+ \in \calT_\eta$ and $N_F \in \calF_\eta$.
The former is clear.
The proof of the latter again
depends on whether $X_i \in \mod A$ or $X_i \in (\mod A)[1]$.

(i)
Assume $X_i \in \mod A$.
Since $L_F \in \simple \calW_\eta$ by Proposition \ref{Prop_L_F_simple} (2),
we have $N_F \in \calF_\eta$ by (2).

(ii)
Assume $X_i \in (\mod A)[1]$.
Let $N'$ be a submodule of $N_F$.
We have the commutative diagram of short exact sequences
\begin{align*}
\begin{tikzpicture}[baseline=0pt,->,scale=1.2]
\node (11) at ( 0, 0.5) {$0$};
\node (12) at ( 2, 0.5) {$N' \cap X_i^-\vphantom{Y_i^+}$};
\node (13) at ( 4, 0.5) {$N'\vphantom{Y_i^+}$};
\node (14) at ( 6, 0.5) {$N'/(N' \cap X_i^-)\vphantom{Y_i^+}$};
\node (15) at ( 8, 0.5) {$0$};
\node (21) at ( 0,-0.5) {$0$};
\node (22) at ( 2,-0.5) {$X_i^-\vphantom{Y_i^+}$};
\node (23) at ( 4,-0.5) {$N_F\vphantom{Y_i^+}$};
\node (24) at ( 6,-0.5) {$L_F\vphantom{Y_i^+}$};
\node (25) at ( 8,-0.5) {$0$};
\draw (11) to (12);
\draw (12) to (13);
\draw (13) to (14);
\draw (14) to (15);
\draw (21) to (22);
\draw (22) to (23);
\draw (23) to (24);
\draw (24) to (25);
\draw (12) to (22);          
\draw (13) to (23);          
\draw (14) to (24);          
\end{tikzpicture}
\end{align*}
with all vertical maps injective.
By Proposition \ref{Prop_facet_X_Y_almost} (2), 
we have $X_i^- \in \calF_\eta$, so $\eta(N' \cap X_i^-) \le 0$.
By Proposition \ref{Prop_L_F_simple} (2), $L_F \in \simple \calW_\eta$ holds, 
which gives $\eta(N'/(N' \cap X_i^-)) \le 0$.
Thus $\eta(N') \le 0$.

If $\eta(N')=0$, then 
$\eta(N' \cap X_i^-)=0$ and $\eta(N'/(N' \cap X_i^-))=0$.
In this case, $(N' \cap X_i^-,N'/(N' \cap X_i^-))$ is 
isomorphic to $(0,0)$ or $(0,L_F)$ 
by $X_i^- \in \calF_\eta$ and $L_F \in \simple \calW_\eta$,
so $N'$ is isomorphic to $0$ or $L_F$.
If the latter case holds,
the commutative diagram implies that 
$0 \to X_i^- \to N_F \to L_F \to 0$ splits,
which contradicts (2). 
Thus $(N' \cap X_i^-,N'/(N' \cap X_i^-))$ is $(0,0)$, and $N'=0$.
Therefore we have $N_F \in \calF_\eta$.

(4)
Fix $\eta \in F^\circ$, and let $\theta=r[U_i]+\eta$ with $r \in \R_{>0}$.
If $r$ is sufficiently small,
then $\rmt_F W_i \in \calT_\theta$ and $\rmf_F W_i \in \calF_\theta$.
By Proposition \ref{Prop_reduc_K_0} (2), 
all elements in $C^\circ(U_i)+\eta$ are TF equivalent,
so for all $r \in \R_{>0}$,
we have $\rmt_F W_i \in \calT_\theta$ and $\rmf_F W_i \in \calF_\theta$.
Thus by the uniqueness of canonical sequences,
we have that
$0 \to \rmt_F W_i \to W_i \to \rmf_F W_i \to 0$ coincides with
$0 \to \rmt_\theta W_i \to W_i \to \rmf_\theta W_i \to 0$.

(5)
By (1)(3), the short exact sequence 
$0 \to \rmt_F W_i \to W_i \to \rmf_F W_i \to 0$ is in $\calW_U$.
This is sent to a short exact sequence
$0 \to \Phi(\rmt_F W_i) \to \Phi(W_i) \to \Phi(\rmf_F W_i) \to 0$ in $\mod B$
by Proposition \ref{Prop_reduc} (2).

Let $\xi \in \pi(F^\circ)$.
By the uniqueness of canonical sequences,
it suffices to show
$\Phi(\rmt_F W_i) \in \calT_\xi$ and $\Phi(\rmf_F W_i) \in \calF_\xi$.
Take $\eta \in F^\circ$ such that $\xi=\pi(\eta)$,
and take $\theta \in C^\circ(U_i)+\eta$.
Then $\xi=\pi(\theta)$ holds.
By (4), $\rmt_F W_i \in \calT_\theta$ and $\rmf_F W_i \in \calF_\theta$.
This and (3)(1) give
$\rmt_F W_i \in \calT_\theta \cap \calW_U$ and 
$\rmf_F W_i \in \calF_\theta \cap \calW_U$.
Then by Proposition \ref{Prop_reduc_K_0} (3) and $\pi(\theta)=\xi$,
we have $\Phi(\rmt_F W_i) \in \calT_\xi$ and $\Phi(\rmf_F W_i) \in \calF_\xi$.
\end{proof}

By using the short exact sequences in Lemma \ref{Lem_W_i-TF}, 
we prove Theorem \ref{Thm_Sigma(M_I)} in the case $\#I=1$.
Note that $\TF_{n-m}(M_i)$ and $\Sigma_{n-m}(M_i)$ 
consist of the full-dimensional elements in $\TF(M_i)$ and $\Sigma(M_i)$,
respectively, because $|B|=n-m$.
In particular, $\Sigma_{n-m}(M_i)=\max \Sigma(M_i)$ holds.

\begin{Prop}\label{Prop_M_i-TF}
Let $U=\bigoplus_{i=1}^m U_i \in \twopsilt A$.
For each $i \in \{1,\ldots,m\}$ and $F \in \Facet_i D(U)$,
the following statements hold.
\begin{enumerate}
\item
There uniquely exists $E_F \in \TF(M_i)_{n-m}$ such that 
$\pi(F^\circ)=\pi(F)^\circ \subset E_F$.
\item
In $K_0(\mod A)$, we have $[\rmf_F W_i]=\epsilon_F[L_F]-[X_i]$.
\item
Let $F \ne F' \in \Facet_i D(U)$.
Then $E_F \cap E_{F'}=\emptyset$.
\item
We have $\Sigma_{\{i\}}=\Sigma(M_i)$, and
$\pi(F)^\circ=E_F$ and $\pi(F)=\overline{E_F}$ hold. 
\item
The map $\pi$ induces bijections
\begin{align*}
\Face^\circ_{\{i\}} D(U) & \to \Sigma(M_i), 
& \Facet_i D(U) & \to \Sigma_{n-m}(M_i).
\end{align*}
Their inverses are given by $\sigma \mapsto \pi^{-1}(\sigma) \cap \partial_i$.
\end{enumerate}
\end{Prop}

\begin{proof}
(1)
First, $\pi(F)^\circ=\pi(F^\circ)$ is Lemma \ref{Lem_pi_interior}.
For each $\xi \in \pi(F^\circ)$,
we have $\rmt_\xi M_i=\Phi(\rmt_F W_i)$, $\rmw_\xi M_i=0$ 
and $\rmf_\xi M_i=\Phi(\rmf_F W_i)$
by Lemma \ref{Lem_W_i-TF} (5).
Thus $\pi(F^\circ)$ is contained in some single $M_i$-TF equivalence class,
say $E_F$.
Since $\rmw_{\xi} M_i=0$, Lemma \ref{Lem_M-TF_open} (2) implies 
$E_F \in \TF_{n-m}(M_i)$.
Thus $\pi(F^\circ) \subset E_F \in \TF_{n-m}(M_i)$ holds.

(2)
If $\epsilon_F=+$, then 
\begin{align*}
\epsilon_F[L_F]-[X_i]=[L_F]-[X_i]
\stackrel{\text{Lem.~\ref{Lem_W_i-TF} (2)}}{=}[N_F]
\stackrel{\text{Lem.~\ref{Lem_W_i-TF} (3)}}{=}[\rmf_F W_i].
\end{align*}

If $\epsilon_F=-$, then the triangle ($*$) $W_i \to Y_i \to X_i \to W_i[1]$
and Lemma \ref{Lem_W_i-TF} (1)(2)(3) give
\begin{align*}
\epsilon_F[L_F]-[X_i]\stackrel{(*)}{=}-[L_F]-[Y_i]+[W_i]
&\stackrel{\text{Lem.~\ref{Lem_W_i-TF} (2)}}{=}-[N_F]+[W_i]\\
&\stackrel{\text{Lem.~\ref{Lem_W_i-TF} (1)}}{=}[\rmf_F X_i^-]
\stackrel{\text{Lem.~\ref{Lem_W_i-TF} (3)}}{=}[\rmf_F W_i]
\end{align*}

(3)
Theorem \ref{Thm_L_F_inn_vec} (1) gives 
$\epsilon_F[L_F] \ne \epsilon_{F'}[L_{F'}]$ in $K_0(\mod A)$.
Thus (2) gives $[\rmf_F W_i] \ne [\rmf_{F'} W_i]$,
and hence $\rmf_F W_i \not \simeq \rmf_{F'} W_i$.
By Proposition \ref{Prop_reduc} (2),
$\Phi(\rmf_F W_i) \not \simeq \Phi(\rmf_{F'} W_i)$.
By (1), we get $\rmf_{E_F} M_i \not \simeq \rmf_{E_{F'}} M_i$;
hence $E_F \ne E_{F'}$. 
Since $E_F,E_{F'} \in \TF(M_i)$, we have $E_F \cap E_{F'}=\emptyset$.

(4)
By Corollary \ref{Cor_max_Face_dim},
there exists a bijection 
$\Psi_1 \colon \Facet_i D(U) \to \max \Sigma_{\{i\}}$
given by $F \mapsto \pi(F)$.
On the other hand, (3) gives an injection
$\Facet_i D(U) \to \TF_{n-m}(M_i)$ such that $F \mapsto E_F$.
Composing the bijection 
$\TF_{n-m}(M_i) \to \Sigma_{n-m}(M_i)=\max \Sigma(M_i)$ 
given by $E \mapsto \overline{E}$
in Lemma \ref{Lem_M-TF_open} (1),
we have an injection
$\Psi_2 \colon \Facet_i D(U) \to \max \Sigma(M_i)$
such that $F \mapsto \overline{E_F}$.

Then we have an injection 
$\Psi_2\Psi_1^{-1} \colon \max \Sigma_{\{i\}} \to \max \Sigma(M_i)$.
For any $\sigma=\pi(F) \in \max \Sigma_{\{i\}}$
with $F \in \Facet_i D(U)$, 
we have 
\begin{align*}
\sigma=\pi(F)=\overline{\pi(F)^\circ} \stackrel{\text{(1)}}{\subset}
\overline{E_F}=\Psi_2\Psi_1^{-1}(\sigma).
\end{align*}
Both generalized fans $\Sigma_{\{i\}}$ and $\Sigma(M_i)$ 
are finite and complete in $K_0(\proj B)_\R$ 
by Propositions \ref{Prop_Sigma_I_fan} and \ref{Prop_M-TF_fan} (2),
so Lemma \ref{Lem_fan_inj} gives $\Sigma_{\{i\}}=\Sigma(M_i)$ and
$\Psi_2=\Psi_1$.

Since $\Psi_2=\Psi_1$,
for each $F \in \Facet_i D(U)$, we have $\pi(F)=\overline{E_F}$.
Then $\pi(F)^\circ=(\overline{E_F})^\circ=E_F$ hold,
where we use Lemma \ref{Lem_M-TF_open} (1) for the latter equality. 

(5)
The first bijection follows from (4) and Theorem \ref{Thm_Sigma_I}.
This and Corollary \ref{Cor_max_Face_dim} give the second one.
\end{proof}

Now we can prove Theorem \ref{Thm_Sigma(M_I)}.

\begin{proof}[Proof of Theorem \ref{Thm_Sigma(M_I)}]
By Proposition \ref{Prop_M_i-TF} (5), the assertions hold if $\#I=\{1\}$.
Then Remark \ref{Rem_M-TF_oplus} and Proposition \ref{Prop_Sigma_I_cap}
implies the assertions for any nonempty subset $I \subset \{1,\ldots,m\}$.
If $I=\emptyset$, then $\Face^\circ_I D(U)=\{D(U)\}$ and $M_I=0$ hold, 
so both generalized fans $\Sigma_I$ and $\Sigma(M_I)$ 
are $\{K_0(\proj B)_\R\}$,
and we have the assertions also in this case. 
\end{proof}

By combining Corollary \ref{Cor_Sigma_I_all} and 
Theorem \ref{Thm_Sigma(M_I)},
we have the following description of faces of $D(U)$.
Recall that $\pi|_{L(U)} \colon L(U) \to K_0(\proj B)_\R$ is a bijection
from Theorem \ref{Thm_pi_L(U)} (2),
and that $\rho$ is its inverse map as 
$\rho \colon K_0(\proj B)_\R \to L(U)$.

\begin{Cor}\label{Cor_Sigma(M_I)_all}
Let $U=\bigoplus_{i=1}^m U_i \in \twopsilt A$.
\begin{enumerate}
\item
There exists a bijection
\begin{align*}
\{ (I,\sigma) \mid I \subset \{1,\ldots,m\}, \ \sigma \in \Sigma(M_I) \}
&\to \Face D(U), \\
(I,\sigma) &\mapsto C(U/U_I)+\rho(\sigma),
\end{align*}
whose inverse map sends $F \in \Face D(U)$ to 
$(I_F,\pi(F))$.
\item
Let $F \in \Face^\circ_I D(U)$ correspond to $(I,\sigma)$ in (1).
Then we have 
\begin{align*}
\dim_\R F=(m-\#I)+\dim_\R \sigma, \quad
\lambda_U(F)=C(U/U_I), \quad \lambda'_U(F)=\rho(\sigma).
\end{align*}
\item
Let $F_1,F_2 \in \Face D(U)$ correspond to 
$(I_1,\sigma_1),(I_2,\sigma_2)$ in (1).
Then $F_1 \subset F_2$ holds if and only if
$I_1 \supset I_2$ and $\sigma_1 \subset \sigma_2$. 
\end{enumerate} 
\end{Cor}

The following is useful when we study the faces of 
$D(U)$ explicitly.

\begin{Rem}\label{Rem_low_dim}
Assume that $D(U)$ is strongly convex
(cf.~Proposition \ref{Prop_strong_convex}).
Then $D(U)$ is generated by its 1-dimensional faces.
By Corollary \ref{Cor_Sigma(M_I)_all} (2), 
all 1-dimensional faces are 
\begin{enumerate}
\item
$\rho(\sigma)$ with $\sigma \in \Sigma(M)$ being 1-dimensional; and
\item
$C(U_i)$ with $i \in \{1,\ldots,m\}$ satisfying $\{0\} 
\in \Sigma(M/M_i)$.
\end{enumerate}
Therefore $C(U_i)$ is a 1-dimensional face of $D(U)$
if and only if $M/M_i$ is sincere in $\mod B$.

Moreover to obtain abstract structures of faces of $D(U)$,
it is enough to determine the 1-dimensional subfaces 
of each 2-dimensional face, 
and we can use Corollary \ref{Cor_Sigma(M_I)_all} also for this purpose.
Namely, all 2-dimensional faces are 
\begin{enumerate}
\item
$\rho(\sigma)$ with $\sigma \in \Sigma(M)$ being 2-dimensional; and
\item
$C(U_i)+\rho(\sigma)$ with $i \in \{1,\ldots,m\}$ and 
$\sigma \in \Sigma(M/M_i)$ being 1-dimensional; and
\item
$C(U_I)$ with $I \subset \{1,\ldots,m\}$ satisfying 
$\#I=2$ and $\{0\} \in \Sigma(M/M_I)$.
\end{enumerate}
\end{Rem}

To calculate $\rho$ explicitly, the following property 
on the maximal completion $S=\bigoplus_{i=1}^n S_i$ is useful.
Note that the value $\rho_i(\xi)$ or $\xi(\rmf_\xi M_i)$ below 
is not determined only by $\Sigma(M_i)$ in general.
For the definition of $d_{U_i}$, see before Proposition \ref{Prop_dual_basis}.

\begin{Thm}\label{Thm_rho_explicit}
Let $U=\bigoplus_{i=1}^m U_i \in \twopsilt A$.
Define maps $\rho_1,\ldots,\rho_m \colon K_0(\proj B)_\R \to \R$
so that any $\xi=\sum_{k=1}^{n-m} x_k[P(k)_B]$ with $x_k \in \R$
satisfies
\begin{align*}
\rho(\xi)=\sum_{i=1}^m \rho_i(\xi)[U_i]+\sum_{j=m+1}^n x_{j-m}[S_j].
\end{align*}
\begin{enumerate}
\item
Let $i \in \{1,\ldots,m\}$ and $\sigma=\overline{E} \in \max \Sigma(M_i)$
with $E \in \TF_{n-m}(M_i)$.
Then any $\xi \in \sigma$ satisfies
\begin{align*}
\rho_i(\xi)d_{U_i}=|\xi(\rmf_E M_i)|=|\xi(\rmf_\xi M_i)|
=|{\min\{\xi(N) \mid 
\text{$N$ is a factor module of $M_i$}\}}|.
\end{align*}
In particular, $\rho_i$ is $\R$-linear on $\sigma$.
\item
The map $\rho$ is $\R$-linear 
on each $\sigma \in \max \Sigma(M_{\{1,\ldots,m\}})$.
Thus $\rho$ is piecewise linear.
\end{enumerate}
\end{Thm}

\begin{proof}
(1)
By Proposition \ref{Prop_M_i-TF} (5),
$\sigma \in \max \Sigma(M_i)=\max \Sigma_{\{i\}}$ holds,
and we take $F \in \Facet_i D(U)$ such that $\sigma=\pi(F)$.
Set $\theta:=\rho(\xi)$.

By Proposition \ref{Prop_dual_basis} for $X=\SH(S)$,
we have $\rho_i(\xi)d_{U_i}=\rho(\xi)(X_i)=\theta(X_i)$.
Since $\xi \in \sigma$, we have $\theta \in \pi^{-1}(\sigma) \cap L(U)
\subset \pi^{-1}(\pi(F)) \cap \partial_i=F$,
where the last equality is by Proposition \ref{Prop_M_i-TF} (5). 
Thus we get $\theta(L_F)=0$ by Theorem \ref{Thm_L_F_inn_vec} (1), 
which implies $\theta(X_i)=-\theta(\epsilon_F[L_F]-[X_i])$.
Proposition \ref{Prop_M_i-TF} (2) implies 
$-\theta(\epsilon_F[L_F]-[X_i])=-\theta(\rmf_F W_i)$.
Therefore $\rho_i(\xi)d_{U_i}=-\theta(\rmf_F W_i)$ holds.

For any $\eta \in F^\circ$, we have $\eta(\rmf_F W_i) \le 0$,
so $\theta \in F=\overline{F^\circ}$ implies $\theta(\rmf_F W_i) \le 0$.
Thus we obtain $\rho_i(\xi)d_{U_i}=|\theta(\rmf_F W_i)|$.

Since $\pi(F^\circ)=\pi(F)^\circ=\sigma^\circ=E$,
we have $\Phi(\rmf_F W_i)=\rmf_E M_i$ by Lemma \ref{Lem_W_i-TF} (5),
and then Proposition \ref{Prop_reduc_K_0} (3) and $\pi(\theta)=\xi$ imply
$\theta(\rmf_F W_i)=\xi(\rmf_E M_i)$.
Thus we get $\rho_i(\xi)d_{U_i}=|\xi(\rmf_E M_i)|$,
which gives the first equality.

By Lemma \ref{Lem_M-TF_max_min}, 
the second and the third equalities hold.

(2) 
By Remark \ref{Rem_M-TF_oplus}, for each $i \in \{1,\ldots,m\}$, 
there exists $\sigma_i \in \Sigma(M_i)$ such that $\sigma \subset \sigma_i$.
Then all $\rho_i$ are $\R$-linear on $\sigma$ by (1), and so is $\rho$.
\end{proof}

This proposition is related to the minimal completion $T$ of $U$ as follows.

\begin{Ex}
Let $U=\bigoplus_{i=1}^m U_i \in \twopsilt A$.

We consider the TF equivalence class
$C^\circ(B[1])$ in $K_0(\proj B)_\R$.
It is contained in a unique $M$-TF equivalence class $E$.
For each $i \in \{1,\ldots,m\}$, clearly $\rmf_E M_i=M_i$ holds.

Proposition \ref{Prop_dual_basis_mut} states that
$a_{j,i}d_{U_i}=b_{j,i}d_{U_j}$ holds
for each $i \in \{1,\ldots,m\}$ and $j \in \{m+1,\ldots,n\}$,
where $a_{j,i},b_{j,i} \in \Z_{\ge 0}$ as follows.
\begin{enumerate}
\item
In the triangle $S_j \to U'_j \to T_j \to S_j$ in $\sfK^\rmb(\proj A)$, 
we have $U'_j \simeq \bigoplus_{i=1}^m U_i^{\oplus a_{j,i}}$ in $\add U$.
Thus $[U'_j]=\sum_{i=1}^m a_{j,i}[U_i]$ holds in $K_0(\proj A)_\R$.
\item
In the triangle $X_i[-1] \to W_i \to Y_i \to X_i$ in $\sfD(A)$,
the object $X_j$ appears exactly $b_{j,i}$ times as composition factors
of $W_i$ in $\Filt \{X_{m+1},\ldots,X_n\}=\calW_U$  
(see Proposition \ref{Prop_sim_W_U}).
Thus $[W_i]=\sum_{j=m+1}^n b_{j,i}[X_j]$ holds in $K_0(\mod A)_\R$.
\end{enumerate}

We write $L(k)_B$ for the simple top of $P(k)_B$.
By Proposition \ref{Prop_sim_W_U}, we have
\begin{align*}
[M_i]=\sum_{j=m+1}^n b_{j,i}[L(j-m)_B] \
\text{in $K_0(\mod B)$}, \quad
d_j=\dim_K \End_B(L(j-m)_B),
\end{align*}

If $\xi=-[P(j-m)_B]$ with $j \in \{m+1,\ldots,n\}$,
then $\xi$ belongs to $\overline{E}$, 
so 
\begin{align*}
&\xi(\rmf_E M_i)=\xi(M_i)
\stackrel{\text{Prop.~\ref{Prop_dual_basis}}}{=}-b_{j,i}d_{U_j}, \quad
\rho_i(\xi)\stackrel{\text{Thm.~\ref{Thm_rho_explicit} (1)}}{=}
\frac{|\xi(\rmf_E M_i)|}{d_{U_i}}
=b_{j,i}\frac{d_{U_j}}{d_{U_i}} \quad (i \in \{1,\ldots,m\}), \\
&\rho(\xi)=\sum_{i=1}^m b_{j,i}\frac{d_{U_j}}{d_{U_i}}[S_i]-[S_j]
\stackrel{\text{Prop.~\ref{Prop_dual_basis_mut}}}{=}
\sum_{i=1}^m a_{j,i}[S_i]-[S_j]=[U'_j]-[S_j]=[T_j].
\end{align*}
\end{Ex}

The following example gives how the faces of $D(U)$ are recovered from $M_I$.

\begin{Ex}
We use the setting of Example \ref{Ex_A4_D(U)} again.
Thus $A$ is the path algebra of the quiver $1 \to 2 \to 3 \to 4$,
and $U=U_1 \oplus U_2  \in \twopsilt A$
with $U_1=(P(4) \to P(3))$ and $U_2=P(1)$.
Recall that $B \simeq K(1 \to 2)$ and that
\begin{align*}
M_1=\Phi(W_1)=\begin{smallmatrix}2\end{smallmatrix}, \quad
M_2=\Phi(W_2)=\begin{smallmatrix}1\\2\end{smallmatrix}.
\end{align*}
Therefore for each $I \subset \{1,2\}$,
the generalized fan $\Sigma(M_I)$ is the following.
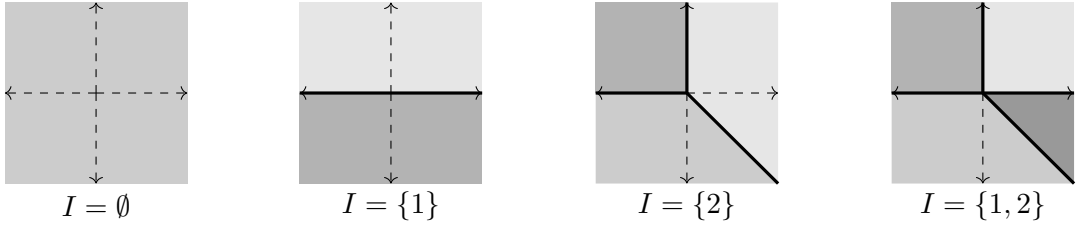
\begin{align*}
\begin{tikzpicture}[baseline=0pt,scale=0.6]
\node (00)[coordinate] at ( 0, 0) {};
\node (+0)[coordinate] at ( 2, 0) {};
\node (++)[coordinate] at ( 2, 2) {};
\node (0+)[coordinate] at ( 0, 2) {};
\node (-+)[coordinate] at (-2, 2) {};
\node (-0)[coordinate] at (-2, 0) {};
\node (--)[coordinate] at (-2,-2) {};
\node (0-)[coordinate] at ( 0,-2) {};
\node (+-)[coordinate] at ( 2,-2) {};
\node (I)              at ( 0,-2.5) {$I=\emptyset$};
\fill[black!20] (+-)--(++)--(-+)--(--)--cycle;
\draw[dashed,->] (00) to (+0);
\draw[dashed,->] (00) to (0+);
\draw[dashed,->] (00) to (-0);
\draw[dashed,->] (00) to (0-);
\end{tikzpicture}&&
\begin{tikzpicture}[baseline=0pt,scale=0.6]
\node (00)[coordinate] at ( 0, 0) {};
\node (+0)[coordinate] at ( 2, 0) {};
\node (++)[coordinate] at ( 2, 2) {};
\node (0+)[coordinate] at ( 0, 2) {};
\node (-+)[coordinate] at (-2, 2) {};
\node (-0)[coordinate] at (-2, 0) {};
\node (--)[coordinate] at (-2,-2) {};
\node (0-)[coordinate] at ( 0,-2) {};
\node (+-)[coordinate] at ( 2,-2) {};
\node (I)              at ( 0,-2.5) {$I=\{1\}$};
\fill[black!10] (+0)--(++)--(-+)--(-0)--cycle;
\fill[black!30] (+0)--(+-)--(--)--(-0)--cycle;
\draw[dashed,->] (00) to (+0);
\draw[dashed,->] (00) to (0+);
\draw[dashed,->] (00) to (-0);
\draw[dashed,->] (00) to (0-);
\draw[very thick] (-0) to (+0);
\end{tikzpicture}&&
\begin{tikzpicture}[baseline=0pt,scale=0.6]
\node (00)[coordinate] at ( 0, 0) {};
\node (+0)[coordinate] at ( 2, 0) {};
\node (++)[coordinate] at ( 2, 2) {};
\node (0+)[coordinate] at ( 0, 2) {};
\node (-+)[coordinate] at (-2, 2) {};
\node (-0)[coordinate] at (-2, 0) {};
\node (--)[coordinate] at (-2,-2) {};
\node (0-)[coordinate] at ( 0,-2) {};
\node (+-)[coordinate] at ( 2,-2) {};
\node (I)              at ( 0,-2.5) {$I=\{2\}$};
\fill[black!10] (00)--(+-)--(++)--(0+)--cycle;
\fill[black!30] (00)--(0+)--(-+)--(-0)--cycle;
\fill[black!20] (00)--(-0)--(--)--(+-)--cycle;
\draw[dashed,->] (00) to (+0);
\draw[dashed,->] (00) to (0+);
\draw[dashed,->] (00) to (-0);
\draw[dashed,->] (00) to (0-);
\draw[very thick] (00) to (0+);
\draw[very thick] (00) to (-0);
\draw[very thick] (00) to (+-);
\end{tikzpicture}&&
\begin{tikzpicture}[baseline=0pt,scale=0.6]
\node (00)[coordinate] at ( 0, 0) {};
\node (+0)[coordinate] at ( 2, 0) {};
\node (++)[coordinate] at ( 2, 2) {};
\node (0+)[coordinate] at ( 0, 2) {};
\node (-+)[coordinate] at (-2, 2) {};
\node (-0)[coordinate] at (-2, 0) {};
\node (--)[coordinate] at (-2,-2) {};
\node (0-)[coordinate] at ( 0,-2) {};
\node (+-)[coordinate] at ( 2,-2) {};
\node (I)              at ( 0,-2.5) {$I=\{1,2\}$};
\fill[black!10] (00)--(+0)--(++)--(0+)--cycle;
\fill[black!30] (00)--(0+)--(-+)--(-0)--cycle;
\fill[black!20] (00)--(-0)--(--)--(+-)--cycle;
\fill[black!40] (00)--(+-)--(+0)--cycle;
\draw[dashed,->] (00) to (+0);
\draw[dashed,->] (00) to (0+);
\draw[dashed,->] (00) to (-0);
\draw[dashed,->] (00) to (0-);
\draw[very thick] (00) to (+0);
\draw[very thick] (00) to (0+);
\draw[very thick] (00) to (-0);
\draw[very thick] (00) to (+-);
\end{tikzpicture}
\end{align*}

We recover $D(U)$ 
from $\Sigma(M_I)$ and the maximal completion $S$ of $U$
by using Remark \ref{Rem_low_dim} and Theorem \ref{Thm_rho_explicit}.

By Proposition \ref{Prop_strong_convex}, $D(U)$ is strongly convex,
because $H^0(U) \oplus H^{-1}(\nu U)$ is sincere.
Thus we can use Remark \ref{Rem_low_dim}.
Since the 1-dimensional elements of $\Sigma(M)$ are
\begin{align*}
\sigma_1:=\R_{\ge 0}[P(1)_B], \quad \sigma_2:=\R_{\ge 0}[P(2)_B], \quad
\sigma_3:=\R_{\ge 0}(-[P(1)_B]), \quad \sigma_4:=\R_{\ge 0}([P(1)_B]-[P(2)_B])
\end{align*}
and since
\begin{align*} 
\{0\} \in \Sigma(M_2)=\Sigma(M/M_1), \quad 
\{0\} \notin \Sigma(M_1)=\Sigma(M/M_2),
\end{align*}
the 1-dimensional faces of $D(U)$ are
\begin{align*}
F_1:=\rho(\sigma_1), \quad F_2:=\rho(\sigma_2), \quad 
F_3:=\rho(\sigma_3), \quad F_4:=\rho(\sigma_4), \quad
F_5:=C(U_1).
\end{align*}
We set $F_{134}:=F_1+F_3+F_4$ and so on also here.
Then similarly the 2-dimensional faces of $D(U)$ are
\begin{align*}
&
\rho(\sigma_1+\sigma_2)=F_{12}, \quad 
\rho(\sigma_1+\sigma_4)=F_{14}, \quad
\rho(\sigma_2+\sigma_3)=F_{23}, \quad
\rho(\sigma_3+\sigma_4)=F_{34}, \\
&
C(U_1)+\rho(\sigma_2)=F_{25}, \quad
C(U_1)+\rho(\sigma_3)=F_{35}, \quad
C(U_1)+\rho(\sigma_4)=F_{45}, \\
&
C(U_2)+\rho(\R[P(1)_B])=F_{13}.
\end{align*}

Therefore the intersection of $D(U)$
and an approriate affine hyperplane is the following pyramid,
where the bold dashed line indicates the intersection of
$C(U)$ and the affine hyperplane.
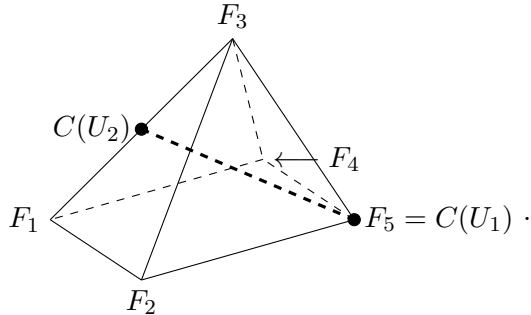
\begin{align*}
\begin{tikzpicture}[baseline=0pt,scale=0.8]
\node (F3)[coordinate,label= 90:{$F_3$}]        at ( 0  , 3  ) {};
\node (F1)[coordinate,label=180:{$F_1$}]        at (-3  , 0  ) {};
\node (F2)[coordinate,label=270:{$F_2$}]        at (-1.5,-1  ) {};
\node (F5)[coordinate,label=  0:{$F_5=C(U_1)$}] at ( 2  , 0  ) {};
\node (F4)[coordinate,label=  0:{$\longleftarrow F_4$}] at ( 0.5, 1  ) {};
\node (U2)[coordinate,label=180:{$C(U_2)$}]     at (-1.5, 1.5) {};
\draw (F3) to (F1);
\draw (F3) to (F2);
\draw[dashed] (F3) to (F4);
\draw (F3) to (F5);
\draw (F1) to (F2);
\draw (F2) to (F5);
\draw[dashed] (F5) to (F4);
\draw[dashed] (F4) to (F1);
\draw[dashed,very thick] (U2) to (F5);
\draw[fill=black] (F5) circle [radius=0.1];
\draw[fill=black] (U2) circle [radius=0.1];
\end{tikzpicture}.
\end{align*}

In Example \ref{Ex_A4_D(U)},
we have seen $[S_1]=[U_1]=[P(3)]-[P(4)]$, $[S_2]=[U_2]=[P(1)]$,
$[S_3]=[P(2)]$, $[S_4]=[P(3)]$.
Theorem \ref{Thm_rho_explicit} gives the value of $\rho(\xi)$
for $\xi$ in $F_1,F_2,F_3,F_4 \in \Face^\circ_{\{1,2\}} D(U)$ 
as follows.
\begin{center}
\begin{tabular}{c|cc|cc|c}
$\xi$ & $\rmf_\xi(M_1)$ & $\rmf_\xi(M_2)$ & $\rho_1(\xi)$ & $\rho_2(\xi)$ &
$\rho(\xi)$ \\
\hline
$[P(1)_B]$ & 0 & 0 & 0 & 0 & $[S_3]=[P(2)]$ \\
$[P(2)_B]$ & 0 & 0 & 0 & 0 & $[S_4]=[P(3)]$ \\
$-[P(1)_B]$ & 0 & $\smallone$ & 0 & 1 & $[S_2]-[S_3]=[P(1)]-[P(2)]$ \\
$[P(1)_B]-[P(2)_B]$ & $\smalltwo$ & 0 & 1 & 0 & 
$[S_1]+[S_3]-[S_4]=[P(2)]-[P(4)]$
\end{tabular}
\end{center}
These four elements $\rho(\xi)$ and $[U_1]=[P(3)]-[P(4)]$ are
in the affine hyperplane
\begin{align*}
\{ x_1[P(1)]+x_2[P(2)]+x_3[P(3)]+x_4[P(4)] \mid 2x_1+x_2+x_3=1 \}
\subset K_0(\proj A)_\R.
\end{align*}
This affine hyperplane and $D(U)$ intersects as follows.
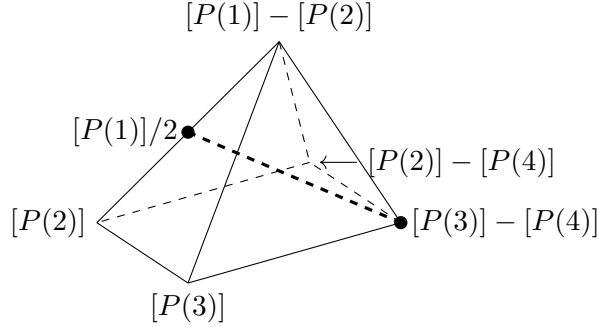
\begin{align*}
\begin{tikzpicture}[baseline=0pt,scale=0.8]
\node (F3)[coordinate,label= 90:{$[P(1)]-[P(2)]$}] at ( 0  , 3  ) {};
\node (F1)[coordinate,label=180:{$[P(2)]$}]        at (-3  , 0  ) {};
\node (F2)[coordinate,label=270:{$[P(3)]$}]        at (-1.5,-1  ) {};
\node (F5)[coordinate,label=  0:{$[P(3)]-[P(4)]$}] at ( 2  , 0  ) {};
\node (F4)[coordinate]                             at ( 0.5, 1  ) {};
\node (U2)[coordinate,label=180:{$[P(1)]/2$}]      at (-1.5, 1.5) {};
\node (F4')  at ( 0.5, 1  ) {};
\node (F4'') at ( 3  , 1  ) {$[P(2)]-[P(4)]$};
\draw (F3) to (F1);
\draw (F3) to (F2);
\draw[dashed] (F3) to (F4);
\draw (F3) to (F5);
\draw (F1) to (F2);
\draw (F2) to (F5);
\draw[dashed] (F5) to (F4);
\draw[dashed] (F4) to (F1);
\draw[dashed,very thick] (U2) to (F5);
\draw[->] (F4'') to (F4');
\draw[fill=black] (F5) circle [radius=0.1];
\draw[fill=black] (U2) circle [radius=0.1];
\end{tikzpicture}
\end{align*}
\end{Ex}

\section{The main results}
\label{Sec_main_result}

In this section, we describe the TF equivalence classes in $D(U)$
by the following result.
The symbol $\TF_\Lambda^\Gamma$ denotes
the set of TF equivalence classes in $\Gamma$,
for an algebra $\Lambda$ and a subset $\Gamma \subset K_0(\proj \Lambda)_\R$
which is a union of TF equivalence classes.

\begin{Thm}\label{Thm_TF_2^m_B}
Let $U \in \twopsilt A$.
\begin{enumerate}
\item
The bijection $C(U) \times L(U) \to D(U)$ in Theorem \ref{Thm_C(U)_L(U)} (1)
induces a bijection
\begin{align*}
\Psi_1 \colon \TF_A^{C(U)} \times \TF_A^{L(U)} \to \TF_A^{D(U)}, \quad
(E,E') \mapsto E+E',
\end{align*}
whose inverse is given by $E \mapsto (\lambda_U(E),\lambda'_U(E))$.
\item
The bijection $\pi|_{L(U)} \colon L(U) \to K_0(\proj B)_\R$
in Theorem \ref{Thm_pi_L(U)} (2) induces bijections
\begin{align*}
\Psi_2 \colon \TF_A^{C(U)} \times \TF_A^{L(U)} 
&\to 2^{\{1,\ldots,m\}} \times \TF_B,& 
(C^\circ(U_I),E) &\mapsto (I,\pi(E)),\\
\TF_A^{L(U)} &\to \TF_B,& E &\mapsto \pi(E),
\end{align*}
whose inverses are given by $(I,E) \mapsto (C(U_I),\rho(E))$
and $E \mapsto \rho(E)$, respectively.
\item
The bijection $\Psi_2 \Psi_1^{-1}$ coincides with 
\begin{align*}
\Psi_3 \colon \TF_A^{D(U)} \to 2^{\{1,\ldots,m\}} \times \TF_B, \quad
E \mapsto (\{i \in \{1,\ldots,m\} \mid E \subset D(U_i)\},\pi(E)),
\end{align*}
and its inverse is given by $(I,E) \mapsto C(U_I)+\rho(E)$.
\end{enumerate}
\end{Thm}

Since $D^\circ(U)=\bigcap_{i=1}^m D^\circ(U_i)$
and $L(U)=D(U) \setminus (\bigcup_{i=1}^m D^\circ(U_i))$,
Theorem \ref{Thm_TF_2^m_B} (3) gives bijections
\begin{align*}
\TF_A^{D^\circ(U)} \to \{\{1,\ldots,m\}\} \times \TF_B, \quad
\TF_A^{L(U)} \to \{\emptyset\} \times \TF_B.
\end{align*}
They recover the bijections
$\TF_A^{D^\circ(U)} \to \TF_B$ in Proposition \ref{Prop_reduc_K_0} (2)
and $\TF_A^{L(U)} \to \TF_B$ in (2), respectively.

Moreover, for faces of $D(U)$, we have the following result.
Recall that we have set 
$I_F:=\{i \in \{1,\ldots,m\} \mid F \subset \partial_i\}$.

\begin{Cor}
Let $U \in \twopsilt A$ and $F \in \Face D(U)$.
Set $J_F:=\{1,\ldots,m\} \setminus I_F$.
Then the bijection $\Psi_2\Psi_1^{-1}=\Psi_3$ 
in Theorem \ref{Thm_TF_2^m_B} (3) is restricted to a bijection
\begin{align*}
\TF_A^F \to 2^{J_F} \times \TF_B^{\pi(F)}.
\end{align*}
\end{Cor}

\begin{proof}
By Proposition \ref{Prop_C(U/U_I_F)} (1)(2), 
the bijection $C(U) \times L(U) \to D(U)$ is restricted to a bijection
$C(U_{J_F}) \times (F \cap L(U)) \to F$.
Thus $\Psi_1$ gives a bijection
$\TF_A^{C(U_{J_F})} \times \TF_A^{F \cap L(U)} \to \TF_A^F$.
The image of $\TF_A^{C(U_{J_F})} \times \TF_A^{F \cap L(U)}$ under $\Psi_2$ is
$2^{J_F} \times \TF_B^{\pi(F)}$ by Proposition \ref{Prop_C(U/U_I_F)} (1).
\end{proof}

To show Theorem \ref{Thm_TF_2^m_B}, the following key proposition is enough.

\begin{Prop}\label{Prop_TF_lambda_lambda'}
Let $\theta,\eta \in D(U)$,
and take the unique faces $F,G \in \Face D(U)$ such that
$\theta \in F^\circ$ and $\eta \in G^\circ$.
Then the following conditions are equivalent.
\begin{enumerate2}
\item
The elements $\theta$ and $\eta$ are TF equivalent.
\item
The faces $F=G$ coincide, and $[U]+\theta,[U]+\eta$ are TF equivalent.
\item
The faces $F=G$ coincide, and $\pi(\theta),\pi(\eta)$ are TF equivalent.
\item
The elements $\lambda_U(\theta),\lambda_U(\eta)$ are TF equivalent,
and $\pi(\theta),\pi(\eta)$ are TF equivalent.
\item
The elements $\lambda_U(\theta),\lambda_U(\eta)$ are TF equivalent,
and $\lambda'_U(\theta),\lambda'_U(\eta)$ are TF equivalent.
\end{enumerate2}
\end{Prop}

The most nontrivial part is showing (c)$\Leftrightarrow$(d),
but this has been virtually done in
Proposition \ref{Prop_lambda_F^circ} and Theorem \ref{Thm_Sigma(M_I)}.
The equivalence (b)$\Leftrightarrow$(c) comes from
just Proposition \ref{Prop_reduc_K_0} (2),
and (d)$\Leftrightarrow$(e) follows applying (a)$\Leftrightarrow$(d)
to $\lambda'_U(\theta),\lambda'_U(\eta)$ instead of $\theta,\eta$.
Thus this section is mostly devoted to showing (a)$\Leftrightarrow$(b).

For any subcategory $\calC_1,\calC_2 \subset \mod A$, we set 
$\calC_1*\calC_2$ as the full subcategory of $M \in \mod A$ which admit
short exact sequences
$0 \to M_1 \to M \to M_2 \to 0$ with $M_1 \in \calC_1$ and $M_2 \in \calC_2$.
Then we can easily check the following property.

\begin{Lem}\label{Lem_tors_general}
Let $(\calT,\calF)$ be a torsion pair in $\mod A$. 
If $\calT \subset \calT'$, then $\calT'=\calT*(\calF \cap \calT')$.
If $\calF \subset \calF'$, then $\calF'=(\calT \cap \calF')*\calF$.
\end{Lem}

For each $\theta \in D(U)$,
clearly $C^\circ(U)+\theta \subset D^\circ(U)+D(U)=D^\circ(U)$ hold.
We first compare the semistable torsion pairs for $\theta$ and $\theta'$.

\begin{Prop}\label{Prop_W_U_cap_W_theta}
Let $U \in \twopsilt A$,
$\theta \in D(U)$ and $\theta' \in C^\circ(U)+\theta$.
\begin{enumerate}
\item
We have the equalities
\begin{align*}
\ovcalT_{\theta'}=\ovcalT_U \cap \ovcalT_\theta, \quad
\ovcalF_{\theta'}=\ovcalF_U \cap \ovcalF_\theta, \quad
\calW_{\theta'}=\calW_U \cap \calW_\theta.
\end{align*}
\item
We have the equalities
\begin{align*}
\ovcalT_\theta=\ovcalT_{\theta'} * 
(\calF_{\theta'} \cap \ovcalT_\theta), \quad
\ovcalF_\theta=(\calT_{\theta'} \cap \ovcalF_{\theta'}) *
\ovcalF_{\theta'}.
\end{align*}
\end{enumerate}
\end{Prop}

\begin{proof}
We only prove the assertions for $\ovcalT_\theta$.
The assertions for $\ovcalF_\theta$ are dual.
For (1), the equality for $\calW_\theta$ comes from the other two.

(1)
Set $\eta:=\theta'-\theta$.
By Proposition \ref{Prop_cone_TF} (1), the ``$\supset$'' part follows as
$\ovcalT_U \cap \ovcalT_\theta=\ovcalT_\eta \cap \ovcalT_\theta
\subset \ovcalT_{\eta+\theta}=\ovcalT_{\theta'}$.

We next show the ``$\subset$'' part.
First, $\theta' \in D^\circ(U)$ implies $\calF_U \subset \calF_{\theta'}$,
so we get $\ovcalT_{\theta'} \subset \ovcalT_U$.
On the other hand, $C^\circ(U)+\theta \subset D^\circ(U)$ is contained in 
the TF equivalence class $[\theta']$ 
by Proposition \ref{Prop_reduc_K_0} (2).
Since $\theta$ is in the closure of $C^\circ(U)+\theta$,
we get $\ovcalT_{\theta'} \subset \ovcalT_\theta$.
Thus $\ovcalT_{\theta'} \subset \ovcalT_U \cap \ovcalT_\theta$
follows as desired.

(2)
By (1), we have $\ovcalT_{\theta'} \subset \ovcalT_\theta$.
Then we can apply Lemma \ref{Lem_tors_general}.
\end{proof}

Let $F \in \Face D(U)$, and take $\theta \in F^\circ$.
As in Lemma \ref{Lem_face_M-TF} (2),
neither $\rmw_\theta Y^+$ nor the set $\supp_\theta(\rmw_\theta Y^+)$
of its composition factors in $\calW_\theta$
depends on the choice of $\theta \in F^\circ$.
Thus the notation 
\begin{align*}
\fac_F(\rmw_F Y^+)&:=
\{ \text{factor objects of $\rmw_F Y^+$ in $\calW_\theta$} \} \\
&=\{ \text{factor modules of $Y^+$ which belong to $\calW_\theta$} \}
\end{align*}
is well-defined.
Similar properties hold for $X^-$, so we define 
\begin{align*}
\sub_F(\rmw_F X^-)&:=
\{ \text{subobjects of $\rmw_F X^-$ in $\calW_\theta$} \} \\
&=\{ \text{submodules of $X^-$ which belong to $\calW_\theta$} \}.
\end{align*}
Then we have the following properties.

\begin{Prop}\label{Prop_Filt_fac_F}
Let $U \in \twopsilt A$, $F \in \Face D(U)$,
$\theta \in F^\circ$ and $\theta' \in C^\circ(U)+\theta$.
\begin{enumerate}
\item
We have the equalities
\begin{align*}
\calF_{\theta'} \cap \ovcalT_\theta=\calF_{\theta'} \cap \calW_\theta
=\Filt(\sub_F(\rmw_F X^-)), \quad
\calT_{\theta'} \cap \ovcalF_\theta=\calT_{\theta'} \cap \calW_\theta
=\Filt(\fac_F(\rmw_F Y^+)).
\end{align*}
\item
We have the equalities
\begin{align*}
\ovcalT_\theta=
\ovcalT_{\theta'}*\Filt(\sub_F(\rmw_F X^-)), \quad
\ovcalF_\theta=
\Filt(\fac_F(\rmw_F Y^+))*\ovcalF_{\theta'}.
\end{align*}
\end{enumerate}
\end{Prop}

\begin{proof}
We only show the statements on $\Filt(\sub_F(\rmw_F X^-))$.

(1)
It suffices to show 
\begin{align*}
\calF_{\theta'} \cap \ovcalT_\theta
\stackrel{\text{(i)}}{=}\calF_{\theta'} \cap \calW_\theta
\stackrel{\text{(ii)}}{\subset}\Filt(\sub_F(\rmw_F X^-))
\stackrel{\text{(iii)}}{\subset}\calF_{\theta'} \cap \calW_\theta.
\end{align*}

(i)
By Proposition \ref{Prop_W_U_cap_W_theta} (1), 
we get $\calF_{\theta'} \subset \ovcalF_{\theta'} \subset
\ovcalF_\theta$,
so we have 
$\calF_{\theta'} \cap \ovcalT_\theta \subset 
\calF_{\theta'} \cap \calW_\theta$.
The converse is obvious.

(ii)
Let $L \in \calF_{\theta'} \cap \calW_\theta$.
By Proposition \ref{Prop_W_U_cap_W_theta} (1), 
$\ovcalT_{\theta'}=\ovcalT_U \cap \ovcalT_\theta$ holds,
so $\calF_{\theta'}$ is the smallest torsion-free class containing both 
$\calF_U$ and $\calF_\theta$.
As in Proposition \ref{Prop_T_U_T(Y^+)}, $\calF_U=\sfF(X^-)$ holds.
Thus we can take a sequence
$0=L_0 \subsetneq L_1 \subsetneq \cdots \subsetneq L_\ell=L$ such that 
each $L_k/L_{k-1}$ is a submodule of $X^-$ or belongs to $\calF_\theta$.

Then $L_\ell/L_{\ell-1}$ must be a submodule of $X^-$,
because $L \in \calW_\theta$ implies $L_\ell/L_{\ell-1} \in \ovcalT_\theta$.
Since $X^- \in \calF_U \subset \calF_{\theta'}$ by $\theta' \in D^\circ(U)$, 
we have $L_\ell/L_{\ell-1} \in \calF_{\theta'} \cap \ovcalT_\theta$.
Then (i) implies $L_\ell/L_{\ell-1} \in \calF_{\theta'} \cap \calW_\theta$.
Thus $L_\ell/L_{\ell-1}$ is a submodule of $X^-$,
and $L_\ell/L_{\ell-1} \in \calW_\theta$.

In particular, $L_{\ell-1} \in \calF_{\theta'} \cap \calW_\theta$ holds,
so we can repeat the argument above.
Therefore each $L_k/L_{k-1}$ is a submodule of $X^-$,
and belongs to $\calW_\theta$.
Thus each $L_k/L_{k-1}$ is in $\sub_\theta(\rmw_\theta X^-)$,
and we get $L=L_\ell \in \Filt(\sub_\theta(\rmw_\theta X^-))
=\Filt(\sub_F(\rmw_F X^-))$.

(iii)
The module $\rmw_F X^-$ is a submodule of $X^-$ by \eqref{w_theta F}.
Since $X^- \in \calF_U \subset \calF_{\theta'}$, 
we have $\rmw_F X^- \in \calF_{\theta'}$.
Thus $\Filt(\sub_F(\rmw_F X^-)) \subset \calF_{\theta'}$ holds.
On the other hand, $\Filt(\sub_F(\rmw_F X^-)) \subset \calW_\theta$
by definition.

(2) follows from (1) and Proposition \ref{Prop_W_U_cap_W_theta} (2).
\end{proof}

Now we can show Proposition \ref{Prop_TF_lambda_lambda'}.

\begin{proof}[Proof of Proposition \ref{Prop_TF_lambda_lambda'}]
(a)$\Rightarrow$(b)
Since each face in $D(U)$ is a union of TF equivalence class 
by Lemma \ref{Lem_face_M-TF}, we get $F=G$.
Moreover Proposition \ref{Prop_W_U_cap_W_theta} (1) implies that 
$[U]+\theta$ and $[U]+\eta$ are TF equivalent.

(b)$\Rightarrow$(a) follows from Proposition \ref{Prop_Filt_fac_F} (2).

(b)$\Leftrightarrow$(c)
We have $[U]+\theta,[U]+\eta \in D^\circ(U)$. 
Moreover $\pi([U]+\theta)=\pi(\theta)$ and $\pi([U]+\eta)=\pi(\eta)$
follow from Definition-Proposition \ref{Def-Prop_pi}.
Now Proposition \ref{Prop_reduc_K_0} (2) applies.

(c)$\Rightarrow$(d)
By Proposition \ref{Prop_lambda_F^circ}, we have
$\lambda_U(\theta),\lambda_U(\eta) \in C^\circ(U/U_{I_F})$.
They are TF equivalent by Proposition \ref{Prop_cone_TF} (1).

(d)$\Rightarrow$(c)
In view of Corollary \ref{Cor_Sigma(M_I)_all},
it suffices to show $(I_F,\pi(F))=(I_G,\pi(G))$.
We have $I:=I_F=I_G$ by Proposition \ref{Prop_lambda_F^circ}.
By Theorem \ref{Thm_Sigma(M_I)} and Lemma \ref{Lem_pi_interior},
$\pi(F)^\circ=\pi(F^\circ)$ and 
$\pi(G)^\circ=\pi(G^\circ)$ are $M_I$-TF equivalence classes.
They must coincide, 
since $\pi(\theta) \in \pi(F^\circ)$ and $\pi(\eta) \in \pi(G^\circ)$
are TF equivalent.
Then we get $\pi(F)^\circ=\pi(G)^\circ$,
and hence $\pi(F)=\pi(G)$, taking closures.

(d)$\Leftrightarrow$(e)
We apply (a)$\Leftrightarrow$(d) to 
$\lambda'_U(\theta),\lambda'_U(\eta)$ instead of $\theta,\eta$.
Then we get that $\lambda'_U(\theta),\lambda'_U(\eta)$ are TF equivalent
if and only if 
(i)
$\lambda_U(\lambda'_U(\theta)),\lambda_U(\lambda'_U(\eta))$ 
are TF equivalent and
(ii)
$\pi(\lambda'_U(\theta)),\pi(\lambda'_U(\eta))$ are TF equivalent.
The part (i) is always satisfied,
because $\lambda_U(\lambda'_U(\theta))=\lambda_U(\lambda'_U(\eta))=0$
by Lemma \ref{Lem_axiom_proj} (3).
The part (ii) is just that $\pi(\theta),\pi(\eta)$ are TF equivalent
by Theorem \ref{Thm_pi_L(U)} (1).
Thus $\lambda'_U(\theta),\lambda'_U(\eta)$ are TF equivalent
if and only if $\pi(\theta),\pi(\eta)$ are TF equivalent.
\end{proof}

Then Theorem \ref{Thm_TF_2^m_B} can be shown as follows.

\begin{proof}[Proof of Theorem \ref{Thm_TF_2^m_B}]
(1) is done in Proposition \ref{Prop_TF_lambda_lambda'}.

(2)
We clearly have a bijection $\TF_A^{C(U)} \to 2^{\{1,\ldots,m\}}$
by $C^\circ(U_I) \mapsto I$
by Proposition \ref{Prop_cone_TF} (1).
Thus it remains to show that $\pi$ induces a bijection 
$\TF_A^{L(U)} \to \TF_B$.
Let $\theta,\eta \in L(U)$.
Then they are TF equivalent if and only if
(i) $\lambda_U(\theta),\lambda_U(\eta)$ are TF equivalent and
(ii) $\lambda'_U(\theta),\lambda'_U(\eta)$ are TF equivalent.
The part (i) can be removed,
because $\lambda_U(\theta)=\lambda_U(\eta)=0$
by Proposition \ref{Prop_lambda_image} (1) and 
Lemma \ref{Lem_axiom_proj} (3).
Thus we obtain the desired bijection $\TF_A^{L(U)} \to \TF_B$.

(3)
Let $E \in \TF_A^{D(U)}$. 
Then we can take the unique $I \subset \{1,\ldots,m\}$
such that $\lambda_U(E)=C^\circ(U_I)$ by (1).
By (1)(2), $\Psi_2\Psi_1^{-1}(\theta)$ is $(I,\pi(\lambda'_U(E)))$.
We have $I=\{i \in \{1,\ldots,m\} \mid E \subset D^\circ(U_i)\}$
by Lemma \ref{Lem_lambda_DcV} (1).
Moreover $\pi(\lambda'_U(E))=\pi(E)$ holds by Theorem \ref{Thm_pi_L(U)} (1).
Thus $\Psi_2\Psi_1^{-1}$ is as in the statement.

Its inverse $\Psi_1\Psi_2^{-1}$ obviously 
sends $(I,E)$ to $C(U_I)+\rho(E)$ by (1)(2).
\end{proof}


\begin{thebibliography}{BCDMTY}
\bibitem[AIR]{AIR} 
T. Adachi, O. Iyama, I. Reiten, 
\emph{$\tau$-tilting theory}, 
Compos. Math., \textbf{150.3} (2014), 415--452.
\bibitem[AET]{AET}
T. Adachi, H. Enomoto, M. Tsukamoto,
\emph{Intervals of $s$-torsion pairs in extriangulated categories 
with negative first extensions},
Math. Proc. Cambridge Philos. Soc., \textbf{174.3} (2023), 451--469.
\bibitem[Ai]{Ai} 
T. Aihara, 
\emph{Tilting-connected symmetric algebras}, 
Algebr. Represent. Theory, \textbf{16.3} (2013), 873--894.
\bibitem[AiI]{AI}
T. Aihara, O. Iyama, 
\emph{Silting mutation in triangulated categories}, 
J. Lond. Math. Soc. (2), \textbf{85.3} (2012), 633--668.
\bibitem[Al]{Al-Nofayee} 
S. Al-Nofayee, 
\emph{Simple objects in the heart of a $t$-structure}, 
J. Pure Appl. Algebra, \textbf{213.1} (2009), 54--59.
\bibitem[ALS]{ALS}
L. Angeleri-H\"{u}gel, R. Laking, F. Sentieri,
\emph{Torsion pairs via the Ziegler spectrum},
arXiv:2403.00475.
\bibitem[AHIKM]{AHIKM}
T. Aoki, A. Higashitani, O. Iyama, R. Kase, Y. Mizuno,
\emph{Fans and polytopes in tilting theory I: Foundations},
arXiv:2203.15213v4, to appear in Compos. Math.
\bibitem[As1]{A-semi} 
S. Asai, 
\emph{Semibricks}, 
Int. Math. Res. Not., \textbf{2020.16} (2020), 4993--5054.
\bibitem[As2]{A} 
S. Asai, 
\emph{The wall-chamber structures of the real Grothendieck groups}, 
Adv. Math. \textbf{381} (2021), Paper No. 107615.
\bibitem[As3]{A-smc}
S. Asai,
\emph{Mutation of simple-minded collections revisited},
in preparation.
\bibitem[AsI1]{AsI1} 
S. Asai, O. Iyama, 
\emph{Semistable torsion classes and canonical decompositions 
in Grothendieck groups},
Proc. Lond. Math. Soc. (3), \textbf{129.5} (2024), e12639.
\bibitem[AsI2]{AsI2} 
S. Asai, O. Iyama, 
\emph{$M$-TF equivalence in the real Grothendieck group}, 
arxiv:2404.13232v3.
\bibitem[AP]{AP} 
S. Asai, C. Pfeifer, 
\emph{Wide subcategories and lattices of torsion classes}, 
Algebr. Represent. Theory, \textbf{25} (2022), 1611--1629.
\bibitem[AS]{AS} 
M. Auslander, S. O. Smal\o, 
\emph{Almost split sequences in subcategories}, 
J. Algebra, \textbf{69.2} (1981), 426--454.
\bibitem[BCZ]{BCZ}
E. Barnard, A. Carroll, S. Zhu,
\emph{Minimal inclusions of torsion classes},
Algebr. Comb., \textbf{2.5} (2019), 879--901.
\bibitem[BDH]{BDH}
E. Barnard, C. Defant, E. Hanson,
\emph{Pop-Stack Operators for Torsion Classes and Cambrian Lattices},
arXiv:2312.03959v1.
\bibitem[BKT]{BKT}
P. Baumann, J. Kamnitzer, P. Tingley, 
\emph{Affine Mirkovi\'{c}-Vilonen polytopes},
Publ. Math. Inst. Hautes \'{E}tudes Sci., \textbf{120} (2014), 113--205.
\bibitem[BCDMTY]{BCDMTY}
V. Bazier-Matte, N. Chapelier-Laget, G. Douville, K. Mousavand, 
H. Thomas, E. Y{\i}ld{\i}r{\i}m,
\emph{ABHY Associahedra and Newton polytopes of -polynomials for cluster algebras of simply laced finite type},
J. Lond. Math. Soc. (2), \textbf{109.1} (2024), e12817. 
\bibitem[BBD]{BBD}
A. A. Be\u{\i}linson, J. Bernstein, P. Deligne, 
\emph{Faisceaux pervers},
Analysis and Topology on Singular Spaces, I, Luminy, 1981,
Ast\'{e}risque 100, Soci\'{e}t\'{e} math\'{e}matique de France, 1982, 5--171.
\bibitem[B]{Bridgeland}
T. Bridgeland,
\emph{Scattering diagrams, Hall algebras and stability conditions}, 
Algebr. Geom., \textbf{4.5} (2017), 523--561.
\bibitem[BCPW]{BCPW}
N. Broomhead, R. Coelho Sim\~{o}es, D. Pauksztello, J. Woolf,
\emph{Simple tilts of length hearts and simple-minded mutation},
arXiv:2401.02947v2.
\bibitem[BPPW]{BPPW}
N. Broomhead, D. Pauksztello, D. Ploog, J. Woolf,
\emph{The heart fan of an abelian category},
arXiv:2310.02844v2.
\bibitem[BST]{BST}
T. Br\"{u}stle, D. Smith, H. Treffinger, 
\emph{Wall and chamber structure for finite-dimensional algebras}, 
Adv. Math., \textbf{354} (2019), Paper No. 106746.
\bibitem[BY]{BY} 
T. Br\"ustle, D. Yang, \emph{Ordered exchange graphs}, 
Advances in representation theory of algebras, 135--193, 
EMS Ser. Congr. Rep., Eur. Math. Soc., Z\"urich, 2013.
\bibitem[CWZ]{CWZ}
P. Cao, Y. Wang, H. Zhang, 
\emph{Relative left Bongartz completions 
and their compatibility with mutations}, 
Math. Z., \textbf{305} (2023), article no. 27. 
\bibitem[CLS]{CLS} 
D. Cox, J. Little, H. Schenck, \emph{Toric varieties}, 
Graduate Studies in Mathematics, 124. American Mathematical Society, 
Providence, RI, 2011.
\bibitem[DIJ]{DIJ}
L. Demonet, O. Iyama, G. Jasso, 
\emph{$\tau$-Tilting Finite Algebras, Bricks, and g-Vectors}, 
Int. Math. Res. Not., \textbf{2019.3} (2019), 852--892.
\bibitem[DK]{DK} 
R. Dehy, B. Keller, 
\emph{On the Combinatorics of Rigid Objects in 2--Calabi--Yau Categories}, 
Int. Math. Res. Not., \textbf{2008} (2008), rnn029.
\bibitem[DF]{DF} 
H. Derksen, J. Fei, 
\emph{General presentations of algebras}, 
Adv. Math., \textbf{278} (2015), 210--237.
\bibitem[DIRRT]{DIRRT} 
L. Demonet, O. Iyama, N. Reading, I. Reiten, H. Thomas, 
\textit{Lattice theory of torsion classes: Beyond $\tau$-tilting theory}, 
Trans. Amer. Math. Soc. Ser. B, \textbf{10} (2023), 542--612.
\bibitem[ES]{ES}
H. Enomoto, A. Sakai,
\emph{ICE-closed subcategories and wide $\tau$-tilting modules},
Math. Z., \textbf{300} (2022), 541--577.
\bibitem[FZ]{FZ} 
S. Fomin, A. Zelevinsky, 
\emph{Cluster algebras. I. Foundations}, 
J. Amer. Math. Soc., \textbf{15.2} (2002), 497--529.
\bibitem[Ga]{Garcia}
M. Garcia,
\emph{On thick subcategories of the category of projective presentations},
Math. Z., \textbf{312} (2026), article no. 110.
\bibitem[Gn]{Gnedin}
W. Gnedin,
\emph{Silting theory under change of rings},
arXiv:2204.00608v1.
\bibitem[GZ]{GZ}
E. Gupta, Y. Zhou,
\emph{Semibricks and wide subcategories in extended module categories},
arXiv:2511.08157v1.
\bibitem[HI]{HI}
N. Hanihara, O. Iyama,
\emph{Silting correspondences and Calabi-Yau dg algebras},
arXiv:2508.12836v2.
\bibitem[HIKT]{HIKT}
E. Hanson, K. Igusa, M. Kim, G. Todorov, 
\emph{Infinitesimal semi-invariant pictures and co-amalgamation}, 
J. Lond. Math. Soc. (2), \textbf{109.1} (2024), e12786.
\bibitem[HR]{HR}
S. Hassoun, S. Roy, 
\emph{Admissible intersection and sum property}, 
arXiv:1906.03246v3.
\bibitem[IT]{IT}
C. Ingalls, H. Thomas, 
\emph{Noncrossing partitions and representations of quivers}, 
Compos. Math., \textbf{145.6} (2009), 1533--1562. 
\bibitem[IK]{IK}
O. Iyama, Y. Kimura, 
\emph{Classifying subcategories of modules over Noetherian algebras}, 
Adv. Math., \textbf{446} (2024), Paper No. 109631.
\bibitem[IY]{IY}
O. Iyama, D. Yang, 
\emph{Silting reduction and Calabi--Yau reduction of triangulated categories}, 
Trans. Amer. Math. Soc., \textbf{370} (2018), 7861--7898.
\bibitem[J]{Jasso} 
G. Jasso, 
\emph{Reduction of $\tau$-tilting modules and torsion pairs},
Int. Math. Res. Not., \textbf{2015.16} (2015), 7190--7237. 
\bibitem[KT]{KT}
M. Kaipel, H. Treffinger,
\emph{Wall-and-chamber structures for finite-dimensional algebras
and $\tau$-tilting theory},
Representations of Algebras and Related Topics,
Proceedings of the Workshop and the 20th International Conference 
on Representations of Algebras, 261--296.
\bibitem[KV]{KV} 
B. Keller, D. Vossieck, 
\emph{Aisles in derived categories}, 
Deuxi\`{e}me Contact Franco-Belge en Alg\`{e}bre (Faulx-les-Tombes, 1987), 
Bull. Soc. Math. Belg. S\'{e}r. A, \textbf{40.2} (1988), 239--253.
\bibitem[KeY]{KeY}
B. Keller, D. Yang,
\emph{Derived equivalences from mutations of quivers with potential},
Adv. Math., \textbf{226.3} (2011), 2118--2168.
\bibitem[K]{K} 
A. D. King, 
\emph{Moduli of representations of finite dimensional algebras}, 
Quart. J. Math. Oxford Ser. (2), \textbf{45.180} (1994), 515--530.
\bibitem[KoY]{KY} 
S. Koenig, D. Yang,
\emph{Silting objects, simple-minded collections, 
t-structures and co-t-structures for finite dimensional algebras},
Doc. Math., \textbf{19} (2014), 403--438.
\bibitem[M\v{S}]{MS}
F. Marks, J. \v{S}\v{t}ov\'{i}\v{c}ek,
\emph{Torsion classes, wide subcategories and localisations},
Bull. Lond. Math. Soc., \textbf{49.3} (2017), 405--416.
\bibitem[M]{Mizuno}
Y. Mizuno,
\emph{Shard theory of $g$-fans},
Int. Math. Res. Not., \textbf{2024.19} (2024), 13106--13126.
\bibitem[P]{P} 
P.-G. Plamondon, 
\emph{Generic bases for cluster algebras from the cluster category}, 
Int. Math. Res. Not., \textbf{2013.10} (2013), 2368--2420.
\bibitem[PYK]{PYK}
P.-G. Plamondon, T. Yurikusa, B. Keller,
\emph{Tame Algebras Have Dense $g$-Vector Fans},
Int. Math. Res. Not., \textbf{2023.4} (2023), 2701--2747.
\bibitem[Ri]{Rickard}
J. Rickard, 
\emph{Morita theory for derived categories},
J. London Math. Soc. (2), \textbf{39.3} (1989), 436--456.
\bibitem[Ru]{Rudakov}
A. Rudakov, 
\emph{Stability for an abelian category}, 
J. Algebra, \textbf{197.1} (1997), 231--245.
\bibitem[STTV]{STTV}
S. Schroll, A. Tattar, H. Treffinger, N. Williams,
\emph{A geometric perspective on the $\tau$-cluster morphism category},
Math. Z., \textbf{312} (2026), article no. 77.
\bibitem[T]{Tattar}
A. Tattar, 
\emph{Torsion Pairs and Quasi-abelian Categories},
Algebr. Represent. Theory, \textbf{24} (2021), 1557--1581.
\bibitem[Y]{Y} 
T. Yurikusa, 
\emph{Wide subcategories are semistable}, 
Doc. Math., \textbf{23} (2018), 35--47.
\end{thebibliography}
\end{document}